\documentclass[10pt,a4paper]{article}
\usepackage{bbm}

\usepackage[leqno]{amsmath}
\usepackage{amsfonts}
\usepackage{graphicx}
\usepackage{amsmath}
\usepackage{amssymb}
\usepackage{latexsym}
\usepackage{amsmath, amsfonts,amssymb, amsthm, euscript,makeidx,color,mathrsfs}
\usepackage{enumerate}

\usepackage[colorlinks,linkcolor=blue,anchorcolor=green,citecolor=red]{hyperref}

\def\sqr#1#2{{\vcenter{\vbox{\hrule height.#2pt
              \hbox{\vrule width.#2pt height#1pt \kern#1pt \vrule width.#2pt}
              \hrule height.#2pt}}}}
\def\signed #1{{\unskip\nobreak\hfil\penalty50
              \hskip2em\hbox{}\nobreak\hfil#1
              \parfillskip=0pt \finalhyphendemerits=0 \par}}
\def\endpf{\signed {$\sqr69$}}

\def\5n{\negthinspace \negthinspace \negthinspace \negthinspace \negthinspace }
\def\4n{\negthinspace \negthinspace \negthinspace \negthinspace }
\def\3n{\negthinspace \negthinspace \negthinspace }
\def\2n{\negthinspace \negthinspace }
\def\1n{\negthinspace }

\def\dbE{\mathbb{E}}
\def\dbF{\mathbb{F}}

\def\dbH{\mathbb{H}}

\def\dbP{\mathbb{P}}
\def\dbQ{\mathbb{Q}}
\def\dbR{\mathbb{R}}

\def\sB{\mathscr{B}}

\def\sE{\mathscr{E}}

\def\sT{\mathscr{T}}

\def\={\buildrel \triangle \over =}

\def\ds{\displaystyle}

\def\ns{\noalign{\ss}}
\def\a{\alpha}
\def\b{\beta}
\def\g{\gamma}
\def\d{\delta}
\def\e{\varepsilon}
\def\z{\zeta}
\def\k{\kappa}
\def\l{\lambda}
\def\m{\mu}
\def\n{\nu}
\def\si{\sigma}
\def\t{\tau}
\def\f{\varphi}
\def\th{\theta}
\def\o{\omega}

\def\i{\infty}
\def\G{\Gamma}
\def\D{\Delta}

\def\F{\Phi}
\def\O{\Omega}
\def\cB{{\cal B}}

\def\cE{{\cal E}}
\def\cF{{\cal F}}
\def\cG{{\cal G}}

\def\cM{{\cal M}}

\def\cS{{\cal S}}

\def\cY{{\cal Y}}

\def\cl{{\cal l}}
\def\ss{\smallskip}
\def\ms{\medskip}

\def\q{\quad}
\def\qq{\qquad}
\def\hb{\hbox}

\def\limsup{\mathop{\overline{\rm lim}}}
\def\liminf{\mathop{\underline{\rm lim}}}

\def\Ra{\mathop{\Rightarrow}}

\def\lan{\langle}
\def\ran{\rangle}

\def\rf{\eqref}

\def\esssup{\mathop{\rm esssup}}
\def\essinf{\mathop{\rm essinf}}

\def\h{\widehat}
\def\wt{\widetilde}

\def\cd{\cdot}
\def\cds{\cdots}

\def\ae{\hbox{\rm a.e.}}
\def\as{\hbox{\rm a.s.}}
\def\sgn{\hbox{\rm sgn$\,$}}

\def\cl{\overline}

\def\deq{\triangleq}
\def\les{\leqslant}
\def\ges{\geqslant}

\def\({\Big (}
\def\){\Big )}
\def\[{\Big[}
\def\]{\Big]}

\def\bde{\begin{definition}\label}
\def\ede{\end{definition}}
\def\be{\begin{equation}}
\def\bel{\begin{equation}\label}
\def\ee{\end{equation}}
\def\bt{\begin{theorem}\label}
\def\et{\end{theorem}}
\def\bc{\begin{corollary}\label}
\def\ec{\end{corollary}}
\def\bl{\begin{lemma}\label}
\def\el{\end{lemma}}
\def\bp{\begin{proposition}\label}
\def\ep{\end{proposition}}
\def\bas{\begin{assumption}\label}
\def\eas{\end{assumption}}
\def\br{\begin{remark}\label}
\def\er{\end{remark}}
\def\bex{\begin{example}\label}
\def\ex{\end{example}}
\def\ba{\begin{array}}
\def\ea{\end{array}}
\def\ed{\end{document}}

\def\square#1{\vbox{\hrule\hbox{\vrule height#1%
     \kern#1\vrule}\hrule}}
\def\rectangle#1#2{\vbox{\hrule\hbox{\vrule height#1%
     \kern#2\vrule}\hrule}}

\def\T{[0,T]}

\font\tenbb=msbm10 \font\sevenbb=msbm7 \font\fivebb=msbm5

\newfam\bbfam
\scriptscriptfont\bbfam=\fivebb \textfont\bbfam=\tenbb
\scriptfont\bbfam=\sevenbb

\newtheorem{theorem}{\hskip 1.3em Theorem}[section]
\newtheorem{definition}[theorem]{\hskip 1.3em Definition}
\newtheorem{proposition}[theorem]{\hskip 1.3em Proposition}
\newtheorem{corollary}[theorem]{\hskip 1.3em Corollary}
\newtheorem{lemma}[theorem]{\hskip 1.3em Lemma}
\newtheorem{remark}[theorem]{\hskip 1.3em Remark}
\newtheorem{example}[theorem]{\hskip 1.3em Example}

\makeatletter
   
   \@addtoreset{equation}{section}
\makeatother

\begin{document}

\title{Multi-Dimensional Super-Linear Backward Stochastic Volterra Integral Equations}

\author{
Shengjun Fan\footnote{School of Mathematics, China University of Mining and Technology, Xuzhou, China. Email: {\tt shengjunfan}
{\tt @cumt.edu.cn}. This author is supported by National Natural Science Foundation of China (No. 12171471).}  ~~~ Tianxiao Wang \footnote{School of Mathematics, Sichuan University, Chengdu, China. Email: {\tt wtxiao2014@scu.edu.cn}. This author is supported
by the NSF of China under grants 11971332 and 11931011, the Science Development Project of Sichuan University
under grant 2020SCUNL201.} ~~ and ~~~Jiongmin Yong \footnote{Department of Mathematics, University of Central Florida, Orlando, USA. Email: {\tt jiongmin.yong@ucf.edu}. This author is supported in part by NSF Grant DMS-1812921.}
}

\maketitle

\begin{abstract}
In this paper, a systematic investigation is carried out for the general solvability of multi-dimensional backward stochastic Volterra integral equations (BSVIEs) with the generators being super-linear in the adjustment variable $Z$. Two major situations are discussed: (i) When the free term is bounded with the dependence of the generator on $Z$ being of ``diagonally strictly'' quadratic growth and being sub-quadratically coupled with off-diagonal components; (ii) When the free term is unbounded having exponential moments of arbitrary order with the dependence of the generator on $Z$ being diagonally no more than quadratic and being independent of off-diagonal components. Besides, for the case that the generator is super-quadratic in $Z$, some negative results are presented.

\end{abstract}

\ms


\ms

\bf AMS 2020 Mathematical sciences classifications. \rm
60H20, 45D05.

\section{Introduction}

Let $(\O,\cF,\dbF,\dbP)$ be a filtered complete probability space on which a standard one-dimensional Brownian motion $W(\cd)$ is defined with $\dbF=\{\cF_t\}_{t\ges0}$ being its natural filtration augmented by all the $\dbP$-null sets. Consider the following {\it backward stochastic Volterra integral equation} (BSVIE, for short):
\bel{BSVIE1}Y(t)=\psi(t)+\int_t^Tg(t,s,Y(s),Z(t,s))ds-\int_t^TZ(t,s)dW(s),
\ \ t\in[0,T],\ee
where $g:\D[0,T]\times\dbR^m\times\dbR^m\times\O\to\dbR^m$, satisfying proper conditions, is called the {\it generator} of BSVIE \rf{BSVIE1}, with $\D^*[0,T]=\{(t,s)\in[0,T]^2\bigm|0\les t<s\les T\}$ being the upper triangle domain within the square $[0,T]^2$, and $\psi(\cd)$ is called a {\it free term}. The unknown is a pair $(Y(\cd),Z(\cd\,,\cd))$ of processes having certain properties, called an {\it adapted solution}. For convenience, we refer to $Y(\cd)$ as the {\it backward term} and refer to $Z(\cd\,,\cd)$ as the {\it adjustment term} (since such a term makes it possible to have adaptiveness of $Y(\cd)$). When $g(t,s,y,z)\equiv g(s,y,z)$ and $\psi(t)\equiv\xi$ are independent of $t$, \rf{BSVIE1} is reduced to the following {\it backward stochastic differential equation} (BSDE, for short) (of the integral form):
\bel{BSDE1}Y(t)=\xi+\int_t^Tg(s,Y(s),Z(s))ds-\int_t^TZ(s)dW(s),\ \  t\in[0,T].\ee
Linear BSDE was introduced by Bismut in 1973 (\cite{Bismut 1973}) and the general nonlinear version was introduced by Pardoux--Peng in 1990 (\cite{Pardoux-Peng 1990}) where the well-posedness was established under uniform Lipschitz condition of the map $(y,z)\mapsto g(s,y,z)$. Among a huge number of relevant papers, we mention El Karoui--Peng--Quenez \cite{El Karoui-Peng-Quenez 1997} and the books by Ma--Yong \cite{Ma-Yong 1999}, Yong--Zhou \cite{Yong-Zhou 1999}, and Zhang \cite{Zhang 2017} (and the references cite therein). In 2000, Kobylanski studied one-dimensional BSDEs having a quadratic growth of the map $z\mapsto g(s,y,z)$ (QBSDE, for short) with a bounded terminal term (\cite{Kobylanski 2000}). This work was extended by Briand--Hu \cite{Briand-Hu 2006, Briand-Hu 2008} allowing the terminal term to be unbounded (having exponential moments). Further, Delbaen--Hu--Bao (\cite{Delbaen-Hu-Bao 2011}) and Fan--Hu (\cite{Fan-Hu 2021}) respectively investigated the (one-dimensional) cases with $z\mapsto g(s,y,z)$ being super-quadratic and sub-quadratic. For multi-dimensional situations, some extensions for QBSDEs were carried out by Hu--Tang (\cite{Hu-Tang 2016}), and Fan--Hu--Tang (\cite{Fan-Hu-Tang 2020}) for either $z\mapsto g(s,y,z)$ being diagonal, or so-called {\it diagonally strictly quadratic} (see below for a precise definition).

\ms

Study of BSVIEs was begun with the work of Lin in 2002 (\cite{Lin 2002}). For general theory of BSVIEs, the readers are referred to Yong \cite{Yong 2006, Yong 2008}, Wang--Zhang \cite{Wang-Zhang 2007}, Shi--Wang \cite{Shi-Wang 2012}, Djordjevi\'c--Jankovi\'c \cite{Djordjevic-Jankovic 2015}, Hu--Oksendal \cite{Hu-Oksendal 2019}, Popier \cite{Popier 2021} for the well-posedness of BSVIEs, to Wang--Yong \cite{Wang-Yong 2015} for comparison theorems, to Shi--Wang--Yong \cite{Shi-Wang-Yong 2013} for some extensions to mean-field BSVIEs, to Shi--Wen--Xiong \cite{Shi-Wen-Xiong 2019} for doubly BSVIEs, to Hamaguchi \cite{Hamaguchi 2021} for infinite horizonal BSVIEs, to Wang \cite{Wang 2021} for extended BSVIEs, to Wang--Yong \cite{Wang-Yong 2021}, Hern\'{a}ndez \cite{Hernandez 2021}, Hern\'{a}ndez--Possamai \cite{Hernandez-Possamai 2021} for BSVIEs with $Z(s,s)$ appears, to Agram--Djehiche \cite{Agram-Djehiche-2021-SCL} for reflected BSVIEs, to Yong \cite{Yong 2017}, Wang--Yong \cite{Wang-Yong 2019}, Wang \cite{Wang 2021}, Wang--Yong--Zhang \cite{Wang-Yong-Zhang-AIHP} for connections with PDEs, to Yong \cite{Yong 2006, Yong 2008}, Wang--Zhang \cite{Wang-Zhang 2017}, Wang \cite{Wang 2018} for applications in stochastic optimal control theory, to Yong \cite{Yong 2007}, Wang--Shi \cite{Wang-Shi 2013}, Agram \cite{Agram-2019-SAA}, Beissner--Rosazza Gianin \cite{Bessner-Rosazza Ganin-2021}, Wang--Sun--Yong \cite{Wang-Sun-Yong 2021} for connections with mathematical finance, and to Wang--Yong \cite{Wang-Yong 2021}, Hamaguchi \cite{Hamaguchi 2021b} for applications to time-inconsistent stochastic optimal controls. It is well-known by now that the well-posedness can be established under the uniform Lipschitz conditions of $(y,z)\mapsto g(t,s,y,z)$ for the general multi-dimensional cases. For the case of $z\mapsto g(t,s,y,z)$ being quadratic, the one-dimensional well-posedness was established by Wang--Sun--Yong \cite{Wang-Sun-Yong 2021} with bounded free terms, and the multi-dimensional case was studied in Hern\'{a}ndez \cite{Hernandez 2021} with essentially small bounded free terms. To our best knowledge, it is completely blank for the general multi-dimensional super-linear BSVIEs (including sub-quadratic, strictly quadratic, and super-quadratic growth of $z\mapsto g(t,s,y,z)$) with bounded/unbounded free terms. From mathematical point of view, it will be good to fill these blanks for BSVIEs, and to explore further, even for BSDEs. This is a major motivation of the current paper. Our main results can be briefly summarized as follows. Let us emphasis that all results are for multi-dimensional BSVIEs. Also, although we only consider the case of one-dimensional Brwonian motion, all the results hold for multi-dimensional Brownian motions, just with some modification of notations.

\ms

(i) For the case $z\mapsto g(t,s,y,z)$ being diagonally strictly quadratic, well-posedness of BSVIEs will be established under the condition that the free term is bounded, allowing off-diagonal components of $z$ appear sub-quadratically (see a later section for the precise meaning). Moreover, our results allow the generator to have a mixture of quadratic and linear growth components. Note that in the literature, the results for QBSDEs do not cover those for linear growth BSDEs (\cite{Hu-Tang 2016,Fan-Hu-Tang 2020}). Hence, this part of the results are new even for BSDEs.

\ms

(ii) For the case $z\mapsto g(t,s,y,z)$ being super-linear but sub-quadratic, we will establish the well-posedness with more general free terms than the quadratic BSVIEs. Some componentwise extended convexity/concavity of $z\mapsto g(t,s,y,z)$ will be assumed. This part of results are also new for BSDEs because only one-dimensional sub-quadratic BSDEs were studied in the literature (\cite{Fan-Hu 2021}).

\ms

(iii) For the case $z\mapsto g(t,s,y,z)$ being super-quadratic,
we will establish some negative results on the well-posedness, either no existence, or no uniqueness. In some sense, these results are extensions of similar ones for BSDEs (\cite{Delbaen-Hu-Bao 2011}).

\ms

Now, let us present another important motivation of the current paper.

\ms

A given random variable $\xi\in L^1_{\cF_T}(\O;\dbR)$ can represent the return of some asset position at some future time $T$. A standard {\it dynamic entropic risk measure} of $\xi$ with a {\it risk aversion parameter} $\g>0$ is defined by
\bel{rho^1}\rho^1_t(\xi)={1\over\g}\log\dbE_t\big[e^{-\g\xi}\big],\ \  t\in[0,T],\ee
where $\dbE_t[\,\cd\,]=\dbE[\,\cd\,|\,\cF_t]$ (\cite{El Karoui-Peng-Quenez 1997}). It is known (\cite{Follmer-Penner 2006,Follmer-Knispel 2011}) that the above is a convex risk measure (not coherent), on the set
$$\sE_\g^1=\Big\{\xi\in L^1_{\cF_T}(\O;\dbR)\bigm|\sup_{t\in[0,T]}
\dbE_t\big[e^{-\g\xi}\big]<\infty\Big\}.$$
Clearly, $L^\infty_{\cF_T}(\O;\dbR)$ is a proper subset of $\sE_\g^1$, as the latter also contains some unbounded random variables, besides bounded ones. Let us consider the following BSDE (of integral form)
\bel{BSDE2}Y(t)=-\xi+\int_t^T{\g\over2}|Z(s)|^2ds-\int_t^TZ(s)dW(s),\ \  t\in[0,T].\ee
Suppose $(Y(\cd),Z(\cd))$ is the adapted solution of this BSDE.
%
%
%
%
%
%
%
%
%
Then, a straightforward calculation shows that
$$Y(t)={1\over\g}\log\dbE_t\big[e^{-\g\xi}\big]=\rho_t^1(\xi).$$
This means that the backward term $Y(\cd)$ of the adapted solution $(Y(\cd),Z(\cd))$ to the BSDE \rf{BSDE2} can be exactly taken as a dynamic entropic risk measure for $\xi$, as long as $\xi\in\sE^1_\g$.

\ms

Next, we want to extend the above dynamic entropic risk measure to a larger set than $\sE_\g^1$. To this end, for any $\a>0$, we define
$$\xi^{\{\a\}}=(\xi^+)^\a-(\xi^-)^\a\equiv\left\{\ba{ll}
\ss\ds\xi^\a,\qq\xi\ges0,\\
\ss\ds-|\xi|^\a,\ \ \xi<0.\ea\right.$$
This is the odd extension of $\xi^\a$, $\xi\ges0$ to $\dbR$, symmetric with respect to the origin, which is well-defined on $\dbR$ and is strictly increasing. Now, for $\a,\g>0$, we introduce
$$\sE_\g^\a=\Big\{\xi\in L^1_{\cF_T}(\O;\dbR)\bigm|\sup_{t\in[0,T]}
\dbE_t\big[e^{-\g\xi^{\{\a\}}}\big]<\infty\Big\}.$$
Since
$$\dbE_t\big[e^{-\g\xi^{\{\a\}}}\big]
=\dbE_t\big[e^{-\g\xi^\a}I_{(\xi\ges0)}\big]+\dbE_t
\big[e^{\g|\xi|^\a}I_{(\xi<0)}\big],\q 0\les\dbE_t\big[e^{-\g\xi^\a}I_{(\xi\ges0)}\big]\les1,
$$
we see that
$$\sE^\a_\g=\Big\{\xi\in L^1_{\cF_T}(\O;\dbR)\bigm|\sup_{t\in[0,T]}\dbE_t\big[e^{\g|\xi|^\a}
I_{(\xi<0)}\big]<\infty\Big\}.$$
Clearly, both $\a\mapsto\sE^\a_\g$ and $\g\mapsto\sE^\a_\g$ are (strictly) decreasing, in particular,
$$\sE^1_\g\varsubsetneq\sE^\a_\g,\ \ \a\in(0,1).$$
Now, for the return $\xi\in\sE^\a_\g\setminus\sE^1_\g$, inspired by \rf{rho^1}, we may try to define an extended dynamic entropic risk measure of such a $\xi$ by the following:
\bel{rho-a}\rho_t^\a(\xi)=\({1\over\g}\log\dbE_t\big[e^{-\g
\xi^{\{\a\}}}\big]\)^{\{{1\over\a}\}},\ \  t\in[0,T].\ee
However, the above only satisfies the monotonicity, and it does not even have the translation invariance. Hence, this is not a good dynamic risk measure. To have a reasonable entropic type dynamic risk measure for $\xi\in\sE_\g^\a$, suggested by the case of $\sE_\g^1$, we consider the following BSDE:
\bel{BSDE2*}Y(t)=-\xi+\int_t^T{\g\over \k}|Z(s)|^{2\over2-\a}ds+\int_t^TZ(s)
dW(s),\ \  t\in[0,T],\ee
for some $\k>1$, whose generator has a sub-quadratic growth in $Z$ (since ${2\over2-\a}\in(1,2)$). According to Fan--Hu \cite{Fan-Hu 2021}, for any $\xi\in\sE^\a_\g$, the above BSDE admits a unique adapted solution $(Y(\cd),Z(\cd))$ satisfying
\bel{estimate} e^{\frac{\g}{\k}|Y(t)|^\a}I_{(Y(t)>0)}\les K\dbE_t\Big[e^{\g|\xi|^\a}I_{(\xi<0)}\Big]\ee
for some $K>1$. In the current case, the map $\xi\mapsto Y(t)\equiv Y(t;\xi)$ satisfies the monotonicity, translation invariance, and convex. Hence, we may take such a map as an extended entropic dynamic risk measure on the set $\sE_\g^\a$.
This leads to the study of BSDEs with the generators being sub-quadratic in $Z$.
\ms

Next, according to Yong \cite{Yong 2007}, Wang--Sun--Yong \cite{Wang-Sun-Yong 2021}, in the one-dimensional case, if instead of the return $\xi$ at the terminal time $T$, one has an $\cF_T$-measurable position process $\psi(\cd)$ (not necessarily $\dbF$-adapted), to dynamically measure the risk, it will be more proper to consider the following BSVIE:
\bel{BSVIE00}Y(t)=-\psi(t)+\int_t^Tg(t,s,Z(t,s))ds-\int_t^TZ(t,s)dW(s),
\ \  t\in[0,T],\ee
for a proper generator $g$, where $z\mapsto g(t,s,z)$ could be of linear, super-linear (which includes sub-quadratic, quadratic, or even super-quadratic). If $(Y(\cd),Z(\cd\,,\cd))$ is an adapted solution of BSVIE \rf{BSVIE00}, we can call $Y(\cd)$ an {\it equilibrium dynamical (entropic) risk measure} of $\psi(\cd)$. In the case that $\psi(\cd)$ is bounded, and $z\mapsto g(t,s,z)$ is quadratic, (the one-dimensional BSVIE) \rf{BSVIE00} admits a unique adapted solution $(Y(\cd),Z(\cd\,,\cd))$ so that we have an {\it equilibrium entropic dynamical risk measure} for $\psi(\cd)$ (\cite{Wang-Sun-Yong 2021}).

\ms

It is common that there are more than one position processes that have to be measured their dynamic risks individually. For example, for a (commercial and/or investment) bank, there might be some loan positions which have quite different features. They are closely related (coupled) and the dynamical risks should be measured separately. Therefore, we naturally have the vector-valued (equilibrium) dynamical risk measure process which satisfies a multi-dimensional BSDEs or BSVIEs. This consideration gives us a motivation of studying multi-dimensional BSVIEs with the generator $g$ having various types of growth in $z$. From this, people might see the application potential contained in the main results of the current paper.

\ms

The rest of the paper is organized as follows. Some preliminary is presented in Section 2, including the revisit of relevant BSDE results and some crucial lemmas to be used later. Section 3 is devoted to the case of bounded free terms for diagonally strictly quadratic case and the mixed linear and quadratic case. In Section 4, we discuss BSVIEs with unbounded free for the generators having no more than quadratic growth. Some negative results are presented in Section 5 when the generator is super-quadratic in $Z$. Finally, some concluding remarks are collected in Section 6.

\section{Preliminaries}

\subsection{Some spaces}

Let us first introduce some basic spaces. For each $t\in [0,T]$ and any Euclidean space $\dbH$ with the norm $|\cd|$ which could be $\dbR^n,\dbR^{n\times d}$, etc., we define the following basic space:
$$L^p_{\cF_t}(\O;\dbH)=\Big\{\xi:\O\to \dbH  \bigm|\xi\hb{ is $\cF_t$-measurable, }\|\xi\|_p\equiv\big(\dbE|\xi|^p\big)^{1\over p}<\infty\Big\},\q p\in[1,\infty).$$
In an obvious way, we can define $L^\infty_{\cF_t}(\O;\dbH)$. Clearly, $L^p_{\cF_t}(\O;\dbH)$ is a Banach space under the norm $\|\,\cd\,\|_p$, for all $p\in[1,\infty]$. When the range space $\dbH$ is clear from the context and is not necessarily to be emphasized, we will omit $\dbH$. In particular, we will denote $L^p_{\cF_T}(\O)=L^p_{\cF_T}(\O;\dbH)$.

\ms

Next, we introduce spaces of stochastic processes. In order to avoid repetition, all processes $(t,\o)\mapsto\f(t,\o)$ are assumed to be at least $\cB[0,T]\otimes\cF_T$-measurable without further mentioning, where $\cB[0,T]$ is the Borel $\si$-field of $[0,T]$. For $p,q\in[1,\infty)$,
$$\ba{ll}
\ds L^p_{\cF_T}(\O;L^q(0,T;\dbH))\1n=\1n\Big\{\f:[0,T]\1n\times\1n\O\to\dbH\bigm|
\dbE\(\int_0^T\3n|\f(t)|^qdt\)^{p\over q}\3n<\infty\Big\},\\
\ns\ds L^p_{\cF_T}(\O;L^\infty(0,\1n T;\dbH))\1n=\1n\Big\{\f:[0,T]\1n\times\1n\O\to\dbH\bigm|
\dbE\(\esssup_{t\in[0,T]}|\f(t)|^p\)\1n<\1n\infty\Big\}\\
\ns\ds L^p_{\cF_T}(\O;C([0,\1n T];\dbH))\1n=\1n\Big\{\f:[0,T]\1n\times\1n\O\to\dbH\bigm|
t\mapsto\f(t,\o)\hb{ is continuous, }\dbE\(\sup_{t\in[0,T]}|\f(t)|^p\)\1n<\1n\infty\Big\},\\
\ds L^q_{\cF_T}(0,T;L^p(\O;\dbH))\1n=\1n\Big\{\f:[0,T]\1n\times\1n\O\to\dbH
\bigm|
\int_0^T\(\dbE|\f(t)|^p\)^{q\over p}dt<\infty\Big\},\\
\ns\ds L^\infty_{\cF_T}(0,T;L^p(\O;\dbH))\1n=\1n\Big\{\f:[0,T]\1n\times\1n\O\to
\dbH\bigm|
\esssup_{t\in[0,T]}\(\dbE|\f(t)|^p\)^{1\over p}
<\1n\infty\Big\},\\
\ds C_{\cF_T}([0,T];L^p(\O;\dbH))\1n=\1n\Big\{\f:[0,T]\1n\times\1n\O\to\dbH
\bigm|t\mapsto\f(t, \omega )\hb{ is continuous, }\sup_{t\in[0,T]}\(\dbE|\f(t)|^p\)^{1\over p}<\1n\infty\Big\}.\ea$$
The spaces $L^\infty_{\cF_T}(\O;L^\infty(0,T;\dbH))$, $L^\infty_{\cF_T}(0,T;L^\infty(\O;\dbH))$, $L^\infty_{\cF_T}(\O;C([0,T];\dbH))$, and  $C_{\cF_T}([0,T];$ $L^\infty(\O;\dbH))$ can be defined in obvious ways. For all $p\in[1,\infty)$ , we denote
\bel{Lp}L^p_{\cF_T}(0,T;\dbH)\equiv L^p_{\cF_T}(0,T;L^p(\O;\dbH))=L^p_{\cF_T}(\O;L^p(0,T;\dbH))
=L^p((0,T)\times\O;\dbH).\ee
We point out that when $p=\infty$, the last two equalities in the above might not be true without certain type of measurability (see Example 5.0.10 in \cite{Fattorini 1999} and Example 1.42 in \cite{Roubicek 2005}). Therefore, we do not mention the above with $p=\infty$ to avoid some technicalities. Note that processes in the above spaces are not necessarily $\dbF$-adapted. The subset of, say, $L^p_{\cF_T}(\O;L^q(0,T;\dbH))$, consisting of all $\dbF$-progressively measurable processes, is denoted by $L^p_\dbF(\O;$ $L^q(0,T;\dbH))$ (having subscript $\dbF$ instead of $\cF_T$). All the $\dbF$-progressively measurable version of the other spaces listed in the above can be denoted similarly; for example, $C_\dbF([0,T];L^p(\O;\dbH))$, and so on. Finally, in the above, $[0,T]$ can be replaced by any $[t,T]$, and again, when the range space $\dbH$ is clear from the context, we will omit $\dbH$; for example, $L^p_\dbF(\O;C[t,T])$, and so on.

\ms

For each $t\in[0,T]$, $\m,\th>0$, $q\ges1$, we also introduce the following spaces of processes:
$$\ba{ll}
\ss\ds\cE_{\cF_t}^{\m,\th}(\O;\dbR^n)=\Big\{\xi\in L^1_{\cF_t}(\O;\dbR^n)\bigm|\dbE\big[\exp\big(\m|\xi|^\th\big)\big]<\infty\Big\},\\
\ss\ds\cE^{\m,\th}_{\cF_T}(\O;L^q(t,T;\dbR^n))=\Big\{\f(\cd)\in L^{ q }_{\cF_T}(t,T;\dbR^n)\bigm| \f(\cd)\hb{ is $\cB[t,T]\otimes\cF_T$-measurable},\\
\ss\ds\qq\qq\qq\qq\qq\qq\qq\qq\qq\qq\qq\(\int_t^T|\f(s)|^qds\)^{1\over q}\in\cE^{\m,\th}_{\cF_T}(\O;\dbR)\Big\},\\
\ss\ds\cE_\dbF^{\m,\th}(\O;L^q(t,T;\dbR^n))=\Big\{\f(\cd)\in\cE^{\m,\th}_{\cF_T}(\O;
L^q(t,T;\dbR^n))
\bigm|\f(\cd)\hb{ is $\dbF$-progressively measurable}\Big\}.\ea$$
The spaces $\cE_{\cF_T}^{\m,\th}(\O;L^\i(t,T;\dbR^n))$ and $\cE_\dbF^{\m,\th}(\O;L^\i(t,T;\dbR^n))$ can be defined in an obvious way. Further,
$$\ba{ll}
\ss\ds\cE^{\m,\th}_{\cF_T}(\O;C([t,T];\dbR^n))=\Big\{\f(\cd)\in\cE^{\m,\th}_{\cF_T}
(\O;L^\infty(t,T;\dbR^n))\bigm|
\f(\cd)\hb{ is continuous on $[t,T]$}\Big\},\\
\ss\ds\cE_\dbF^{\m,\th}(\O;C([t,T];\dbR^n))=\Big\{\f(\cd)\in\cE^{\m,\th}_{\cF_T}(\O;
C([t,T];\dbR^n))
\bigm|\f(\cd)\hb{ is $\dbF$-progressively measurable}\Big\}.\ea$$
Note that for each $t\in [0,T]$,
\bel{inclusion}L^\infty_{\cF_t}(\O;\dbR^n)\varsubsetneq\cE^{\m,\th}_{\cF_t}(\O;\dbR^n)\varsubsetneq\bigcap_{p\ges1}
L^p_{\cF_t}(\O;\dbR^n),\q \forall\m,\th>0,\ee
and $\cE^{\m,\th}_{\cF_t}(\O;\dbR^n)$ is decreasing in $\m$ and $\th$, respectively. Similar inclusions and monotonicity also hold for the other spaces defined above with parameters $\m$ and $\th$.

\ms

Next, for each $t\in[0,T)$, denote
$$\D^*[t,T]=\big\{(r,s)\in[t,T]^2\bigm|0\les t\les r<s\les T\big\}.$$
For $p,q\in [1,+\infty)$, we introduce the following space for $Z(\cd\,,\cd)$:
$$\ba{ll}
\ss\ds L^p_\dbF(\O;L^{q}(\D^*[t,T];\dbR^n))=\Big\{\z:\D^*[t,T]\times\O\to\dbR^n\bigm|\forall r\in[t,T],\, \z(r,\cd)\in L^p_\dbF(\O;L^{ q}(r,T;\dbR^n)),\\
\ss\ds\qq\qq\qq\qq\qq\qq\qq\qq\qq\qq\qq\qq\dbE\int_t^T\(\int_r^T|\z(r,s)|^{  q}ds\)^{p\over q}dr<\infty\Big\}.\ea$$
Also, we let
$$\ba{ll}
\ss\ds L^\infty(0,T;L^p_\dbF(\O;L^q(\cd\,,T;\dbR^n)))=\Big\{\f:\D^*[0,T]\times\O\to \dbR^n
\bigm| \forall t\in [0,T],\, \f(t,\cd)\in L^p_\dbF(\O;L^q(t,T;\dbR^n)), \\
\ss\ds\qq\qq\qq\qq\qq\qq\qq\qq\qq\qq\qq\qq\qq\esssup_{t\in[0,T]}\dbE\(\int_t^T|\f(t,s)|^qds\)^{p\over q}<\infty\Big\}.\ea$$
The case that $p$ and/or $q$ is equal to $\infty$ can be defined in an obvious way. Further, for each $t\in[0,T]$, denote by $\sT[t,T]$ the set of all $\dbF$-stopping times $\t$ valued in $[t,T]$. A uniformly integrable $\dbF$-martingale $M=\{M(t): 0\les t\les T\}$ with $M_0=0$ is called a {\it BMO martingale} on $[0,T]$ if
$$\|M(\cdot)\|^2_{{\rm BMO}[0,T]}\deq\sup_{\t\in\sT[0,T]} \Big\|\dbE_\t\big[|M(T)-M(\t)|^2\big]\Big\|_{\infty}
<\infty.$$
It is known (see \cite{Kazamaki 1994}) that for any given $K>0$, there are constants $c_1,c_2>0$ depending only on $K$ such that for any BMO martingale $N(\cd)$ with
$$\big\|N(\cd)\big\|_{{\rm BMO}_{\dbP}[0,T]}\les K,$$
by setting
$$\ba{ll}
\ns\ds d\wt\dbP :=\exp\Big\{\int_0^T N(t)dW(t)-\frac 1 2 \int_0^T |N(t)|^2dt\Big\}d\dbP,\\
\ns\ds\wt M(\cd):=M(\cd)-\lan M,N\ran(\cd),\ea$$
one has that $\wt M(\cd)$ is a BMO martingale under $\wt\dbP$, and
\bel{Equivalent-norms-BMO}c_1\big\|M(\cd)\big\|_{{\rm BMO}_\dbP[0,T]}\les\big\|\wt M(\cd)\big\|_{{\rm BMO}_{\wt\dbP}[0,T]}\les
c_2\big\|M(\cd)\big\|_{{\rm BMO}_{\dbP}[0,T]}.\ee
This means that the BMO norm of the BMO martingale $M(\cd)$ under $\dbP$ is equivalent to that of BMO martingale $\wt M(\cd)$ under $\wt\dbP$.  We now introduce
$$\ba{ll}
\ss\ds\overline{{\rm BMO}}([t,T];\dbR^n)\deq\Big\{\f:[t,T]\times\O\to\dbR^n\bigm| \f(\cd)\hb{ is $\dbF$--progressively measurable and}\\
\ss\ds\qq\qq\qq\qq\qq\qq\qq\qq\qq\big\|\f(\cd)\big\|^2_{\overline{{\rm BMO}}[t,T]}=\sup_{\t\in\sT[t,T]}\Big\|\dbE_\t\Big[\int_\t^T |\f(s)|^2ds\Big]\Big\|_\infty<\infty\Big\},\ea$$
and
$$\ba{ll}
\ss\ds\overline{{\rm BMO}}(\D^*[0,T];\dbR^n)\deq\Big\{\f:\D^*[0,T]\times\O\to\dbR^n\bigm| \f(t,\cd)\hb{ is $\dbF$--progressively measurable}\\
\ss\ds\qq\qq\qq\qq\qq\qq\qq\qq\qq\qq\qq\hb{on $[t,T]$, $\as~t\in[0,T]$, and}\\
\ss\ds\qq\qq\qq\qq\qq\qq\qq\qq\qq\qq\qq\big\|\f(\cd\,,\cd)\big\|^2_{\overline{{\rm BMO}}(\D^*[0,T])}\2n=\esssup_{t\in[0,T]}\big\|\varphi(t,\cd)\big\|^2_{\overline{{\rm BMO}}[t,T]}\1n<\1n\infty\Big\}.\ea$$
Clearly,
\bel{BMO1}\cl{\rm BMO}([t,T];\dbR^n)\supseteq L^\infty_\dbF(\O;L^2(t,T;\dbR^n)).\ee
On the other hand, for good enough process $\f(\cd\,,\cd)$, we have
\bel{}\ba{ll}
\ns\ds\|\f\|_{L^\i(0,T;L^2_\dbF(\O;L^2(\cd\,,T;\dbR^n)))}^2=\esssup_{t\in[0,T]}
\dbE\int_t^T|\f(t,s)|^2ds=\esssup_{t\in[0,T]}
\dbE\[\dbE_t\(\int_t^T|\f(t,s)|^2ds\)\]\\
\ns\ds\les\esssup_{t\in[0,T]}\Big\|\dbE_t\[\int_t^T|\f(t,s)|^2ds\]\Big\|_\infty
\les\esssup_{t\in[0,T]}\sup_{\t\in\sT[t,T]}\Big\|\dbE_\t\[\int_\t^T|\f(t,s)|^2ds\]
\Big\|_\infty\\
\ns\ds=\big\|\f(\cd\,,\cd)\big\|^2_{\overline{{\rm BMO}}(\D^*[0,T])}.\ea\ee
Hence,
\bel{BMO2}\cl{\rm BMO}(\D^*[0,T];\dbR^n)\subseteq L^\infty(0,T;L^2_\dbF(\O; L^2(\cd\,,T;\dbR^n))).\ee

At the moment, it is not clear if the equalities in \rf{BMO1} and \rf{BMO2} hold. They will again be related to the question we faced for \rf{Lp} with $p=\infty$. Fortunately it is irrelevant whether there are such equalities.

\ms

In what follows, for simplicity of the presentation, we will misuse the same notation for different range spaces which can be easily identified from the contexts.

\ms

\bde{solution} \rm A pair $(Y(\cd),Z(\cd\,,\cd))\in L^1_\dbF(0,T;\dbR^n)\times L^{ 2 }_\dbF(\O;L^2(\D^*[0,T];\dbR^n))$ is called an {\it adapted solution} of BSVIE \rf{BSVIE1} if the equation is satisfied in the usual It\^o's sense.

\ede

Note that although in the above definition, we only require $(Y(\cd),Z(\cd\,,\cd))$ to be in a very large space, in our main results, under proper conditions, our adapted solutions will have much better regularity/integrability.

\ms

\subsection{Results for BSDEs revisited}

In this subsection, we recall some relevant results for BSDEs.
Consider the following BSDE of integral form (which could be one-dimensional or multi-dimensional):
\bel{BSDE1}Y(t)=\xi+\int_t^Tf(s,Y(s),Z(s))ds-\int_t^TZ(s)dW(s),\ \  t\in[0,T].\ee
We look at the following classes of problems.

\ms

(i) Quadratic BSDEs with bounded terminal terms.

\ms

The main feature of such a BSDE is that $z\mapsto f(t,y,z)$ has a quadratic growth. Because of this, such an equation is called a {\it quadratic} BSDE (QBSDE, for short). We introduce the following hypothesis.

\ms

{\bf(H2.1)} The map $f:[0,T]\times\dbR^n\times\dbR^n\times\O\to\dbR^n$ is measurable such that $t\mapsto f(t,y,z)$ is $\dbF$-progressively measurable, for all $(y,z)\in\dbR^n\times\dbR^n$. There exists a non-negative valued process $\a(\cd)\in L^\infty_\dbF(\O;L^1(0,T;\dbR))$, a continuous, monotone increasing function $\phi:[0,\infty)\to[0,\infty)$ with $\phi(0)=0$, and constants $\b,\l\ges0$, $\g\ges\bar\g>0$, $\d\in(0,1)$ such that
\bel{H2.1a}\ba{ll}
\ss\ds\sgn(y^i)f^i(t,y,z)\les\a(t)\1n+\1n\b|y|\1n+\1n\l\sum_{j\ne i}|z^j|^{1+\d}\1n+\1n{\g\over2}|z^i|^2,\\
\ss\ds\qq\qq\qq\qq\qq\qq\qq\forall(t,y,z)\in[0,T]\times\dbR^n\times\dbR^n,\q1\les i\les n,\ea\ee
and either
\bel{H2.1b}\ba{ll}
\ss\ds-\a(t)\1n-\1n\b|y|\1n-\1n\l\sum_{j\ne i}|z^j|^{1+\d}\1n+\1n{\bar \g\over2}|z^i|^2\les f^i(t,y,z),\\
\ss\ds\qq\qq\qq\qq\qq\qq\qq\forall(t,y,z)\in[0,T]\times\dbR^n\times\dbR^n,\q1\les i\les n,\ea\ee
or
\bel{H2.1b*}\ba{ll}
\ss\ds\a(t)\1n+\1n\b|y|\1n+\1n\l\sum_{j\ne i}|z^j|^{1+\d}\1n-\1n{\bar \g\over2}|z^i|^2\ges f^i(t,y,z),\\
\ss\ds\qq\qq\qq\qq\qq\qq\qq\forall(t,y,z)\in[0,T]\times\dbR^n\times\dbR^n,\q1\les i\les n.\ea\ee
Further, it holds that
\bel{g-g}\ba{ll}
\ss\ds|f^i(t,y,z)-f^i(t,\bar y,\bar z)|\\
\ss\ds\les\phi(|y|\vee|\bar y|)
\Big[\big(1+|z|+|\bar z|\big)\big(|y-\bar y|+|z^i-\bar z^i|\big)+\big(1+|z|^\d+|\bar z|^\d\big)\sum_{j\neq i}|z^j-\bar z^j|\Big],\\
\ss\ds\qq\qq\qq\qq\qq\qq\qq\forall t\in[0,T],~y,\bar y,z,\bar z\in\dbR^n,~i=1,2,\cds,n.\ea\ee

Note that in the whole paper, we use the notation $\sgn(x):=I_{(x>0)}-I_{(x\les 0)}$. Condition \rf{H2.1b}--\rf{H2.1b*} are referred to as the {\it diagonally strictly quadratic} condition for the generator $f$ in $z$. This condition excludes the case that $z^i\mapsto f^i(t,y,z)$ is of linear growth. The following is Theorem 2.5 of Fan--Hu--Tang \cite{Fan-Hu-Tang 2020} which gives the existence and uniqueness of adapted solutions to QBSDEs with bounded terminal terms.

\bp{Prop2.1} \sl Let {\rm(H2.1)} hold. Then for any $\xi\in L^\infty_{\cF_T}(\O;\dbR^n)$, BSDE \rf{BSDE1} admits a unique adapted solution $(Y(\cd),Z(\cd))\in L^\infty_\dbF(\O;C([0,T];\dbR^n))\times\cl{\rm BMO}([0,T];\dbR^n)$.

\ep

It is very natural to ask what happens if the terminal state is unbounded for QBSDEs. This leads to the second situation as follows.

\ms

(ii) QBSDEs with unbounded terminal terms.

\ms

For such a situation, to get a positive result, the following hypothesis has been assumed in the literature.

\ms

{\bf(H2.2)} Let the measurability condition in (H2.1) hold for $f:[0,T]\times\dbR^n\times\dbR^n\times\O\to\dbR^n$. For each $i=1,2,\cds,n$, $f^i$ only depends on $(t,y,z^i,  \omega )$ (independent of $z^j$, $j\ne i$) and $z^i\mapsto f^i(t,y,z^i)$ is convex for all $(t,y)\in [0,T]\times \dbR^n$. There exists some $\a(\cd)\in\cap_{\m\ges1}\cE^{\m,1}_\dbF(\O;L^1(0,T;\dbR))$
and constants $\b,\g>0$, such that
\bel{f^i}|f^i(t,y,z^i)|\les\a(t)+\b|y|+{\g\over2}|z^i|^2,\q\forall(t,y,z^i)\in[0,T]\times\dbR^n\times\dbR\ee
and
\bel{23}|f(t,y,z)-f(t,\bar y,z)|\les\b|y-\bar y|,\q\forall t\in[0,T],~y,\bar y,z\in\dbR^n.\vspace{0.2cm}
\ee

The following is Theorem 2.9 of Fan--Hu--Tang \cite{Fan-Hu-Tang 2020}.

\bp{Prop2.2} \sl Let {\rm(H2.2)} hold. Then, for any $\xi\in\bigcap_{\m\ges1}\cE_{\cF_T}^{\m,1}(\O;\dbR^n)$, BSDE \rf{BSDE1} admits a unique adapted solution
$$(Y(\cd),Z(\cd))\in\Big[\bigcap_{\m\ges1}\cE^{
\m,1}_\dbF(\O;C(0,T];\dbR^n))\Big]\times\Big[\bigcap_{p\ges1} L^p_\dbF(\O;L^2(0,T;\dbR^n))\Big].$$

\ep

By \rf{inclusion}, we see that the terminal term in the above could be unbounded. However, the conditions assumed for the generator $f$ in (H2.2) is much stronger than that assumed in (H2.1), especially $f^i$ has to be independent of $z^j$, $j\ne i$, although the condition on $\a(\cd)$ in (H2.2) is a little weaker. Also, due to the assumption that $z\mapsto f(t,y,z)$ is diagonally dependent, the diagonally strictly quadratic growth condition (which is used to take care the sub-quadratic growth of off-diagonal component) is not necessary.

\ms

Next, if $z\mapsto f(t,y,z)$ is growing superlinearly but subquadratically, and diagonally dependent, then Proposition \ref{Prop2.2} should still be true. It is expected that we might be able to relax the condition on the terminal term. This will lead to the third situation.

\ms

(iii) Sub-quadratic BSDEs with unbounded terminal term.

\ms

For such a situation, there were only some results for one-dimensional BSDEs, under the following hypothesis:

\ms

{\bf(H2.3)} Let the measurability condition in (H2.1) hold for $f:[0,T]\times\dbR\times\dbR\times\O\to\dbR$. There exist some constants $\d\in(0,1)$, $\b\ges0$, $\g>0$, and $\a(\cd)\in\bigcap_{\m\ges1}\cE^{\m,{2\d\over1+\d}}_\dbF(\O;L^1(0,T;\dbR))$ such that
\bel{|f|}|f(t,y,z)|\les\a(t)+\b|y|+\g|z|^{1+\d},\q\forall(t,y,z)\in[0,T]\times\dbR\times\dbR.\ee
Moreover, $(y,z)\mapsto f(t,y,z)$ is convex for all $t\in[0,T]$.

\ms

The following is from Theorem 3.9 in Fan-Hu \cite{Fan-Hu 2021}.

\bp{Prop2.3} \sl Let {\rm(H2.3)} hold. Then for any $\xi\in\bigcap_{\m\ges1}\cE_{\cF_T}^{\m,{2\d\over1+\d}}(\O;\dbR)$, BSDE \rf{BSDE1} admits a unique adapted solution
$$(Y(\cd),Z(\cd))\in\Big[\bigcap_{\m\ges1}\cE^{\m,{2\d\over1+\d}}_\dbF(\O;C([0,T];\dbR))
\Big]\times  \Big[\bigcap_{p\ges1} L^p_\dbF(\O;L^2(0,T;\dbR))\Big] .$$

\ep

Since for $\d\in(0,1)$, one has ${2\d\over1+\d}<1$. Consequently,
$$\cE^{\m,1}_\dbF(\O;L^1(0,T;\dbR))\subseteq\cE_\dbF^{\m,{2\d\over1+\d}}(\O;L^1(0,T;\dbR)).$$
This implies that the condition on $\a(\cd)$ assumed in (H2.3) is weaker than that in (H2.2). Also,
$$\cE^{\m,1}_{\cF_T}(\O;\dbR)\subseteq\cE_{\cF_T}^{\m,{2\d\over1+\d}}(\O;\dbR).$$
Hence, the condition assumed for $\xi$ in Proposition \ref{Prop2.3} is weaker than that in Proposition \ref{Prop2.2}. The trade-off is that Proposition \ref{Prop2.3} is only for one-dimensional BSDEs.\vspace{0.2cm}

Furthermore, we note that for one-dimensional BSDEs with the dependence of the generator on $Z$ being no more than quadratic, in Fan-Hu-Tang \cite{Fan-Hu-Tang 2019} and Fan-Hu \cite{Fan-Hu 2021}, the authors also proved a general existence and uniqueness result under the following hypothesis:

\ms

{\bf(H2.4)} Let the measurability condition in (H2.1) hold for $f:[0,T]\times\dbR\times\dbR\times\O\to\dbR$. There exist some constants $\b\ges0$, $\g>0$, $\d\in(0,1]$, and $\a(\cd)\in\cap_{\m\ges1}\cE^{\m,{2\d\over1+\d}}_\dbF(\O;L^1(0,T;\dbR))$ such that
\bel{|f|}
|f(t,y,z)|\les\a(t)+\b|y|+\g|z|^{1+\d},
\q\forall(t,y,z)\in[0,T]\times\dbR\times\dbR, \ee
\bel{f2}|f(t,y,z)-f(t,\bar y,z)|\les\b|y-\bar y|,\q\forall(t,y,\bar y,z)\in[0,T]\times\dbR\times\dbR\times\dbR\ee
and
\bel{f3}\ba{ll}
\ds f(t,y,(1-\th)z+\th\bar z)-\th f(t,y,\bar z)\les (1-\th)
\big(\a(t)+\b|y|+\g |z|^{1+\d}\big),\\
\ns\ds\qq\qq\qq\qq\qq \forall (t,y,z,\bar z)\in[0,T]\times\dbR\times\dbR\times\dbR,~\th\in(0,1).
\ea\ee
Moreover, $z\mapsto f(t,y,z)$ is continuous for all $(t,y)\in [0,T]\times \dbR$.\vspace{0.3cm}

Observe that if \rf{f2} and \rf{f3} hold, then for each $(t,y_1,y_2,z_1,z_2)\in[0,T]\times\dbR\times\dbR\times\dbR\times\dbR$ and each $\th\in(0,1)$, we have
$$
\ba{ll}
\ds g(t,y_1,z_1)-\th g(t,y_2,z_2)\les|g(t,y_1,z_1)- g(t,y_2,z_1)|+g(t,y_2,(1-\th)\hat z_{\th}+\th z_2)-\th g(t,y_2,z_2)\\
\ns\ds \hspace{3.75cm} \les \b|y_1-y_2|+(1-\th)(\a(t)+\b |y_2|+\g |\hat z_{\th}|^{1+\d})\\
\ns\ds \hspace{3.75cm} \les (1-\th)(\b |\hat y_{\th}|+\b |y_2|+\a(t)+\b |y_2|+\g |\hat z_{\th}|^{1+\d})\\
\ns\ds \hspace{3.75cm} =(1-\th)(\a(t)+2\b |y_2|+\b |\hat y_{\th}|+\g |\hat z_{\th}|^{1+\d}),\ea$$
where
$$
\hat y_{\th}:=\frac{y_1-\th y_2}{1-\th}\qq {\rm and}\qq \hat z_{\th}:=\frac{z_1-\th z_2}{1-\th}.\vspace{0.1cm}
$$
Then, by virtue of (i) of Theorem 5 in Fan-Hu-Tang \cite{Fan-Hu-Tang 2019} for the case $\d=1$ and Theorem 3.9 in Fan-Hu \cite{Fan-Hu 2021} for the case $\d\in (0,1)$, the following proposition follows immediately.\vspace{0.1cm}

\bp{Prop2.4} \sl Let {\rm(H2.4)} hold. Then for any $\xi\in\bigcap_{\m\ges1}\cE_{\cF_T}^{\m,{2\d\over1+\d}}(\O;\dbR)$, BSDE \rf{BSDE1} admits a unique adapted solution
$$(Y(\cd),Z(\cd))\in\Big[\bigcap_{\m\ges1}\cE^{\m,{2\d\over1+\d}}_\dbF(\O;C([0,T];\dbR))
\Big]\times  \Big[\bigcap_{p\ges1} L^p_\dbF(\O;L^2(0,T;\dbR))\Big].\vspace{0.1cm}
$$
\ep

\br{rmk:2.5} \rm
If the generator $g$ satisfies assumption {\rm(H2.4)} with \rf{f3} replaced with
\bel{f4}\ba{ll}
\ns\ds f(t,y,(1-\th)z+\th\bar z)-\th f(t,y,\bar z)\ges -(1-\th)
\big(\a(t)+\b|y|+\g |z|^{1+\d}\big),\\
\ns\ds\qq\qq\qq\qq\qq \forall (t,y,z,\bar z)\in[0,T]\times\dbR\times\dbR\times\dbR,~\th\in(0,1),\ea\ee
then it is easy to verify that $\bar g$ must satisfy assumption {\rm(H2.4)}, where
$$
\bar g(t,y,z):=-g(t,-y,-z),\q  \forall (t,y,z)\in [0,T]\times\dbR \times\dbR.
$$
Consequently, when \rf{f3} in {\rm(H2.4)} is replaced with \rf{f4}, the conclusion of Proposition \ref{Prop2.4} still holds.

\ms

Furthermore, it is well known that a convex function with linear growth must be a Lipschitz function. And, it is not hard to check that if $z\mapsto f(t,y,z)$ is convex for all $(t,y)\in[0,T]\times \dbR$ and \rf{|f|} holds, then \rf{f3} is also true. Then, {\rm(H2.4)} is weaker than {\rm(H2.3)} and {\rm(H2.2)} in the one-dimensional setting. Hence, Proposition \ref{Prop2.4} generalizes Proposition \ref{Prop2.3} and Proposition \ref{Prop2.2} for the one-dimensional case.

\ms

In addition, in this paper we will show that \rf{f3} (also \rf{f4}) can imply the continuity of $g$ with respect to $z$ in {\rm(H2.4)}, see Section 4 for details.
\er

In what follows, we will substantially extend the above results to BSVIEs, by which we mean that some of the extensions are new even for BSDEs.

\subsection{ A crucial lemma }

We now present a lemma which has played a crucial rule in proving the above results, and it will play some important role below when we discuss the situation of unbounded free terms.

\ms

\bl{Lemma2.5} \sl Suppose that $\xi$ is an $\cF_T$-measurable random variable and random field $f:[0,T]\times\dbR^n\times\dbR^n\times\O\to\dbR^n$ is measurable such that for each $(y,z)\in\dbR^n\times\dbR^n$, $t\mapsto f(t,y,z)$ is $\dbF$-progressively measurable.
Let
$(Y(\cd),Z(\cd))$
be an adapted solution of BSDE \rf{BSDE1}
such that
$Y(\cd)\in\bigcap_{\m\ges 1}\cE^{\mu,\frac{2\d}{1+\d}}_\dbF(\Omega;C([0,T];\dbR^n))$ and
\bel{2.9}\sgn \big(Y^i(s)\big)f^i(s,Y(s),Z(s))\les \a_0(s)+\b|Z^i(s)|^{1+\d},\q s \in[0,T],~1\les i\les n,\ee
for some constants $\b>0$, $\d\in(0,1]$, and some map
$\a_0(\cd)\in\bigcap_{\mu\ges 1}\cE^{\mu,\frac{2\d}{1+\d}}_\dbF(\Omega;L^1(0,T;\dbR_+))$.
Then there exists a constant $K>0$ depending only on $(\d,\b,T)$ such that for each $i=1,\cdots,n$,
\bel{2.13}e^{|Y^i(t)|^{2\d\over1+\d}}+\dbE_t\Big[\int_t^T|Z^i(s)|^2ds\Big]\les K\dbE_t\Big
[e^{K\big(|\xi^i|+\int_t^T\a_0(\t)d\t\big)^{2\d\over1+\d}}\Big],\ \ t\in [0,T],\ee
and for each $p>0$, there exists a constant $K_p>0$ depending only on $(\d,\b,T)$ and $p$ such that for each $i=1,\cdots,n$,
\bel{new-2-5}
\ds\dbE\Big[\sup_{s\in [t,T]} e^{p|Y^i(s)|^{2\d\over1+\d}}\Big]+\dbE\Big[\Big(\int_t^T |Z^i(s)|^2ds\Big)^{p\over 2}\Big]\les K_p\dbE\Big[e^{K_p\big(|\xi^i|+\int_t^T\a_0(\t)d\t\big)^{2\d\over1+\d}}\Big], \ \ t\in [0,T].
\ee
\el

\ms

\it Proof. \rm Let $t\in[0,T)$ be fixed, $(Y(\cd),Z(\cd))$ be the adapted solution of BSDE \rf{BSDE1}. Let
$$\f(r,\rho)=e^{\m(r)\rho^\th},\qq(r,\rho)\in[t,T]\times[\rho_0,\infty),$$
with $\m(\cd)\ges1$, $\m'(r)>0$, $\rho_0>1$, and $\th\in (0,1]$, all being undetermined. Note that the smaller the $\th$, the smaller the $\f(r,\rho)$. We calculate
$$\ba{ll}
\ss\ds\f_r(r,\rho)=\f(r,\rho)\m'(r)\rho^\th,\\
\ss\ds\f_\rho(r,\rho)=\f(r,\rho)\th\m(r)\rho^{\th-1},\\
\ss\ds\f_{\rho\rho}(r,\rho)=\f(r,\rho)\(\th^2\m(r)^2\rho^{2\th-2}+\th(\th-1)\m(r)
\rho^{\th-2}\)\\
\ss\ds\qq\qq=\f(r,\rho)\th\m(r)\rho^{\th-2}\(\th\m(r)\rho^\th-(1-\th)\)\ges\
\f(r,\rho)\th\m(r)\rho^{\th-2}(\th\rho_0^\th-1)>0,\ea$$
provided $\rho_0>\th^{-{1\over\th}}$. For such a case, $\rho\mapsto\f(r,\rho)$ is strictly convex on $[\rho_0,\infty)$. Note that $\ds\lim_{\th\downarrow0}\th^{-{1\over\th}}=+\infty$. Thus, in this approach, when $\th$ is getting small, $\rho_0$ is getting large. For $i=1,\cdots,n$, define
$$\cY^i(r)=|Y^i(r)|+k+\int_t^r\a_0(s)ds,\ \ r\in[t,T],$$
with $k\ges \rho_0$, and let $L^i(\cd)$ denote the local time of $Y^i(\cd)$ at $0$. In view of \rf{2.9}, applying It\^o-Tanaka's formula to $s\mapsto\f(s,\cY^i(s))$, one has
$$\ba{ll}
\ss\ds d\f\big(s,\cY^i (s)\big)=\(\f_s\big(s,\cY^i (s)\big)+\f_\rho\big(s,\cY^i (s)\big)
\big(-\sgn\big(Y^i(s)\big)f^i(s,Y(s),Z(s))+\a_0(s)\big)\\
\ss\ds\qq\qq\qq\ \ +{1\over2}\f_{\rho\rho}\big(s,\cY^i (s)\big)|Z^i(s)|^2\)ds+\f_\rho\big(s,\cY^i (s)\big)dL^i(s)+\f_\rho\big(s,\cY^i (s)\big)\sgn\big(Y^i(s)\big)Z^i(s)dW(s)\\
\ss\ds\hspace{1.95cm}\ges\f\big(s,\cY^i (s)\big)\Big[\m'(s)\cY^i (s)^\th-\th\m(s)\cY^i (s)^{\th-1}
\b |Z^i(s)|^{1+\d}\\
\ss\ds\hspace{4.3cm}+{\th\m(s)\over 2}\cY^i (s)^{\th-2}\(\th\m(s)\cY^i (s)^\th
-(1-\th)\)|Z^i(s)|^2\Big]ds\\
\ss\ds\hspace{2.3cm}+\f\big(s,\cY^i (s)\big)\th\m(s)\cY^i (s)^{\th-1}\sgn\big(Y^i(s)\big)Z^i(s)dW(s),\ \  s\in [t,T].\ea$$
For each $r\in [t,T]$ and $m\ges 1$, define the following stopping time:
$$
\sigma_m^r:=\inf\Big\{s\in [r,T]\big| |Y(s)|+\int_r^s |Z(\t)|^2 d\t\ges m \Big\}\wedge T
$$
with the convention $\inf \emptyset=+\infty$. Then, by integrating on
$[r,\sigma_m^r]$ we get that for each $m\ges 1$ and $r\in [t,T]$,
\bel{new-1}
\ba{ll}
\ss\ds\f\big(r,\cY^i(r)\big)+X_r^m \les\f\big(\sigma_m^r,\cY^i(\sigma_m^r)\big)+\int_r^{\sigma_m^r}
\f\big(s,\cY^i(s)\big)\Big[
-\m'(s)\cY^i(s)^\th\\
\ss\ds\qq\qq\qq +{\th\m(s)\over2}\cY^i(s)^{\th-2}\(2\b\cY^i(s)|Z^i(s)|^{1+\d}
-\th\m(s)\cY^i(s)^\th|Z^i(s)|^2+(1-\th)|Z^i(s)|^2\)\Big]ds,\ea\ee
where
$$
X_r^m:=\int_r^{\sigma_m^r} \f\big(s,\cY^i (s)\big)\th\m(s)\cY^i (s)^{\th-1}\sgn\big(Y^i(s)\big)Z^i(s)dW(s).
$$
We split the rest into two cases.

\ms

\it Case 1. \rm For $\d\in(0,1)$, one has
$$\ba{ll}
\ss\ds2\b\cY^i(s)|Z^i(s)|^{1+\d}=2\b\(\cY^i(s)^\th|Z^i(s)|^2\)^{1+\d\over2}
\cY^i(s)^{1-{\th{1+\d\over2}}}\\
\ss\ds\les{\th\over2}\cY^i(s)^\th|Z^i(s)|^2+K_0\cY^i(s)^{
(1-\th{1+\d\over2}){2\over1-\d}}={\th\over2}\cY^i(s)^\th|Z^i(s)|^2+K_0\cY^i(s)^{
2-\th(1+\d)\over1-\d},\ \ s\in [t,T].\ea$$
In the above, $K_0>0$ is an absolute constant, only depending on $\b,\d,\th$. We now require that
$${2-\th(1+\d)\over1-\d}\les2,$$
which is equivalent to $\th\ges{2\d\over1+\d}$. Hence, the best possibility is to choose
$\th={2\d\over1+\d}$. Then with such a $\th$, we have (noting $\mu(s)\ges 1$ and $\cY^i(s)\ges k$) that for each $m\ges 1$ and $r\in [t,T]$,
\bel{new-2}
\ba{ll}
\ss\ds\f(r,\cY^i(r))+X_r^m\les\f\big(\sigma_m^r,\cY^i(\sigma_m^r)\big)
+\int_r^{\sigma_m^r}
\f\big(s,\cY^i(s)\big)\cY^i(s)^\th
\Big[-\(\m'(s)-{\th K_0\over 2}\m(s)\)\\
\ss\ds\hspace{6cm}-{\th\m(s)\over2}\({\th \over 2}\cY^i(s)^\th-(1-\th)\)|Z^i(s)|^2\)\Big]ds\\
\ss\ds\hspace{2.7cm}\les\f\big(\sigma_m^r,\cY^i(\sigma_m^r)\big)
+\int_r^{\sigma_m^r}\f\big(s,\cY^i(s)\big)
\cY^i(s)^\th
\Big[-\(\m'(s)-{\th K_0\over 2}\m(s)\)\)\\
\ss\ds\hspace{6cm}-{\th\m(s)\over2}\({\th\over 2}k^\th-1\)|Z^i(s)|^2\)\Big]ds.\ea\ee
Now, we choose $k\ges 1+\th^{-{1\over \th}}$ and $\m(\cd)\ges 1$ so that
\bel{new-4}
{\th \over2}\({\th\over2}k^\th-1\)\ges1,\qq\m'(s)\ges {\th K_0\over 2}\m(s),\q s\in[t,T].\ee
Then, for each $m\ges 1$ and $r\in [t,T]$, in view of $\dbE_r[X_r^m]=0$, by taking conditional expectation with respect to $\cF_r$ in \rf{new-2} we can obtain that
\bel{f1}\f(r,\cY^i(r))+\dbE_r\Big[\int_r^{\sigma_m^r}|Z^i(s)|^2ds\Big]\les
\dbE_r\big[\f\big(\sigma_m^r,\cY^i(\sigma_m^r)\big)\big],\ee
with $k$ depending on $\d$. By Fatou's lemma and Lebesgue's dominated convergence theorem, as well as the integrability of $Y(\cd)$ and $\a_0(\cd)$, sending $m$ to infinity in \rf{new-4} yields the following
$$e^{\m(r)\big(|Y^i(r)|+k+\int_t^r\a_0(\t)d\t\big)^{2\d\over 1+\d}}+\dbE_r\Big[\int_r^T|Z^i(s)|^2ds\Big]\les
\dbE_r\Big[e^{\m(T)\big(|\xi^i|+k+\int_t^T\a_0(\t)d\t\big)^{2\d\over1+\d}}\Big],\ \  r\in [t,T],$$
which is just the following result
\bel{2.12}e^{\m(t)\big(|Y^i(t)|+k\big)^{2\d\over1+\d}}+\dbE_t\Big[\int_t^T|Z^i(s)|^2ds\Big]\les
\dbE_t\Big[e^{\m(T)\big(|\xi^i|+k+\int_t^T \a_0(\t)d\t\big)^{2\d\over1+\d}}\Big]\ee
for the case $\d\in(0,1)$ when $r=t$. Furthermore, it follows from Doob's martingale inequality that for each $p>1$,
\bel{new-2-1}\ba{ll}
\ds\dbE\Big[\sup_{r\in [t,T]} e^{p\m(r)\big(|Y^i(r)|+k+\int_t^r\a_0(\t)d\t\big)^{2\d\over1+\d}}\Big]\les
\dbE\Big[\sup_{r\in [t,T]}\Big\{\dbE_r\Big[e^{\m(T)\big(|\xi^i|+k+\int_t^T\a_0(\t)d\t\big)^{2\d\over1+\d}}\Big]\Big\}^p\Big]\\
\hspace{6.2cm}\ds\les \big({p\over p-1}\big)^p\dbE\Big[e^{p\m(T)\big(|\xi^i|+k+\int_t^T\a_0(\t)d\t\big)^{2\d\over1+\d}}\Big].
\ea\ee
On the other hand, by \rf{new-2} and \rf{new-4} we can deduce that for each $m\ges 1$ and $p>1$,
$$
\dbE\Big[\big(\int_r^{\sigma_m^r} |Z^i(s)|^2ds\big)^{p\over 2}\Big]\les 2^{p-1}\Big\{\dbE\Big[\big(\f\big(\sigma_m^r,\cY^i(\sigma_m^r)\big)\big)^{p\over 2}\Big]+\dbE\big[|X_r^m|^{p\over 2}\big]\Big\},\ \ r\in [t,T].
$$
By virtue of BDG's inequality, the definition of $X_r^m$ and H\"{o}lder's inequality and in view of $\cY^i(s)\ges k\ges 1$ and $\th\in (0,1)$, observe that
$$
\ba{ll}
\ds2^{p-1}\dbE\big[|X_r^m|^{p\over 2}\big]\les 2^{p-1} (\m(T))^{p\over 2} \dbE\Big[\Big(\int_r^{\sigma_m^r} \big[\f\big(s,\cY^i (s)\big)\big]^2|Z^i(s)|^2 ds\Big)^{p\over 4}\Big]\\
\hspace{2.5cm}\ds\les 2^{p-1}(\m(T))^{p\over 2} \dbE\Big[ \sup_{s\in [r,\sigma_m^r]}\big[\big(\f\big(s,\cY^i (s)\big)\big)^{p\over 2}\big] \Big(\int_r^{\sigma_m^r}|Z^i(s)|^2 ds\Big)^{p\over 4}\Big]\\
\hspace{2.5cm}\ds\les 2^{2p-3}(\m(T))^p\dbE\[\sup_{s\in [r,\sigma_m^r]}\big(\f\big(s,\cY^i (s)\big)\big)^p\]+{1\over 2} \dbE\Big[\Big(\int_r^{\sigma_m^r}|Z^i(s)|^2 ds\Big)^{p\over 2}\Big],\ \ r\in [t,T].
\ea
$$
We have
$$
\ba{ll}
\ds\dbE\Big[\Big(\int_r^{\sigma_m^r} |Z^i(s)|^2ds\Big)^{p\over 2}\Big]\les 2^p\dbE\Big[\big(\f\big(\sigma_m^r,\cY^i(\sigma_m^r)\big)\big)^{p\over 2}\Big]+2^{2p-2}(\m(T))^p\dbE\Big[\sup_{s\in [r,\sigma_m^r]}\big(\f\big(s,\cY^i (s)\big)\big)^p\Big]\\
\hspace{3.7cm}\ds\les (8\m(T))^p\dbE\Big[\sup_{s\in [r,\sigma_m^r]}\big(\f\big(s,\cY^i (s)\big)\big)^p\Big],\qq r\in [t,T],\ea
$$
and then (letting $m\rightarrow \infty$ and $r=t$)
\bel{new-2-2}\ba{ll}
\ds\dbE\Big[\Big(\int_t^T |Z^i(s)|^2ds\Big)^{p\over 2}\Big]\les (8\m(T))^p\dbE\Big[\sup_{s\in [t,T]}\big(\f\big(s,\cY^i (s)\big)\big)^p\Big]\\
\hspace{3.5cm}\ds =(8\m(T))^p\dbE\Big[\sup_{r\in [t,T]} e^{p\m(r)\big(|Y^i(r)|+k+\int_t^r\a_0(\t)d\t\big)^{2\d\over1+\d}}\Big].
\ea\ee
Combining \rf{new-2-1} and \rf{new-2-2} yields
\bel{new-2-4}
\ba{ll}
\ds\dbE\Big[\sup_{s\in [t,T]} e^{p\m(s)\big(|Y^i(s)|+k+\int_t^s\a_0(\t)d\t\big)^{2\d\over1+\d}}\Big]
+\dbE\Big[\Big(\int_t^T |Z^i(s)|^2ds\Big)^{p\over 2}\Big]\\
\ds\qq\les \big({p\over p-1}\big)^p\big((8\m(T))^p+1 \big)\dbE\Big[e^{p\m(T)\big(|\xi^i|+k+\int_t^T
\a_0(\t)d\t\big)^{2\d\over1+\d}}\Big],\ \ \forall p>1.
\ea\ee
for the case $\delta\in (0,1)$.
Finally, \rf{2.13} and \rf{new-2-5} follows from \rf{2.12} and \rf{new-2-4} as well for the case of $\d\in(0,1)$.

\ms

\it Case 2. \rm For $\d=1$, we have that $\th={2\d\over1+\d}=1$, and it follows from \rf{new-1} that for each $m\ges 1$,
\bel{new-3}
\ba{ll}
\ss\ds\f(r,\cY^i(r))+X_r^m\les\f\big(\sigma_m^r,\cY^i(\sigma_m^r)\big)
+\int_r^{\sigma_m^r}\f\big(s,\cY^i(s)
\big)\Big[-\cY^i(s)\m'(s)\\
\ss\ds\hspace{5.5cm}+{\m(s)\over2}\cY^i(s)^{-1}
\(2\b\cY^i(s)|Z^i(s)|^2-\m(s)\cY^i(s)|Z^i(s)|^2\)\Big]ds\\
\ss\ds\ \ =\f\big(\sigma_m^r,\cY^i(\sigma_m^r)\big)
+\int_r^{\sigma_m^r}\f\big(s,\cY^i(s)\big)
\Big[-\cY^i(s)\m'(s)+{\m(s)\over2}
\big(2\b-\m(s)\big)|Z^i(s)|^2\Big]ds,\ \ r\in [t,T].\ea\ee
Hence, by choosing
$$\m'(s)>0,\qq{\m(s)\over2}\big(\m(s)-2\b\big)\ges 1,\q s\in[0,T],$$
and taking conditional expectation with respect to $\cF_r$ in \rf{new-3} we get \rf{f1} for the case $\d=1$, and sending $m$ to infinity and letting $r=t$ yields \rf{2.12} for the case $\d=1$. Then, a similar argument to that from \rf{new-2-1} to \rf{new-2-2} yields \rf{new-2-4} for the case $\delta=1$.
Finally, \rf{2.13} and \rf{new-2-5} follows from \rf{2.12} and \rf{new-2-4} as well for the case of $\d=1$. \endpf

\br{rmk:2.6}
\rm
The following assertion can be obtained by a similar proof to above. Let the assumptions in Lemma \ref{Lemma2.5} hold with \rf{2.9} replaced by
$$I_{(Y^i(s)>0)}f^i(s,Y(s),Z(s))\les \a_0(s)+\b|Z^i(s)|^{1+\d},\q \as.$$
Then there exist a $k\ges 1$ and a differentiable function $\m:[0,T]\to[1,\infty)$ such that for each $i=1,\cdots,n$,
$$
e^{\m(t)\big((Y^i(t))^+ +k\big)^{2\d\over1+\d}}+\dbE_t\Big[\int_t^T  I_{(Y^i(s)>0)} |Z^i(s)|^2ds\Big]\les
\dbE_t\Big[e^{\m(T)\big((\xi^i)^+ +k+\int_t^T \a_0(\t) d\t\big)^{2\d\over1+\d}}\Big],\ \ t\in[0,T].
$$
Consequently, there exists a constant $K>0$ depending only on $(\d,\b,T)$ such that for each $i=1,\cdots,n$,
$$
e^{[(Y^i(t))^+]^{2\d\over1+\d}}+\dbE_t\Big[\int_t^T I_{(Y^i(s)>0)} |Z^i(s)|^2ds\Big]\les K\dbE_t\Big
[e^{K\big((\xi^i)^+ +\int_t^T\a_0(\t)d\t\big)^{2\d\over1+\d}}\Big],\ \ t\in [0,T].
$$
\er

\section{BSVIEs with Bounded Free Terms}

For convenience, we rewrite BSVIE \rf{BSVIE1} here
\bel{BSVIE1*}Y(t)=\psi(t)+\int_t^Tg(t,s,Y(s),Z(t,s))ds-\int_t^TZ(t,s)dW(s),
\q t\in[0,T].\ee
In this section, we consider the case of coupled multi-dimensional BSVIEs with bounded free terms. In what follows, we use $K>0$ to represent a generic constant which could be different from line to line.

\subsection{Diagonally strictly quadratic case}

Let us begin with the following hypothesis on the generator $g$ of BSVIE \rf{BSVIE1*} which is comparable with (H2.1) for BSDEs.

\ms

{\bf(H3.1)} Map $g:\D^*[0,T]\times\dbR^n\times\dbR^n\times\O\to\dbR^n$ is measurable so that $s\mapsto g(t,s,y,z)$ is $\dbF$-progressively measurable on $[t,T]$ for each $(t,y,z)\in[0,T)\times\dbR^n\times\dbR^n$. There exist a non-negative valued process $\a(\cd\,,\cd)\in L^\infty(0,T;L^\infty_\dbF(\O;L^1(\cd\,,T;\dbR)))$, a continuous, monotone increasing function $\phi:[0,\infty)\to[0,\infty)$ with $\phi(0)=0$, and constants $\b\ges\g>0$, $\d\in[0,1)$ such that
\bel{growth1}\ba{ll}
\ds-\a(t,s)-\b\(|y|+\sum_{j\ne i}|z^j|^{1+\d}\)+\g|z^i|^2\les g^i(t,s,y,z)\les\a(t,s)+\b\(|y|+\sum_{j\ne i}|z^j|^{1+\d}+|z^i|^2\),\\
\ss\ds\qq\qq\qq\qq\qq\qq\qq\qq\forall(t,s,y,z)\in\D^*[0,T]\times\dbR^n\times\dbR^n,~i=1,2,\cds,n.\ea\ee
Further, it holds that
\bel{g-g}\ba{ll}
\ss\ds|g^i(t,s,y,z)-g^i(t,s,\bar y,\bar z)|\les\phi(|y|\vee|\bar y|)
\Big[\big(1+|z|+|\bar z|\big)\big(|y-\bar y|+|z^i-\bar z^i|\big)\\
\ss\ds\qq\qq\qq\qq\qq\qq\qq\qq\qq+\big(1+|z|^\d+|\bar z|^\d\big)\sum_{j\neq i}|z^j-\bar z^j|\Big],\\
\ss\ds\qq\qq\qq\qq\qq\qq\qq\forall(t,s)\in\D^*[0,T],~y,\bar y,z,\bar z\in\dbR^n,~i=1,2,\cds,n.\ea\ee

\ms

Let us introduce the following seemingly a little more general hypothesis than the above (H3.1).

\ms

{\bf(H3.1)$'$ \rm Let (H3.1) hold with \rf{growth1} replaced by the following: For each $i=1,2,\cds,n$, either
\bel{growth1*}\ba{ll}
\ss\ds-\a(t,s)-\b\(|y|+\sum_{j\ne i}|z^j|^{1+\d}\)+\g|z^i|^2\les g^i(t,s,y,z)\les\a(t,s)+\b\(|y|+\sum_{j\ne i}|z^j|^{1+\d}+|z^i|^2\),\\
\ss\ds\qq\qq\qq\qq\qq\qq\qq\qq\qq\forall(t,s,y,z)\in\D^*[0,T]\times \dbR^n\times\dbR^n,\ea\ee
or
\bel{growth2*}\ba{ll}
\ss\ds\a(t,s)+\b\(|y|+\sum_{j\ne i}|z^j|^{1+\d}\)-\g|z^i|^2\ges g^i(t,s,y,z)\ges-\a(t,s)-\b\(|y|+\sum_{j\ne i}|z^j|^{1+\d}+|z^i|^2\),\\
\ss\ds\qq\qq\qq\qq\qq\qq\qq\qq\qq\forall(t,s,y,z)\in\D^*[0,T]\times \dbR^n\times\dbR^n.\ea\ee

We claim that (H3.1) and (H3.1)$'$ are essentially equivalent. In fact, if (H3.1)$'$ holds, by relabeling, we may assume that for some $k=0,1,2,\cds,n$, one has \rf{growth1*} holds for $1\les i\les k$, and \rf{growth2*} holds for $k+1\les i\les n$. If $k=n$, it means that for all $i=1,2,\cds,n$, \rf{growth1*} holds, and if $k=0$, it means that for all $i=1,2,\cds,n$, \rf{growth2*} holds. Now, we define
$$\ba{ll}
\ss\bar z\equiv(\bar z^1,\cds,\bar z^k,\bar z^{k+1},\cds,\bar z^n)=(z^1,\cds,z^k,-z^{k+1},
\cds,-z^n),\\
\ss\bar y\equiv(\bar y^1,\cds,\bar y^k,\bar y^{k+1},\cds,\bar y^n)=(y^1,\cds,y^k,-y^{k+1},
\cds,-y^n),\\
\ss\ds \bar g^i(t,s,y,z)=g^i(t,s,y,z) ,\qq~~1\les i\les k,\\
\ss\ds \bar g^i(t,s,y,z)=-g^i(t,s,\bar y,
\bar z) ,\qq k+1\les i\les n.\ea$$
Then (H3.1) holds for $\bar g(t,s,y,z)$. Thus, (H3.1) and (H3.1)$'$ are equivalent in such a sense. In what follows, we will use (H3.1) instead of (H3.1)$'$.

\ss

Next, condition \rf{growth1} means that the $i$-th component $g^i$ of the generator $g$ has exactly quadratic growth in the $i$-th component $z^i$ of $z$, and has sub-quadratical growth in the other components of $z$. Such a condition is referred to as the {\it diagonally strictly quadratic} condition of $z\mapsto g(t,s,y,z)$, which was introduced for BSDEs the first time in \cite{Fan-Hu-Tang 2019}. Roughly speaking, \rf{growth1} amounts to saying that
$$0<\g\les\liminf_{|z^i|\to\infty}{g^i(t,s,y,z)\over|z^i|^2}\les
\limsup_{|z^i|\to\infty}{g^i(t,s,y,z)\over|z^i|^2}\les\b.$$
Hypothesis (H3.1) is satisfied by many examples. Here is a typical one:
$$g^i(t,s,y,z)=\sqrt{(|y|+1)(|z|+1)}+|z|\cos|z^i|-|z|^{1+\d}+ |z^i|^2,\ \ i=1,\cds,n.$$
To establish the existence and uniqueness of adapted solution $(Y(\cd),Z(\cd\,,\cd))$ to our BSVIE \rf{BSVIE1}, we first present the following result.

\bp{pro:3.1} \sl Let {\rm(H3.1)} hold. Then, for each $\psi(\cd)\in L_{ \cF_T }^\infty(0,T;\dbR^n)$, $y(\cd)\in L^\infty_\dbF(0,T;\dbR^n)$, the following BSVIE
\bel{eq:1.4}Y(t)=\psi(t)+\int_t^T  g(t,s,y(s),Z(t,s))ds-\int_t^T Z(t,s)dW(s),\ \  t\in[0,T]\ee
admits a unique adapted solution $(Y(\cd),Z(\cd\,,\cd))\in L^\infty_\dbF(0,T;\dbR^n)
\times\overline{{\rm BMO}}(\D^*[0,T];\dbR^n)$. Furthermore, there exists a constant $K_0>0$ depending on $(\psi(\cd),\a(\cd\,,\cd),\b,\g,\d,n,T)$ such that if
\bel{|y|}\|y(t)\|_\infty\les K_0e^{K_0(T-t)},\qq\ae\;t\in[0,T],\ee
then
\bel{eq:1.5}\|Y(t)\|_\infty\les K_0e^{K_0(T-t)},\qq\ae\;t\in[0,T],\ee
and
\bel{eq:1.6}\|Z(\cd\,,\cd)\|_{\cl{\rm BMO}(\D^*\T)}\les\sqrt{nK_0}\exp\Big(K_0e^{K_0T}\)\equiv\bar K_0.\ee

\ep

\it Proof. \rm For given $\psi(\cd)\in L_{ \cF_T }^\infty(0,T;\dbR^n)$ and $y(\cd)\in L_\dbF^\infty(0,T;\dbR^n)$, it follows from Proposition \ref{Prop2.1} (or \cite[Theorem 2.5]{Fan-Hu-Tang 2020}) that for $\ae\,t\in[0,T]$, the following BSDE
\bel{eq:1.7}\eta(t,r)=\psi(t)+\int_r^T  g(t,s,y(s),\z(t,s))ds-\int_r^T \z(t,s)dW(s),\ \  r\in[t,T]\ee
admits a unique adapted solution $(\eta(t,\cd),\z(t,\cd))\in L^\infty_\dbF(\O;C([t,T];\dbR^n))\times\overline{{\rm BMO}}([t,T];\dbR^n)$. Moreover, there exists a constant $K>0$ depending on $(\a(\cd\,,\cd),\b,\g,\d,n,T)$ and $\|y(\cd)\|_{L^\infty_\dbF(0,T;\dbR^n)}$ such that
\bel{eq:1.8}
\|\eta(t,\cd)\|_{L^\infty_\dbF(\O;C([t,T];\dbR^n))}
+\|\z(t,\cd)\|_{\overline{{\rm BMO}}[t,T]}\les K,\q\ae~t\in[0,T].\ee
Let
$$Y(t)=\eta(t,t),\qq Z(t,s)=\z(t,s),\qq(t,s)\in\D[0,T].$$
Then $(Y(\cd),Z(\cd\,,\cd))\in L^\infty_\dbF(0,T;\dbR^n)\times \overline{{\rm BMO}}(\D^*[0,T];\dbR^n)$ is an adapted solution of BSVIE \eqref{eq:1.4}.

\ms

Suppose now that $(\wt Y(\cd),\wt Z(\cd\,,\cd))\in L^\infty_\dbF (0,T;\dbR^n)\times\overline{{\rm BMO}}(\D^*[0,T];\dbR^n)$ is another adapted solution of BSVIE \eqref{eq:1.4}. That is,
\bel{eq:1.9}\wt Y(t)=\psi(t)+\int_t^T  g(t,s,y(s),\wt Z(t,s))ds-\int_t^T \wt Z(t,s)dW(s),\ \   t\in[0,T].\ee
With the given $(\wt Y(\cd),\wt Z(\cd\,,\cd))$, for $\ae\;t\in[0,T]$, the following BSDE:
\bel{eq:1.10}\wt\eta(t,r)=\psi(t)+\int_r^Tg(t,s,y(s),\wt Z(t,s))ds-\int_r^T \wt\z(t,s)dW(s),\ \  r\in[t,T]\ee
admits a unique solution $(\wt\eta(t,\cd),\wt\z(t,\cd))\in L^\infty_\dbF(\O;C([t,T];\dbR^n))\times\overline{{\rm BMO}}([t,T];\dbR^n)$ (see Proposition \ref{Prop2.1} or \cite[Theorem 2.5]{Fan-Hu-Tang 2020}). By \rf{eq:1.9} and \rf{eq:1.10}, it is immediate that for ${\rm a.e.}\ t\in[0,T]$,
\bel{eq:1.11}\wt Y(t)=\dbE_t\Big[\psi(t)+\int_t^Tg(t,s,y(s),\wt Z(t,s))ds\Big]=\wt\eta(t,t)\ee
and
$$\int_t^T \wt Z(t,s)dW(s)=\int_t^T \wt\z(t,s)dW(s),$$
which implies
\bel{eq:1.12}\wt Z(t,s)=\wt\z(t,s),\ \ {\rm a.e.}\ s\in (t,T].\ee
It then follows from \rf{eq:1.10} that for $\ae\;t\in[0,T]$,
$$\wt\eta(t,r)=\psi(t)+\int_r^T  g(t,s,y(s),\wt\z(t,s))ds-\int_r^T \wt\z(t,s)dW(s),\ \  r\in[t,T],$$
which means that $(\wt\eta(t,\cd),\wt\z(t,\cd))$ is an adapted solution of BSDE \rf{eq:1.7} in $L^\infty_\dbF(\O;C([t,T];\dbR^n))\times \overline{{\rm BMO}}([t,T];$ $\dbR^n)$, and then
\bel{eq:1.13}\wt\eta(t,\cd)=\eta(t,\cd),\qq\wt\z(t,\cd)=\z(t,\cd).\ee
Combining \rf{eq:1.11}, \rf{eq:1.12} and \rf{eq:1.13} yields that
$$\wt Y(t)=\eta(t,t),\qq\wt Z(t,s)=\z(t,s),\qq\ae\;(t,s)\in\D^*[0,T].$$
Consequently, BSVIE \eqref{eq:1.4} admits a unique adapted solution $(Y(\cd),Z(\cd\,,\cd))\in L^\infty_\dbF(\O;C([t,T];\dbR^n))\times\overline{{\rm BMO}}$ $(\D^*[0,T];\dbR^n)$.

\ms

We now turn to the proof of the second part of Proposition \ref{pro:3.1}. First of all, for any given $k\ges 1$, by Young's inequality that
\bel{Young-inequality-2} k n^{1+\d} |\z(t,s)|^{1+\d}\les{\frac {1} {\wt M}}|\z(t,s)|^2+K,\ \ s\in[t,T].\ee
where $K$ is a constant depending on $\wt M$, $\d$ and $k$,  and
$$\ba{ll}
\ns\ds \wt M\=2\sup_{1\les i\les n}\esssup_{t\in[0,T]}\esssup_{\t}\Big\|\dbE_\t\int_\t^T
|\zeta^i(t,s)|^2ds \Big\|_{\infty}<\infty,
\ea
$$
Define process $\wt \zeta^i(t,\cd)$ as $\frac{1}{\sqrt{\wt M}}\zeta^i(t,\cd)$ on $[t,T]$, we see that
$\big\|\wt\zeta^i(t,\cd)\big\|_{\overline{\rm BMO}[t,T]}\les \frac 1 2.$ By John-Nirenberg inequality, we see that
$$\ba{ll}
\ns\ds \dbE_t\Big[\exp\Big(\frac {1}{\wt M}\int_t^T|\zeta^i(t,s)|^2ds \Big)\Big]\les 2.
\ea
$$
Therefore, for any $k\ges1$,
$$\ba{ll}
\ns\ds \dbE_t \Big[\exp\Big(k\int_t^T|\zeta(t,s)|^{1+\d}ds \Big)\Big]\les
 \dbE_t\Big[\exp\Big(k n^\d \sum_{i=1}^n\int_t^T|\zeta^i(t,s)|^{1+\d}ds \Big)\Big]\\
\ns\ds\ \  = \dbE_t\Big[\prod_{i=1}^n \exp\Big(k n^\d \int_t^T|\zeta^i(t,s)|^{1+\d}ds \Big)\Big]
\les \prod_{i=1}^n \Big\{\dbE_t\Big[\exp\Big(k n^{1+\d} \int_t^T|\zeta^i(t,s)|^{1+\d}ds\Big)\Big] \Big\}^{\frac 1 n}\\
\ns\ds\ \ \les \prod_{i=1}^n \Big\{\dbE_t\Big[\exp\Big(\int_t^T\frac{|\zeta^i(t,s)|^{2}}{\wt M}ds \Big)\Big]\Big\}^{\frac 1 n}e^K \les 2e^K<\infty.
\ea
$$
This together with the integrability  of $\a(t,\cd)$, $y(\cd)$ indicates that
$\a_0(t,\cd)\in \bigcap_{\mu\ges 1}\cE^{\mu,1}_\dbF(\Omega;L^1(t,T;\dbR_+))$, where
$$\a_0(t,s)\=\a(t,s)+\b\big(|y(s)|+|\z(t,s)|^{1+\d}\big),\vspace{0.1cm}$$
Taking into account (H3.1), together with BSDE \rf{eq:1.7}, by virtue of Lemma \ref{Lemma2.5} (or \cite[Proposition 2.1]{Fan-Hu-Tang 2019}),
we have for each $i=1,2,\cds,n$, $\ae\;t\in[0,T]$ and each $r\in[t,T]$, noting that $\psi(\cd)$ is bounded,
\bel{eq:1.14}\ba{ll}
\ss\ds \exp\(\g|\eta^i(t,r)|\)+\dbE_r\Big[\int_r^T |\z^i (t,s)|^2 ds\Big] \\
\ \ \ds\les K \dbE_r\Big[\exp\Big(K|\psi^i(t)|+K\int_r^T \big[\a(t,s)+\b\big(|y(s)|+|\z(t,s)|^{1+\d}\big)\big]ds\Big)\Big]\\
\ \ \ds\les K\exp\(K\int_r^T\|y(s)\|_\infty ds\)\dbE_r\Big[\exp\(K \int_r^T |\z (t,s)|^{1+\d} ds\)\Big].\ea\ee
Note that on the right-hand side, the term $\int_r^T|\z(t,s)|^{1+\d}ds$ appears in the exponent, which cannot be directly controlled by the term $\int_r^T|\z(t,s)|^2ds$. We now make use of the diagonally strictly quadratic condition \rf{growth1}. Let $i=1,\cds,n$, and $m\ges1$, For almost every $t\in[0,T]$ and each $r\in [t,T]$, we define the following $\dbF$-stopping time
$$\si_m^{t,r,i}=\inf\Big\{\t\in [r,T]\bigm|\int_r^\t|\z^i(t,s)|^2 ds\ges m\Big\}\land T,$$
with the convention $\inf\varnothing=+\infty$, and define
$$X_m^{t,r,i}=\eta^i(t,r)-\eta^i(t,\si_m^{t,r,i})+\int_r^{\si_m^{t,r,i}}  \big[\a(t,s)+\b\big(|y(s)|+|\z(t,s)|^{1+\d}\big)\big]ds.$$
Then it follows from the first inequality in \rf{growth1} and BSDE \rf{eq:1.7} that
\bel{eq:1.18}\ba{ll}
\ds\g\int_r^{\si^{t,r,i}_m}\3n\2n|\z^i(t,s)|^2ds\les\int_r^{\si_m^{t,r,i}}
\3n\big[\a(t,s)+\b\big(|y(s)|+|\z(t,s)|^{1+\d}\big)
+g^i(t,s,y(s),\z(t,s))\big]ds\\
\ss\ds\ \qq\qq\qq\qq=X_m^{t,r,i}+\int_r^{\si_m^{t,r,i}}\z^i(t,s)dW(s),\ \ r\in[t,T].\ea\ee
This yields that for each $\e>0$,
\bel{eq:1.19}\ba{ll}
\ds\exp\Big(\g\e\int_r^{\si_m^{t,r,i}} |\z^i(t,s)|^2 ds\Big)\les\exp\Big(\e X_m^{t,r,i}\Big)\cd \exp\Big({3\over2}\e^2\int_r^{\si_m^{t,r,i}} |\z^i(t,s)|^2ds\Big)\\
\hspace{2cm} \ds\qq\cdot \exp\Big(\e\int_r^{\si_m^{t,r,i}}\z^i(t,s)dW(s)-{3\over2}\e^2
\int_r^{\si_m^{t,r,i}}|\z^i(t,s)|^2ds\Big),\ \  r\in[t,T].\ea\ee
Observe that the process
$$H_m^{t,r,i}(\t):=\exp\Big(3\e\int_r^{\t\land\si_m^{t,r,i}} \z^i(t,s)dW(s)-{9\over2}\e^2\int_r^{\t\land\si_m^{t,r,i}} |\z^i(t,s)|^2ds\Big)$$
is a positive martingale with $H_m^{t,r,i}(r)=1$. Taking conditional expectation on both sides of \eqref{eq:1.19} and applying H\"{o}lder's inequality, we obtain that for each $r\in [t,T]$,
$$\ba{ll}
\ss\ds\dbE_r\Big[\exp\Big(\g\e\int_r^{\si_m^{t,r,i}} \3n\3n|\z^i(t,s)|^2 ds\Big)\Big]\1n\les\1n\Big(\dbE_r\Big[\exp\Big(3\e X_m^{t,r,i}\Big)\Big]\Big)^{1\over3}\Big(\dbE_r \Big[\exp\Big({9\over2}\e^2\int_r^{\si_m^{t,r,i}}\3n\3n|\z^i(t,s)|^2 ds\Big)\Big]\Big)^{1\over 3}.\ea$$
Consequently, for $\e\les{2\g\over9}$, we have
$$\Big\{\dbE_r\Big[\exp\Big(\g\e\int_r^{\si_m^{t,r,i}} |\z^i(t,s)|^2 ds\Big)\Big]\Big\}^{2\over3}\les\Big\{\dbE_r \Big[\exp\Big(3\e X_m^{t,r,i}\Big)\Big]\Big\}^{1\over3},\ \  r\in [t,T].$$
In view of the definitions of $\si_m^{t,r,i}$ and $X_m^{t,r,i}$, sending $m\to\infty$ in the above inequality yields that for each $\e\in(0,{2\g\over9}]$ and almost every $t\in[0,T]$,
\bel{eq:1.20}\ba{ll}
\ds\Big\{\dbE_r\Big[\exp\Big(\g\e\int_r^T|\z^i(t,s)|^2 ds\Big)\Big]\Big\}^2
\les \lim\limits_{m\rightarrow \infty}\dbE_r \Big[\exp\Big(3\e X_m^{t,r,i}\Big)\Big]\\
\ \ \ds=\dbE_r\Big\{\exp\Big[3\e \Big(\eta^i(t,r)-\psi^i(t)+\int_r^T\big[\a(t,s)+\b\big(|y(s)|
+|\z(t,s)|^{1+\d}\big)\big]ds\Big)\Big]\Big\}\\
\ \ \ds\les K\exp\Big\{3\e\|\eta(t,r)\|_\infty+K\int_r^T\|y(s)
\|_\infty ds \Big\}\dbE_r\Big[\exp\Big(3\e\b\int_r^T|\z(t,s)|^{1+\d} ds\Big)\Big],\ \ r\in[t,T].
\ea\ee
Thus, it follows from H\"{o}lder's inequality that for almost every $t\in[0,T]$, with $\e_0:={2\g\over9}$,
\bel{eq:1.21}\ba{ll}
\ds\Big\{\dbE_r \Big[\exp\Big({\g\e_0\over n}\int_r^T|\z(t,s)|^2 ds\Big)\Big]\Big\}^2=\Big\{\dbE_r\Big[\exp\({\g\e_0\over n}\sum_{i=1}^n\int_r^T
|\z^i(t,s)|^2ds\)\Big]\Big\}^2\\
\ \ \ds=\Big\{\dbE_r\Big[\prod_{i=1}^n\exp\({\g\e_0\over n}
\int_r^T
|\z^i(t,s)|^2ds\)\Big]\Big\}^2\\
\ \ \ds\les\Big\{\prod_{i=1}^n\Big\{\dbE_r\Big[\exp
\({\g\e_0\over n}\int_r^T
|\z^i(t,s)|^2 ds\)\Big]^n\Big\}^{1\over n}\Big\}^2\\
\ \ \ds=\Big\{\prod_{i=1}^n\Big\{\dbE_r\Big[\exp
\(\g\e_0\int_r^T
|\z^i(t,s)|^2 ds\)\Big]\Big\}^{1\over n}\Big\}^2\\
\ \ \ds\les\1n K\1n\exp\1n\Big[3\e_0\|\eta(t,r)\|_\infty\2n+\2n
K\int_r^T\3n\|y(s)
\|_\infty ds \Big]\Big\{\dbE_r\Big[\exp\Big(3\e_0\b\1n\int_r^T\3n|\z(t,s)|^{1+\d} ds\Big)\Big]\Big\},\ \ r\in[t,T].\ea\ee
By Young's inequality that
\bel{eq:1.22}3\e_0\b|\z(t,s)|^{1+\d}\les{\g\e_0\over n} |\z(t,s)|^2+K,\ \ s\in[t,T].\ee
Coming back to \eqref{eq:1.21}, by \eqref{eq:1.22} we deduce that for $\ae\,t\in[0,T]$,
\bel{eq:1.21*}\ba{ll}
\ds\Big\{\dbE_r \Big[\exp\Big({\g\e_0\over n}\int_r^T|\z(t,s)|^2 ds\Big)\Big]\Big\}^2\\
\ \ \ds\les K\exp\Big(3\e_0\|\eta(t,r)\|_\infty\2n+\1n K
\int_r^T\3n\|y(s)
\|_\infty ds \Big)\dbE_r\Big[\exp\Big({\g\e_0\over n} \int_r^T|\z(t,s)|^2ds\Big)\Big],\ \  r\in[t,T],\ea\ee
which leads to
\bel{eq:1.23}\dbE_r\Big[\exp\Big({\g\e_0\over n}\int_r^T |\z(t,s)|^2ds\Big)\Big]\les K\exp\(3\e_0\|\eta(t,r)\|_\infty+K\int_r^T\|y(s)\|_\infty ds \),\ \ r\in[t,T].\ee
Therefore, for any $\n>0$, noting (by Young's inequality again)
$$\n|\z(t,s)|^{1+\d}\les{\g\e_0\over n^2}|\z(t,s)|^2+K,$$
by H\"{o}lder's inequality and \eqref{eq:1.23} we have
\bel{|z|^1+d}\ba{ll}
\ds\dbE_r\Big[\exp\(\n\int_r^T|\z(t,s)|^{1+\d}ds\)\Big]
\les K\dbE_r\Big[\exp\Big({\g\e_0\over n^2}\int_r^T |\z(t,s)|^2ds\Big)\Big]\\
\ \ \ds\les K\Big\{\dbE_r\Big[\exp\Big({\g\e_0\over n}\int_r^T |\z(t,s)|^2ds\Big)\Big]\Big\}^{1\over n}\\
\ \ \ds\les K\exp\Big[{3\e_0\over n}\|\eta(t,r)\|_\infty+K\int_r^T\|y(s)\|_\infty ds  \Big],\ \  r\in[t,T].\ea\ee
Hence, it follows from \rf{eq:1.14} and \rf{|z|^1+d} that for $i=1,\cdots,n$ and $\ae\,t\in[0,T]$,
\bel{|eta|+|z|}\ba{ll}
\ds \exp\(\g|\eta^i(t,r)|\)+\dbE_r\Big[\int_r^T |\z^i (t,s)|^2 ds\Big] \\
\ \ \ds\les K\exp\(K\int_r^T\|y(s)\|_\infty ds\)\dbE_r\Big[\exp\(K\int_r^T |\z (t,s)|^{1+\d} ds\)\Big]\\
\ \ \ds\les K\exp\(K\int_r^T\|y(s)\|_\infty ds\)\exp\Big[{3\e_0\over n}\|\eta(t,r)\|_\infty+K\int_r^T\|y(s)\|_\infty ds \Big]\\
\ \ \ds\les K\exp\({3\e_0\over  n}\|\eta(t,r)\|_\infty+K\int_r^T\|y(s)\|_\infty ds\), \ \  r\in [t,T].\ea\ee
Note that  $3\e_0={2\g\over3}<\g$. The above leads to
\bel{eq:3.22-1}
\|\eta(t,r)\|_\infty\les K_0\(1+\int_r^T \|y(s)\|_\infty ds\),\ \  r\in [t,T],\ee
for some absolute constant $K_0>0$, independent of $y(\cd)$, in particular. Also, by enlarging $K_0>0$ if necessary, we have that (see \rf{|eta|+|z|})
\bel{|zeta|}\dbE_r\(\int_r^T |\z^i(t,s)|^2ds\)\les K_0\exp\(2K_0\int_r^T\|y(s)\|_\infty ds\), \ \  r\in [t,T].\ee
Now, if
$$\|y(s)\|_\infty\les K_0e^{K_0(T-s)},\q \ae\;s\in[0,T],$$
then it follows from \rf{eq:3.22-1} that for a.e. $t\in[0,T]$,
\bel{eq:1.16}
\|\eta(t,r)\|_\infty\les K_0+K_0\int_r^T K_0e^{K_0(T-s)}ds=K_0 e^{K_0(T-r)},\ \ r\in[t,T].\ee
Then, \rf{eq:1.5} follows by letting $r=t$ in \rf{eq:1.16}. Furthermore, in view of \rf{|zeta|}, one has that for $\ae\,t\in[0,T]$,
$$\big\|\z(t,\cd)\big\|^2_{\overline{{\rm BMO}}[t,T]}\les nK_0 \exp\(2K_0e^{K_0(T-t)}\).$$
Hence, \rf{eq:1.6} holds since
$$\big\|Z(\cd\,,\cd)\big\|_{\overline{{\rm BMO}}( \D^*[0,T])}= \esssup_{t\in[0,T]}\big\|\z (t,\cd)\big\|_
{\overline{{\rm BMO}}[t,T]}\les\sqrt{nK_0}\exp\Big(K_0e^{K_0T}\Big).$$
The proof of Proposition \ref{pro:3.1} is then complete.
\vspace{0.2cm}
\endpf

We are now ready to state and prove the main result of this section.

\bt{th:3.1} \sl Let {\rm(H3.1)} hold. Then for any $\psi(\cd)\in L_{ \cF_T }^\infty(0,T;\dbR^n)$, BSVIE \rf{BSVIE1} admits a unique adapted solution $(Y(\cd),Z(\cd\,,\cd))\in L^\infty_\dbF(0,T;\dbR^n)\times \overline{{\rm BMO}}(\D^*[0,T];\dbR^n)$.
\et

\it Proof. \rm Let $K_0,\bar K_0>0$ be the constants in the above proposition. Let
$$\ba{ll}
\sB\deq\Big\{(U(\cd),V(\cd\,,\cd))\in L_\dbF^\infty(0,T;\dbR^n)\times\overline{{\rm BMO}}(\D^*[0,T];\dbR^n)\bigm|\\
\ss\ds\qq\qq\qq\|U(t)\|_\infty\les K_0e^{K_0(T-t)} \ \ {\rm a.e.}\ t\in[0,T]\ \ {\rm and}\ \ \|V(\cd\,,\cd)\|_{\overline{{\rm BMO}} (\D^*[0,T]) }\les\bar K_0\Big\}\ea$$
which is a convex closed set in the Banach space $L_\dbF^\infty (0,T;\dbR^n)\times \overline{{\rm BMO}}(\D^*[0,T];\dbR^n)$. By Proposition \ref{pro:3.1} we know that for each $(y(\cd),z(\cd\,,\cd))\in \sB$, the following BSVIE
\bel{eq:1.26}\ba{ll}
\ss\ds Y(t)=\psi(t)+\int_t^T  g(t,s,y(s),Z(t,s))ds-\int_t^T Z(t,s)dW(s),\ \  t\in[0,T]\ea\ee
admits a unique adapted solution $(Y(\cd),Z(\cd\,,\cd))\in\sB$. That is to say, the map
$$\G\big(y(\cd),z(\cd\,,\cd)\big)\deq\big(Y(\cd),Z(\cd\,,\cd)\big),\q \big(y(\cd),z(\cd\,,\cd)\big)\in\sB$$
is well-defined and maps from $\sB$ to itself. In order to prove the desired result, it suffices to prove that for some $k>0$, the map $\G$ is contract in $\sB$ with the following equivalent norm:
\bel{equivalent-norm}
\Big\|\big(U(\cd),V(\cd\,,\cd)\big)\Big\|_{\sB_k}\deq\sqrt{\esssup_{t\in[0,T]}\Big(\Big\|e^{{k t\over 2}}U(t)\Big\|^2_\infty+\Big\|e^{{k\cd\over2}}V(t,\cd) \Big\|^2_{\overline{{\rm BMO}} [t,T] }\Big)}.\ee
Now, for any $(y(\cd),z(\cd\,,\cd)), (\wt y(\cd),\wt z(\cd\,,\cd))\in\sB$, set
$$(Y(\cd),Z(\cd\,,\cd))=\G(y(\cd),z(\cd\,,\cd)),\qq
(\wt Y(\cd),\wt Z(\cd\,,\cd))=\G(\wt y(\cd),\wt z(\cd\,,\cd)).$$
It follows from Proposition \ref{pro:3.1} that
\bel{eq:1.27}Y(t)=\eta(t,t),\ Z(t,s)=\z(t,s),\ \wt Y(t)=\wt\eta(t,t),\ \wt Z(t,s)=\wt\z(t,s),\ \ (t,s)\in\D^*[0,T],
\ee
where for almost every $t\in[0,T]$, $(\eta(t,\cdot), \z(t,\cdot))$ and $(\wt\eta(t,\cdot), \wt\z(t,\cdot))$ are respectively the unique solutions of the following BSDEs in $L_\dbF^\infty (t,T;\dbR^n)\times \overline{{\rm BMO}}([t,T];\dbR^n)$,
$$\left\{\2n\ba{ll}
\ds\eta(t,r)=\psi(t)+\int_r^T  g(t,s,y(s),\z(t,s))ds-\int_r^T \z(t,s)dW(s),\ \ r\in [t,T],\\
\ss\ds \wt \eta (t,r)=\psi(t)+\int_r^T  g(t,s,\wt y(s),\wt\z(t,s))ds-
\int_r^T \wt\z(t,s)dW(s),\ \ r\in [t,T].
\ea\right.$$
By the definition of $\sB$ and \rf{eq:1.27}, for almost every $t\in[0,T]$, we have
\bel{eq:1.28}
\phi(|y(t)|)\les\phi(|y(t)|\vee|\wt y(t)|)\les \phi\big(K_0e^{K_0(T-t)}\big)\les\phi\big(K_0e^{K_0T}\big)\ee
and
\bel{eq:1.29}
\big\|\z(t,\cd)\big\|_{\overline{\rm BMO} [t,T] },\big\|\wt\z(t,\cd)\big\|_{\overline{\rm BMO} [t,T] }\les\bar K_0.\ee
The rest of the proof is divided into three steps.\vspace{0.2cm}

{\bf Step 1}. {\it Girsanov's transform}.\vspace{0.2cm}

For $(t,s)\in\D^*[0,T]$, denote
$$\ba{ll}
\ss\ds\D y(s)=y(s)-\wt y(s),\q
\D\eta(t,s)=\eta(t,s)-\wt\eta(t,s),\q\D\z(t,s)=\z(t,s)-\wt\z(t,s),\\
\ss\ds\D\eta^i(t,s)\1n=\1n\eta^i(t,s)\1n-\1n\wt\eta^i(t,s),\q\D\z^i(t,s)
\1n=\1n\z^i(t,s)\1n-\1n\wt\z^i(t,s),\q 1\les i\les n.\ea$$
Let
$$\h\z_i(t,s)=(\z^1(t,s),\cds,\z^{i-1}(t,s),\wt\z^i(t,s),\z^{i+1}(t,s),\cds,
\z^n(t,s))^\top,\q1\les i\les n.$$
Then
$$\ba{ll}
\ss\ds \z(t,s)-\h\z_i(t,s)=(0,\cds,0,\D\z^i(t,s),0,\cds,0)^\top,\\
\ss\ds\h\z_i(t,s)-\wt\z(t,s)=(\D\z^1(t,s),\cds,\D\z^{i-1}(t,s),0,\D\z^{i+1}(t,s),\cds,\D\z^n(t,s))^\top.\ea$$
%
Observe that
$$\ba{ll}
\ss\ds g^i(t,s,y(s),\z(t,s))-g^i(t,s,\wt y(s),\wt\z(t,s))\\
\ \ \ds=g^i(t,s,y(s),\z(t,s))-g^i(t,s,y(s),\h\z_i(t,s))+g^i(t,s,y(s),\h\z_i(t,s))-g^i(t,s,\wt y(s),\wt\z(t,s))\\
\ \ \ds=  g^i_z(t,s)\D\z^i(t,s)+\cG^i(t.s),\ea$$
where
$$\ba{ll}
%
%
\ss\ds g^i_z(t,s)\={ g^i(t,s,y(s),\z(t,s))-g^i(t,s,y(s),\h\z_i(t,s))
\over\D\z^i(t,s)}{\bf1}_{\{\D\z^i(t,s)\ne0\}},\\
\ss\ds\cG^i(t,s)\=g^i(t,s,y(s),\h\z_i(t,s))-g^i(t,s,\wt y(s),\wt\z(t,s)).\ea$$
By \rf{g-g} and \rf{eq:1.28}, we have
\bel{eq:1.32}\ba{ll}
%
%
\ss\ds|g^i_z(t,s)|\les \phi\big(K_0e^{K_0T}\big)\(1+|\z(t,s)|
+|\h\z_i(t,s)|\) ,\qq j\ne i,\\
\ss\ds|\cG^i(t,s)|\les\phi\big(K_0e^{K_0T}\big)\Big[\(1+|\h\z_i(t,s)|+|\wt\z(t,s)|\)|\D y(s)|\\
\ss\ds\qq\qq\qq\qq\qq+\sum_{j\ne i}\(1+|\h\z_i(t,s)|^\d+|\wt\z(t,s)|^\d\)|\D\z^j(t,s)|\Big].\ea\ee
According to the above notation, we have
\bel{eq:1.30}\ba{ll}
\ss\ds\D\eta^i(t,r)=\2n\int_r^T\3n \(  g^i_z(t,s) \D\z^i(t,s)+\cG^i(t,s)\)ds
-\2n\int_r^T\3n\D\z^i(t,s)dW(s),\ \  r\in[t,T].\ea\ee
Extend $g^i_z(t,s)$ to be zero for $s\in[0,t)$ and let
$$N_i(t,r):=\int_0^rg^i_z(t,s) dW(s),\ \  r\in[0,T].$$
In view of \rf{eq:1.29} and \rf{eq:1.32}, we know that $N_i(t,\cd)$ is a BMO martingale under $\dbP$, and that
\bel{eq:1.33}\|N_i(t,\cd)\|_{{\rm BMO}_{\dbP}[0,T]}\les K\Big(1+2\big\|\z(t,\cd)\big\|_{\overline{{\rm BMO}}[t,T]}+\big\|\wt\z(t,\cd)\big\|_{\overline{{\rm BMO}}[t,T]}\Big)\les K.\ee
Then, by \cite[Theorem 2.3]{{Kazamaki 1994}} we know that the stochastic exponential $\sE[N_i(t,\cdot)](\cdot)$ of $N_i(t,\cdot)$, given by
$$\sE[N_i(t,\cdot)](r)=\exp\Big\{\int_0^r g^i_z(t,s) dW(s)-{1\over2} \int_0^r| g^i_z(t,s) |^2ds\Big\},\ \ r\in [0,T],$$
is a uniformly integrable martingale under $\dbP$.\vspace{0.2cm}

Now, we define the probability measure $\dbP^{t,i}$ on $(\Omega,\cF_T)$ by the following:
$$d \dbP^{t,i}= \sE[N_i(t,\cdot)](T) d\dbP.$$
Note that
$$M_i(t,r)=\int_0^r \D\z^i(t,s)dW(s),\ \  r\in[0,T]$$
is also a BMO martingale under $\dbP$. It follows from \cite[Theorem 3.6]{{Kazamaki 1994}} that the process
\bel{eq:1-33}
\wt W_i(t,r)= W(r)-\int_0^r g^i_z(t,s) ds,\ \  r\in[0,T]
\ee
is a standard Brownian motion under the probability measure $\dbP^{t,i}$, and that the process
$$
\wt M_i(t,r)=\int_0^r \D\z^i(t,s)d\wt W_i(t,s),\ \ r\in[0,T]
$$
the Girsanov's transform of $M_i(t,\cdot)$, is a BMO martingale under the probability measure $\dbP^{t,i}$.  Moreover, by  \eqref{eq:1.30} and \eqref{eq:1-33} we know that for $i=1,\cdots,n$,
\bel{eq:1.34}\D\eta^i(t,r)+\int_r^T\D\z^i(t,s)d\wt W_i(t,s)=\int_r^T\cG^i(t,s)ds,\ \  r\in [t,T].\vspace{0.2cm}\ee

{\bf Step 2}. {\it Basic estimate.}\vspace{0.2cm}

 Fix $i=1,\cds,n$, $0\les t\les r\les T$ and $k>0$. Let $\dbE^{t,i}_r[\xi]$ denote the conditional expectation of the random variable $\xi$ with respect to $\cF_r$ under the probability measure $\dbP^{t,i}$. Taking square and $\dbE^{t,i}_r[\cdot]$ in both sides of \eqref{eq:1.34} yields that
\bel{eq:1.35}|\D\eta^i(t,r)|^2+\dbE^{t,i}_r\(\int_r^T
|\D\z^i(t,s)|^2 ds\)=\dbE^{t,i}_r\(\Big|\int_r^T\cG^i(t,s)ds\Big|^2\).\ee
which implies that
\bel{eq:1.35}e^{kr}|\D\eta^i(t,r)|^2+\dbE^{t,i}_r\Big(\int_r^T e^{kr}|\D\z^i(t,s)|^2 ds\Big)= \dbE^{t,i}_r\Big[e^{kr}\Big(\int_r^T |\cG^i(t,s)| ds\Big)^2\Big].\ee
Note that
$$\int_r^Tke^{ks}\Big(\int_s^T|\D\z^i(t,u)|^2du\Big)ds
=\Big(e^{ks}\int_s^T |\D\z^i(t,u)|^2du\Big)\Big|^T_r+\int_r^T e^{k s}|\D\z^i(t,s)|^2ds.$$
As a result,
\bel{eq:1.36}\int_r^T e^{ks}|\D\z^i(t,s)|^2ds=\int_r^Te^{k r}|\D\z^i(t,s)|^2ds+k\int_r^T\Big(\int_s^T e^{k s}|\D\z^i(t,u)|^2 du\Big)ds.\ee
Combining \rf{eq:1.35} and \rf{eq:1.36}, we deduce that
\bel{eq:1.37}\ba{ll}
\ds e^{k r}|\D\eta^i(t,r)|^2+\dbE^{t,i}_r\Big(\int_r^T e^{k s}|\D\z^i(t,s)|^2 ds\Big)\\
\ \ \ds\les2\dbE^{t,i}_r\Big[e^{kr}\Big(\int_r^T \big|\cG^i(t,s)\big| ds\Big)^2\Big]+k \dbE^{t,i}_r\Big\{\int_r^T\dbE^{t,i}_s\Big[e^{k s}
\Big(\int_s^T \big|\cG^i(t,u)\big| du\Big)^2\Big]ds\Big\}\\
\ \ \ds\les2\dbE^{t,i}_r\Big[e^{k r}\Big(\int_r^T \big|\cG^i(t,s)\big| ds\Big)^2\Big]+\dbE^{t,i}_r\Big[\int_r^T k e^{k s}\Big(\int_s^T \big|\cG^i(t,u)\big| du\Big)^2ds\Big].\ea\ee
}
Furthermore, observe by \rf{eq:1.32} and H\"{o}lder's inequality that
$$\ba{ll}
\ds\Big(\int_r^T|\cG^i(t,s)|ds\Big)^2\les K\Big\{\int_r^T\Big[\(1+|\z(t,s)|+|\wt\z(t,s)|\)|\D y(s)|\\
\ds\qq\qq\qq\qq\qq\qq\qq\qq+\sum_{j\ne i}\(1+|\z(t,s)|^\d+|\wt\z(t,s)|^\d\)|\D\z^j(t,s)|\Big]ds\Big\}^2\\
\ds\les K\Big\{\Big[\int_r^Te^{-{k\over2}s}e^{{k\over2}s} \(1+|\z(t,s)|+|\wt\z(t,s)|\)|\D y(s)|ds\Big]^2\\
\ss\ds\qq\qq\qq+\sum_{j\ne i}\Big[\int_r^Te^{-{k\over2}s}e^{{k\over2}s}\(1+|\z(t,s)|^\d+|\wt \z(t,s)|^\d\)|\D\z^j(t,s)|ds\Big]^2\Big\}\\
\ds\les K\Big[\(\int_r^Te^{-ks}ds\)\(\int_r^T e^{ks}\big(1+|\z(t,s)|+|\wt\z(t,s)|\big)^2
|\D y(s)|^2ds\)\\
\ss\ds\qq\qq+\sum_{j\ne i}\(\int_r^Te^{-ks}\(1+|\z(t,s)|^\d+|\wt\z(t,s)|^\d\)^2ds\)
\(\int_r^Te^{ks}|\D\z^j(t,s)|^2ds\)\Big]\\
\ds\les{Ke^{-kr}\over k}\int_r^T e^{ks}\vartheta(t,s)|\D y(s)|^2ds\\
\ss\ds\qq\qq+K\(\int_r^Te^{-{k\over1-\d}s}ds\)^{1-\d}\(\int_r^T|\vartheta(t,s)ds\)^\d
\(\int_r^Te^{ks}|\D\z(t,s)|^2ds\)\\
\ds\les{Ke^{-kr}\over k}\int_r^T e^{ks}\vartheta(t,s)|\D y(s)|^2ds+{Ke^{-kr}\over k^{1-\d}}\(\int_r^T\vartheta(t,s)ds\)^\d
\(\int_r^Te^{ks}|\D\z(t,s)|^2ds\),\ea$$
where $\vartheta(t,s):=1+|\z(t,s)|^2+|\wt\z(t,s)|^2,\ s\in[t,T]$.  Then,
\bel{eq:1.41}
\ds e^{kr}\Big(\int_r^T\big| \cG^i(t,s)\big|ds\Big)^2\ds\les{K\over k}\int_r^T e^{ks}\vartheta(t,s)|\D y(s)|^2ds+{K\over k^{1-\d}}\(1+\int_r^T \vartheta(t,s) ds\)\int_r^Te^{ks}|\D\z(t,s)|^2ds.\ee
Thus, we also have that for each $s\in [t,T]$,
$$
\ds e^{ks\over2}\Big(\int_s^T \big| \cG^i(t,u)\big| du\Big)^2\ds\les{K\over k}\int_s^Te^{ku\over 2}\vartheta(t,u)|\D y(u)|^2 du+{K\over k^{1-\d}}\(1+\int_s^T\vartheta(t,u) du\)\int_s^Te^{ku\over2}|\D\z(t,u)|^2du.$$
It then follows from the above inequality that
\bel{eq:1.43}
\ba{ll}
\ds \int_r^T ke^{ks}\Big(\int_s^T \big| \cG^i(t,u)\big| du\Big)^2 ds\les K\int_r^T \Big(e^{ks\over 2}\int_s^T e^{ku\over 2}\vartheta(t,u)|\D y(u)|^2 du\Big) ds\\
\ds\qq\qq\qq +{K\over k^{1-\d}}\(1+\int_r^T \vartheta(t,s)ds\)\int_r^T\(ke^{ks\over 2} \int_s^T e^{k u\over 2}|\D\z(t,u)|^2 du\) ds.\\
\ \ \ds={2K\over k}\(e^{ks\over 2}\int_s^T e^{ku\over 2}\vartheta(t,u)|\D y(u)|^2 du\Big|^T_r+\int_r^T e^{ks}\vartheta(t,s)|\D y(s)|^2ds\)\\
\ds\qq+{K\over k^{1-\d}}\(1+\int_r^T \vartheta(t,s)ds\)\(2e^{ks\over2}\int_s^Te^{ku\over2}|\D\z(t,u)|^2 du\Big|^T_r+2\int_r^T e^{ks}|\D\z(t,s)|^2 ds\)\\
\ \ \ds\les {K\over k}\int_r^T e^{ks}\vartheta(t,s)|\D y(s)|^2ds+{K\over k^{1-\d}}\(1+\int_r^T \vartheta(t,s)ds\)\int_r^T e^{ks}|\D\z(t,s)|^2ds.
\ea\ee
Now, we substitute \eqref{eq:1.41} and \eqref{eq:1.43} into \eqref{eq:1.37} to obtain that
\bel{eq:1.45}\ba{ll}
\ds e^{kr}|\D\eta^i(t,r)|^2+\dbE^{t,i}_r\Big[\int_r^T e^{k s}|\D\z^i(t,s)|^2 ds\Big]\\
\ \ \ds\les{K\over k}\dbE^{t,i}_r\Big[\int_r^T e^{ks}\vartheta(t,s)|\D y(s)|^2 ds\Big]+{K\over k^{1-\d}}\dbE^{t,i}_r\Big[\Big(1+\int_r^T \vartheta(t,s)ds\Big)\int_r^T e^{ks}|\D\z(t,s)|^2 ds\Big].\vspace{0.2cm}
\ea\ee

{\bf Step 3}. {\it Contraction mapping}.

\ms

 By the definitions of $\vartheta(\cd\,,\cd)$, $\dbE^{t,i}[\,\cd\,]$ and $\dbP^{t,i}$ together with \rf{Equivalent-norms-BMO}, \rf{eq:1.29}, \rf{eq:1.33} and the energy inequality, it follows that for each $i=1,\cds,n$, each $t\in[0,T]$ and each $\t\in\sT[t,T]$,
\bel{eq:1.46-1}
\Big\|e^{k\cd\over 2}\D\z^i(t,\cd) \Big\|^2_{\overline{{\rm BMO}}[t,T]}\les K \sup_{\t\in\sT[t,T]}\left\|\dbE^{t,i}_\t\Big[\int_\t^T e^{k s}|\D\z^i(t,s)|^2 ds\Big]\right\|_\infty,\ee
and
$$\ba{ll}
\ss\ds\dbE^{t,i}_\t\Big[\int_\t^T\vartheta(t,s)ds\Big]\les K\(1+ \big\|\z(t,\cd)\big\|^2_{\overline{{\rm BMO}}[t,T]}+\big\|\wt\z(t,\cd)\big\|^2_{\overline{{\rm BMO}}[t,T]}\)\les K,\\
\ss\ds\dbE^{t,i}_\t\Big[\(1+\int_\t^T\vartheta(t,s)ds\)^2\Big]
\les K\(1+\big\|\z(t,\cd)\big\|^4_{\overline{{\rm BMO}}[t,T]}+\big\|\wt\z(t,\cd\big\|^4_{\overline{{\rm BMO}}[t,T]}\)\les K,\\
\ss\ds\dbE^{t,i}_\t\Big[\(\int_\t^T e^{ks}|\D\z(t,s)|^2 ds\)^2\Big]\les K\Big\|e^{k\cd\over 2}\D\z(t,\cd) \Big\|^4_{\overline{{\rm BMO}}[t,T]},\ea$$
and then, by virtue of H\"{o}lder's inequality, we have
\bel{eq:1.46-2}
\dbE^{t,i}_\t\Big[\int_\t^T e^{ks}\vartheta(t,s)|\D y(s)|^2 ds\Big]\les K\Big\|e^{k\cd\over 2}\D y(\cd)\Big\|^2_{L^\infty_\dbF(t,T;\dbR^n)},\ee
and
\bel{eq:1.46-3}\dbE^{t,i}_\t\Big[\Big(1+\int_\t^T \vartheta(t,s)ds\Big)\int_\t^T e^{ks}|\D\z(t,s)|^2 ds\Big]\les
K\Big\| e^{k\cd\over 2}\D\z(t,\cd) \Big\|^2_{\overline{{\rm BMO}}[t,T]}.\ee
Note that \eqref{eq:1.45} holds still when $r$ is replaced with any $\t\in\sT[t,T]$. It follows from \eqref{eq:1.46-1}--\eqref{eq:1.46-3} that for almost every $t\in[0,T]$ and $i=1,\cds,n$,
$$\ds e^{k t}|\D\eta^i(t,t)|^2+\Big\|e^{k\cd\over 2}\D\z^i(t,\cd)\Big\|^2_{\overline{{\rm BMO}}[t,T]}\ds\les{K\over k}\Big\| e^{k\cd\over 2}\D y(\cd) \Big\|^2_{L_\dbF^\infty(t,T;\dbR^n)}+{K\over k^{1-\d}}\Big\| e^{k\cd\over 2}\D\z(t,\cd)\Big\|^2_{\overline{{\rm BMO}}[t,T]}.$$
Picking $k$ large enough, we obtain
$$\big\|(Y(\cd)-\wt Y(\cd), Z(\cd\,,\cd)-\wt Z(\cd\,,\cd))\big\|_{\sB_k}\les
{1\over2}\big\|(y(\cd)-\wt y(\cd),z(\cd\,,\cd)-\wt z(\cd\,,\cd))\big\|_{\sB_k}.$$
This means that $\G$ is a contraction mapping on the convex closed $\sB$ under the norm $\|\cd\|_{\sB_k}$, which is just the desired result. The proof is then complete. \endpf

\br{rmk:3.3}
\rm Let us make some comparisons with the relevant literature on quadratic growth BSVIEs.

\ms

(i) In \cite{Hernandez 2021}, multi-dimensional BSVIEs were studied when the generator is allowed to depend on the diagonal value $Z(s,s)$. Certain differentiability on the generator was assumed. In addition, the approach used relies on the interesting ideas developed in \cite{Tevzadze 2008} where the relevant data are assumed to be bounded and sufficiently small. 

\ms

(ii) In \cite{Wang-Sun-Yong 2021}, one-dimensional BSVIEs  \rf{BSVIE1} with quadratic growth were investigated. In contrast, our assumption (\ref{g-g}) in Theorem \ref{pro:3.1} is weaker than that in \cite[Theorem 3.5]{Wang-Sun-Yong 2021} even in the one-dimensional framework.
\er

\br{rmk:4.4}
\rm
One crucial step for the proof of Proposition \ref{pro:3.1} is the estimate (\ref{eq:1.45}). To establish that, we have borrowed some interesting ideas in \cite[Lemma 3.1]{Shi-Wang 2012}. In the above proof, we have imposed the element in $\sB$ with a new equivalent norm in (\ref{equivalent-norm}). This helps us bypass the using stochastic Fredholm integral equations, and simplify the arguments as in \cite{Wang-Sun-Yong 2021} and \cite{Yong 2008}. Such equivalent norm techniques also appeared in e.g. \cite{Wang-Zhang 2007}. However, we have to introduce more delicate arguments due to our quadratic growth framework.

\er

\subsection{A mixed case of linear and quadratic growth}

Note that due to the diagonally strictly quadratic condition in (H3.1), Theorem \ref{th:3.1} could not cover the classical case of $z\mapsto g(t,s,y,z)$ growing linearly. To fill this gap, in this subsection, we are going to consider a class of multi-dimensional coupled BSVIEs for which some components of the generator have linear growth $z$, and others have diagonally strictly quadratic growth.

\ms

To begin with, let us first recall the multi-dimensional BSDEs:
\bel{eq:3.48}\ba{ll}
\ss\ds Y(t)=\xi+\int_t^Tf(s,Y(s),Z(s))ds-\int_t^T Z(s)dW(s),\ \  t\in[0,T].\ea\ee
where the free term $\xi$ is an $\cF_T$-measurable $\dbR^n$-valued random variable, and the generator $f:[0,T]\times\dbR^n\times\dbR^n\times\O\to\dbR^n$ has the property that $s\mapsto f(s,y,z)$ is $\dbF$-progressively measurable for all $(y,z)\in\dbR^n\times \dbR^n$. We emphasize that in the literature, with bounded and/or some unbounded terminal states, only the following three mutually exclusive cases were studied: For the map $z\mapsto f(s,y,z)$,

\ms

$\bullet$ It has a linear growth (see the classical literature \cite{Pardoux-Peng 1990, El Karoui-Peng-Quenez 1997});

\ms

$\bullet$ It is diagonally strictly quadratic, and sub-quadratically for off-diagonal components (see \cite{Fan-Hu-Tang 2020});

\ms

$\bullet$ It is diagonally no more than quadratic growth and with bounded for off-diagonal components (see \cite{Hu-Tang 2016,Fan-Hu-Tang 2020}).

\ms

We now would like to look at the mixed case of BSVIEs by which we mean that some components of the generator $f$ have linear growth in $z$ and the other components have at most quadratic growth (not necessarily strict) in $z$ in a special way. More precisely, after a possible relabeling, we have the following assumptions on the generator $f$ of BSDE (\ref{eq:3.48}).

\ms

{\bf(H3.2)} There exist a non-negative valued process $\a(\cd)\in L^\infty_\dbF(\O;L^1(0,T;\dbR))$, an increasing deterministic function $\phi:\dbR\to\dbR_+$ and a constant $\b>0$ such that for all $(t,y,z)\in[0,T]\times\dbR^n\times\dbR^n$,
\bel{BSDE-f0}|f^i(t,y,z)|\les\a(t)+\phi(|y|)+\b\big(|z|+|z^i|^2\big),\q 1\les i\les n,\ee
and for some $k=0,1,2,\cds,n$,
\bel{BSDE-f1}\left\{\2n\ba{ll}
\ss\ds\sgn(y^i) f^i(t,y,z)\les\a(t)+\b\big(|y|+|z^i|^2\big),\qq 1\les i\les k,\\
\ss\ds\sgn(y^i) f^i(t,y,z)\les\a(t)+\b\(|y|+\sum_{j=k+1}^n|z^j|\),\qq k+1\les i\les n.\ea\right.\ee
Here, $k=0$ means that the first condition in \rf{BSDE-f1} is absent, and $k=n$ means that the second condition in \rf{BSDE-f1} is absent. In addition, there exists a $\d\in[0,1)$ such that for each $t\in[0,T]$ and $y,\bar y,z,\bar z\in\dbR^n$,
\bel{f-f}\ba{ll}
\ss\ds |f^i(t,y,z)-f^i(t,\bar y,\bar z)|\les\phi(|y|\vee|\bar y|)\Big[\big(1+|z|+|\bar z|\big)\big(|y-\bar y|+|z^i-\bar z^i|\big)\\
\ss\ds\qq\qq\qq\qq\qq\qq\qq+\big(1+|z|^\d+|\bar z|^\d\big)\sum_{j\neq i}|z^j-\bar z^j|\Big],\qq1\les i\les n.\ea\ee

 Inspired by Lemma \ref{Lemma2.5}, the following establishes an a priori estimate on the solution of BSDE \eqref{eq:3.48} under hypothesis (H3.2), and it is comparable with \cite[Lemma 4.1]{Fan-Hu-Tang 2020}.

\bp{pro:3.2} \sl Let {\rm(H3.2)} hold and $\xi\in L^\infty_{\cF_T}(\O;\dbR^n)$. Assume that for some $T_0\in [0,T)$, BSDE \eqref{eq:3.48} has an adapted solution $(Y(\cd),Z(\cd))\in L^\infty_\dbF(\O;C([T_0,T];\dbR^n))\times \overline{{\rm BMO}}([T_0,T];\dbR^n)$ on time interval $[T_0,T]$. Then, there exists an absolute  constant $K$ depending only on $(n,\beta,T,\xi,\a(\cd))$ and being independent of $T_0$ such that
$$\|Y(\cd)\|^2_{L^\infty_\dbF(\O;C[T_0,T])}+\|Z(\cd)\|^2_{\overline{\rm BMO}[T_0,T]}\les K.$$
\ep

\ms

\it Proof. \rm First of all, taking into account the first inequality in \rf{BSDE-f1}, together with BSDE \rf{eq:3.48}, by virtue of Lemma \ref{Lemma2.5} (or \cite[Proposition 2.1]{Fan-Hu-Tang 2019}), taking $\d=1$ and
$$\a_0(s)=\a(s)+\b|Y(s)|,$$
we have that for each $i=1,2,\cds,k$ and $t\in [T_0,T]$,
\bel{eq:3.57}\ba{ll}
\ns\ds\exp\Big(|Y^i(t)|\Big)+\dbE_t\(\int_t^T |Z^i(s)|^2 ds\)\les  \ds K\dbE_t\Big[\exp\(K |\xi^i|+K\int_t^T \Big(\a(s)+\b |Y(s)|\Big)ds\)\Big]\\
\hspace{5.5cm}\les \ds K\exp\(K\|\xi^i\|_\infty+K\int_t^T \|Y(\cdot)\|_{L^\infty_\dbF(\O;C[s,T])}ds\),
\ea\ee
and then
\bel{eq:3.58}
|Y^i(t)|\les K+K\int_t^T\|Y(\cd)\|_{L^\infty_\dbF(\O;C[s,T])} ds.\ee
\bel{eq:3.60}\sum_{j=1}^k\|Z^j(\cd)\|^2_{\cl{\rm BMO}[t,T]}\les K\exp\(K\int_t^T\|Y(\cd)\|_{L^\infty_\dbF(\O;C[s,T])}ds\).\ee
Furthermore, in view of the second inequality of in \rf{BSDE-f1}, using It\^o's formula, Young's inequality and H\"{o}lder's inequality, we can deduce that for each $i=k+1,\cds,n$ and $t\in [T_0,T]$,
\bel{eq:3.61}\ba{ll}
\ds|Y^i(t)|^2+\dbE_t\Big[\int_t^T|Z^i(s)|^2ds\Big]
=\dbE_t|\xi^i|^2+\dbE_t\Big[\int_t^T2Y^i(s)f^i(s,Y(s),Z(s))ds
\Big]\\
\ \ \ds\les\|\xi\|_\infty^2+{1\over6}\|Y^i(\cd)\|^2_{
L^\infty_\dbF(\O;
C[t,T])}+6\dbE_t\Big\{\Big[\int_t^T\3n\(\a(s)+\b |Y(s)|\1n+\1n\b\3n \sum_{j=k+1}^n|Z^j(s)|\)ds\Big]^2\Big\}\\
\ \ \ds\les K+{1\over6}\|Y^i(\cd)\|^2_{L^\infty_\dbF(\O;
C[t,T])}+K\int_t^T\sum_{j=1}^n|Y^j(s)|^2ds
+K\dbE_t\Big[\(\int_t^T\sum_{j=k+1}^n|Z^j(s)|ds\)^2\Big].\vspace{0.2cm}
\ea
\ee
With the last inequality \eqref{eq:3.61} in hand, using a similar computation to step 2 in the proof of Theorem \ref{th:3.1} and using H\"{o}lder's inequality we deduce that for each  $i=k+1,\cds,n$ and $t\in [T_0,T]$,
$$\ba{ll}
\ds|\bar Y^i(t)|^2+\dbE_t\Big[\int_t^T|\bar Z^i(s)|^2 ds\Big]\\
\ \ \ds\les 3K e^{\rho T}+\frac{1}{2}\|\bar Y^i(\cd)\|^2_{L^\infty_\dbF(\O;C[t,T])}\1n+\1n K\2n\int_t^T\2n\sum_{j=1}^n|\bar Y^j(s)|^2ds+{1\over2n^2} \dbE_t\Big[\int_t^T\3n\(\sum_{j=k+1}^n|\bar Z^j(s)|\)^2 ds\Big]\\
\ \ \ds\les K+\frac{1}{2}\|\bar Y^i(\cd)\|^2_{L^\infty_\dbF(\O;C[t,T])}+K\int_t^T\sum_{j=1}^n\|\bar Y^j(\cd)\|^2_{L^\infty_\dbF(\O;C[s,T])}ds
+{1\over2n}\sum_{j=k+1}^n\|\bar Z^j(\cdot)\|^2_{\cl{{\rm BMO}}[t,T]},\ea$$
where $\rho >0$ only depends on $n$, $\b$,
$$ \ \bar Y^i(s):=e^{\rho s\over 2}Y^i(s)\ \ {\rm and}\ \ \bar Z^i(s):=e^{\rho s\over 2}Z^i(s),\ \ s\in [t,T].$$
Thus, for each $i=k+1,\cdots, n$, and each $t\in [T_0,T]$, we have
$$
\ba{ll}
\ds{1\over2}\|\bar Y^i(\cd)\|^2_{ L^\infty_\dbF(\O;C[t,T])
}+\|\bar Z^i(\cd)\|^2_{\overline{{\rm BMO}}[t,T]}\\
\ \ \ds\les K+K\int_t^T\sum_{j=1}^n\|\bar Y^j(\cd)\|^2_{L^\infty_\dbF(\O;C[s,T])}ds
\ds +{1\over2n}\sum_{j=k+1}^n\|\bar Z^j(\cd)\|^2_{\overline{{\rm BMO}}[t,T]},
\ea
$$
and then
\bel{eq:3.64}
\ds \sum_{j=k+1}^n\|\bar Y^j(\cdot)\|^2_{L^{\infty}_{\dbF}(\Omega;C[t,T])}+\sum_{j=k+1}^n\|\bar Z^j(\cdot)\|^2_{\overline{{\rm BMO}}([t,T])}\les  K+ K \int_t^T\sum_{j=1}^n\|\bar Y^j(\cdot)\|^2_{
L^{\infty}_{\dbF}(\Omega;C[s,T])} ds.\vspace{0.2cm}
\ee
On the other hand, by \eqref{eq:3.58} and H\"{o}lder's inequality we deduce that for each $ i=1,\cdots, k$ and $t\in [T_0,T]$,
$$
|\bar Y^i(t)|^2 \les K +K\int_t^T \sum_{j=1}^n\|\bar Y^j(\cdot)\|^2_{
L^{\infty}_{\dbF}(\Omega;C[s,T])
} ds,
$$
and then
\bel{eq:3.65}
\sum_{j=1}^k\|\bar Y^j(\cdot)\|^2_{L^{\infty}_{\dbF}(\O;C[t,T])} \les K + K\int_t^T \sum_{j=1}^n\|\bar Y^j(\cdot)\|^2_{L^{\infty}_{\dbF}(\O;C[s,T])} ds.\vspace{0.2cm}
\ee
Combining \eqref{eq:3.64} and \eqref{eq:3.65} yields that
\bel{Added3.65-1}\ba{ll}
\ns\ds\sum_{j=1}^n\|\bar Y^j(\cd)\|^2_{L^\infty_\dbF(\O;C[t,T])} \les K+K\int_t^T\sum_{j=1}^n\|\bar Y^j(\cd)\|^2_{L^\infty_\dbF(\O;C[s,T])}ds, \ \  t\in [T_0,T].%
\ea\ee
It then follows from Gronwall's inequality that
\bel{eq:3.66}\sum_{j=1}^n\|\bar Y^j(\cd)\|^2_{L^\infty_\dbF(\O;C[t,T])}\les K,\ \ t\in [T_0,T].\ee
Finally, combining \eqref{eq:3.60}, \eqref{eq:3.64} and \eqref{eq:3.66} yields that \vspace{0.1cm}
\bel{eq:3.67}\sum_{j=1}^n\|\bar Z^j(\cd)\|^2_{\cl{\rm BMO}[t,T]}\les K, \ \  t\in [T_0,T].\ee
Thus, the desired conclusion follows from \eqref{eq:3.66}--\eqref{eq:3.67} immediately. \vspace{0.4cm}  \endpf

\ms

The following existence and uniqueness result on the bounded solution of multi-dimensional BSDEs strengthens \cite[Theorem 2.4]{Fan-Hu-Tang 2020}, which amounts to saying $k=n$ in (H3.2)

\bt{th:3.2} \sl Let {\rm (H3.2)} hold. Then, for any $\xi\in L^\infty_{{\color{red} \cF_T}}(\O;\dbR^n)$, BSDE \eqref{eq:3.48} admits a unique adapted solution $(Y(\cd),Z(\cd))\in L^\infty_\dbF(\O;C([0,T];\dbR^n))\times\overline{{\rm BMO}}([0,T];\dbR^n)$.
\et

\it Proof. \rm
In view of \cite[Theorem 2.1]{Fan-Hu-Tang 2020}, Proposition \ref{pro:3.2} and the proof of \cite[Theorem 2.4]{Fan-Hu-Tang 2020}, the desired conclusion follows immediately.\vspace{0.5cm}
\endpf

Now, let us come back the study on BSVIEs, and introduce the following assumption on the generator $g$ of BSVIE \rf{BSVIE1}, which is comparable with (H3.2).

\ms

{\bf(H3.3)} Let (H3.1) hold with \rf{growth1} replaced by the following:
For some $k=0,1,2,\cds,n$ and each $(t,s,y,z)\in \D^*[0,T]\times \dbR^n\times\dbR^n$, it holds that
\bel{Add-in-H3.2}\left\{\2n\ba{ll}
\ds|g^i(t,s,y,z)|\les\a(t,s)+\b\big(|y|+|z^i|^2\big),\qq 1\les i\les k,\\
\ss\ds |g^i(t,s,y,z)|\les\a(t,s)+\b\(|y|+\sum_{j=k+1}^n|z^j|\),\qq k+1\les i\les n.\ea\right.\ee

Like (H3.2), when $k=0$, the first condition in \rf{Add-in-H3.2} is absent and when $k=n$, the second condition of \rf{Add-in-H3.2} is absent. The following Proposition \ref{pro:3.3} is comparable with the previous Proposition \ref{pro:3.1}.

\bp{pro:3.3} \sl Let {\rm(H3.3)} hold. Then, for each $\psi(\cd)\in L^\infty(0,T;\dbR^n)$ and $y(\cd)\in L^\infty_{\dbF}(0,T;\dbR^n)$, BSVIE \rf{eq:1.4} admits a unique adapted solution $(Y(\cd),Z(\cd\,,\cd))\in L^\infty_\dbF(0,T;\dbR^n)\times \overline{{\rm BMO}}( \D^*[0,T];\dbR^n)$. Furthermore, there exists a constant $\rho>0$ depending on $(n,\b)$, and a constant $K_0>0$ depending only on $(\psi(\cd),\a(\cd,\cd))$ and $(\beta,n,T)$ such that
with
$$\bar y(t):=e^{\rho s\over 2} y(t),\q\bar Y(t):=e^{\rho s\over 2} Y(t),\q\bar Z(t,s):=e^{\rho s \over 2} Z(t,s),\q (t,s)\in\D^*[0,T],$$
as long as
$$\|\bar y(\cd)\|^2_{L^\infty_\dbF(t,T;\dbR^n)}\les K_0e^{K_0(T-t)},\q \ae\ t\in[0,T],$$
it holds that
$$\|\bar Y(\cd)\|^2_{L^\infty_{\dbF}(t,T;\dbR^n)}\les K_0e^{K_0 (T-t)},\q \ae\ t\in[0,T],$$
and
$$\|\bar Z(\cd\,,\cd)\|^2_{\overline{{\rm BMO}}(\D^*[0,T])}\les
K_0\exp\big(K_0e^{K_0T}\big)+K_0e^{K_0T}=:\bar K_0.$$
\ep

\ms

\it Proof. \rm In view of (H3.3), it follows from Theorem \ref{th:3.2} that for almost every $t\in[0,T]$, the following BSDE
\bel{eq:3.69}\ba{ll}
\ss\ds \eta(t,r)=\psi(t)+\int_r^T  g(t,s,y(s),\z(t,s))ds-\int_r^T \z(t,s)dW(s),\ \  r\in[t,T]\ea\ee
admits a unique solution $(\eta(t,\cd),\z(t,\cd))\in L^\infty_{\dbF}(\O;C([t,T];\dbR^n))\times\overline{{\rm BMO}}([t,T];\dbR^n)$. Moreover, by Proposition \ref{pro:3.2} we know that there exists a constant $K>0$ depending only on $\|y(\cd)\|_{L^\infty_{\dbF}(0,T;\dbR^n)}$ and $(\a(\cd\,,\cd),\psi(\cd),\b,n,T)$ such that
$$\|\eta(t,\cd)\|_{L^\infty_\dbF(\O;C([t,T])}
+\|\z(t,\cd)\|_{\overline{{\rm BMO}}[t,T]}\les K,\q\ae\; t\in[0,T].$$
Let
\bel{eq:3.70}Y(t)=\eta(t,t),\q Z(t,s)=\z(t,s),\q (t,s)\in\D^*[0,T].\ee
Then $(Y(\cd),Z(\cd\,,\cd))\in L^\infty_{\dbF}(0,T;\dbR^n)\times \overline{{\rm BMO}}(\D^*[0,T];\dbR^n)$ is an adapted solution of BSVIE \rf{eq:1.4}. Moreover, in view of Theorem \ref{th:3.2} again, a same argument as in Proposition \ref{pro:3.1} yields the uniqueness of the solution of BSVIE \rf{eq:1.4}.

\ms

Finally, in view of \rf{Add-in-H3.2}, using a similar argument to obtain \eqref{eq:3.60}, (\ref{eq:3.64}) and \eqref{Added3.65-1}, we can deduce that there exists a constant $K_0>0$ depending only on $(\a(\cd\,,\cd),\psi(\cd),\b,n,T)$ such that for almost every $t\in[0,T]$,
\bel{eq:3.76}
\sum_{j=1}^k\|\z^j(t,\cd)\|^2_{\overline{{\rm BMO}}[t,T]}\les K_0 \exp\(K_0 \int_t^T \sum_{j=1}^n\|y^j(\cd)\|_{L^\infty_\dbF(s,T;\dbR)} ds\),\ee
\bel{eq:3.77}\ba{ll}
\ds \sum_{j=k+1}^n\|\bar \z^j(t,\cd)\|^2_{\overline{{\rm BMO}}[t,T]}\les K_0+K_0\int_t^T\sum_{j=1}^n\|\bar y^j(\cdot)\|^2_{L^\infty_\dbF (s,T;\dbR)} ds\vspace{0.2cm}
\ea
\ee
and
\bel{eq:3.78}
\sum_{j=1}^n\|\bar \eta^j(t,\cdot)\|^2_{L^{\infty}_{\dbF}(\O;C[t,T])} \les K_0+ K_0\int_t^T \sum_{j=1}^n\|\bar y^j(\cdot)\|^2_{L^\infty_\dbF(s,T;\dbR)} ds,
\ee
where for each $j=1,\cdots,n$,
$$
\bar y^j(s):=e^{\rho s \over 2}y^j(s),\ \ \bar \eta^j(t,s):=e^{\rho s \over 2}\eta^j(t,s)\ \ {\rm and}\ \ \bar \z^j(t,s):=e^{\rho s \over 2}\z^j(t,s),\ \ s\in (t,T].$$
Thus, in view of \rf{eq:1.16}, the desired assertion follows from \eqref{eq:3.70}--\eqref{eq:3.78} immediately.\endpf

\ms\ms

The following Theorem \ref{th:3.3} is comparable with the previous Theorem  \ref{th:3.1}, which is the main result of this subsection.

\bt{th:3.3} \sl Let {\rm(H3.3)} hold. Then, for each $\psi(\cd)\in L^\infty(0,T;\dbR^n) $, BSVIE \rf{BSVIE1} admits a unique adapted solution $(Y(\cd),Z(\cd\,,\cd))\in L^\infty_{\dbF} (0,T;\dbR^n)\times \overline{{\rm BMO}}(\D^*[0,T];\dbR^n)$.
\et

\it Proof. \rm Let $K_0$, $\bar K_0$ and $\rho$ be the constants in Proposition \ref{pro:3.3}. It is clear that
$$\ba{ll}
\ss\ds\sB:=\Big\{U(\cd),V(\cd\,,\cd))\in L^\infty_\dbF (0,T;\dbR^n)\times\overline{{\rm BMO}}(\D^*[0,T];\dbR^n)\bigm|\\
\ss\ds\qq\qq \|\bar U(\cd)\|^2_{ L^\infty_\dbF(t,T;\dbR^n)}\les K_0e^{K_0(T-t)},\ \ \ae\ t\in[0,T] \ \ {\rm and}\ \ \|\bar V(\cd\,,\cd)\|^2_{\overline{{\rm BMO}}(\D^*[0,T])}\les \bar K_0,\\
\ss\ds\qq\qq\hb{where for each}\ t\in[0,T], \ \ \bar U(s):=e^{ks\over 2}U(s)\ \ {\rm and}\ \ \bar V(t,s):=e^{ks\over 2}V(t,s),\ s\in (t,T]\Big\}\ea$$
is a convex closed set in Banach space $L^\infty_{\dbF} (0,T;\dbR^n)\times \overline{{\rm BMO}}(\D^*[0,T];\dbR^n)$. By Proposition \ref{pro:3.3} we know that for each $(y(\cd),z(\cd\,,\cd))\in \sB$, the following BSVIE
$$
\ds Y(t)=\psi(t)+\int_t^T  g(t,s,y(s),Z(t,s))ds-\int_t^T Z(t,s)dW(s),\ \  t\in[0,T]
$$
admits a unique adapted solution $(Y(\cdot),Z(\cdot,\cdot))\in \sB$. That is to say, the map
$$
\G\big(y(\cdot),z(\cdot,\cdot)\big):=\big(Y(\cdot),Z(\cdot,\cdot)\big),\ \ \forall \big(y(\cdot),z(\cdot,\cdot)\big)\in \sB
$$
is well defined and stable in $\sB$. The rest proof runs as in Theorem \ref{th:3.1}. \vspace{0.2cm} \endpf

\section{Multi-dimensional BSVIEs with Unbounded Free Terms}

In this section, we are going to investigate multi-dimensional BSVIEs with unbounded free terms. We will treat the quadratic and sub-quadratic cases in a unified framework. Therefore, in this section, we let $\d\in(0,1]$, allowing $\d=1$ (which corresponds to the quadratic case).

\ms

We first introduce the following assumption on the generator $g$ of BSVIEs (\ref{BSVIE1}) which is comparable with assumptions (H2.2)--(H2.4) for BSDEs.
%

\ms

{\bf(H4.1)} The map $g:\D^*[0,T]\times\dbR^n\times \dbR^n\to\dbR^n$ is measurable so that $s\mapsto g(t,s,y,z)$ is $\dbF$-progressively measurable on $[t,T]$ for each $(t,y,z)\in[0,T)\times\dbR^n\times\dbR^n$. There exist constants $\b>0$, $\d\in(0,1]$ and a non-negative process $\a:\D^*[0,T]\times\O\to\dbR_+$ with $s\mapsto\a(t,s)$ being $\dbF$-progressively measurable on $[t,T]$, and with
$$\dbE\Big\{\esssup_{t\in[0,T]}\Big[\exp\(p\,\Big|\int_t^T\2n\a(t,s)ds\Big|^{2\d\over 1+\d}\)\Big]\Big\}<\infty,\q \forall p\ges1,$$
such that $g(t,s,y,z):=(g^1(t,s,y,z^1),\cdots, g^n(t,s,y,z^n))^\top$ with
\bel{Add-4.1}\ba{ll}
\ss\ds|g^i(t,s,y, z^i )|\1n\les\1n\a(t,s)
\1n+\1n\b\big(|y|\1n+\1n|z^i|^{1+\d}\big),\qq\forall(t,s,
y,z)\1n\in\1n\D^*[0,T]\1n\times\1n\dbR^n\1n\times\1n\dbR^n,~1\les i\les n,\\
\ss\ds|g(t,s,y,z)-g(t,s,\bar y,z)|\les\b|y-\bar y|,\qq\q\forall(t,s)\in\D^*[0,T],~y,\bar y,z\in\dbR^n,\ea\ee
and
\bel{eq:4.1}\ba{ll}
\ss\ds
g^i(t,s,y,(1-\th)z^i+\th\bar z^i)-\th g^i(t,s,y,\bar z^i)\les (1-\th)
\big[\a(t,s)+\b\big(|y|+|z^i|^{1+\d}\big)\big],\\
\ss\ds\qq\qq\qq\qq
\qq\forall(t,s)\in\D^*[0,T],~y,z,\bar z\in\dbR^n,~\th\in(0,1),~1\les i\les n.\ea\ee

Relevant to the above, let us introduce the following hypothesis.

\ms

{\bf(H4.1)$'$} Let (H4.1) hold with \rf{eq:4.1} replaced by the following:
For each $i=1,2,\cds,n$, either
\bel{eq:4.1*}\ba{ll}
\ss\ds
g^i(t,s,y,(1-\th)z^i+\th\bar z^i)-\th g^i(t,s,y,\bar z^i)\les (1-\th)
\big[\a(t,s)+\b\big(|y|+|z^i|^{1+\d}\big)\big],\\
\ss\ds\qq\qq\qq\qq\qq
\qq\forall(t,s)\in\D^*[0,T],~y,z,\bar z\in\dbR^n,~\th\in(0,1),\ea\ee
or
\bel{eq:4.1**}\ba{ll}
\ss\ds
g^i(t,s,y,(1-\th)z^i+\th\bar z^i)-\th g^i(t,s,y,\bar z^i)\ges-(1-\th)
\big[\a(t,s)+\b\big(|y|+|z^i|^{1+\d}\big)\big],\\
\ss\ds\qq\qq\qq\qq\qq
\qq\forall(t,s)\in\D^*[0,T],~y,z,\bar z\in\dbR^n,~\th\in(0,1).\ea\ee

Although (H4.1)$'$ seems to be a little general than (H4.1), we now show that actually they are equivalent after some minor modifications. In fact, by relabeling, we may assume that for some $k=0,1,2,\cds,n$, one has \rf{eq:4.1*} holds for $1\les i\les k$, and \rf{eq:4.1**} holds for $k+1\les i\les n$. If $k=0$, it means that for all $i=1,2,\cds,n$, \rf{eq:4.1**} holds, and if $k=n$, it means that for all $i=1,2,\cds,n$, \rf{eq:4.1*} holds. Now, we define
$$\ba{ll}
\ss\bar z=(\bar z^1,\cds,\bar z^k,\bar z^{k+1},\cds,\bar z^n)=(z^1,\cds,z^k,-z^{k+1},
\cds,-z^n),\\
\ss\bar y=(\bar y^1,\cds,\bar y^k,\bar y^{k+1},\cds,\bar y^n)=(y^1,\cds,y^k,-y^{k+1},
\cds,-y^n),\\
\ss\ds\bar g^i(t,s, y,{\color{blue} z^i})=g^i(t,s,\bar y,{\color{blue}\bar z^i}),\q 1\les i\les k,\\
\ss\ds\bar g^i(t,s,y,{\color{blue} z^i})=-g^i(t,s,\bar y,
{\color{blue}\bar z^i}),\q k+1\les i\les n.\ea$$
Then (H4.1)$'$ holds for $g(t,s,y,z)$ if and only if (H4.1) holds for $\bar g(t,s,\bar y,\bar z)$.

\ms

Next, we introduce the following hypothesis.

\ms

{\bf(H4.2)} Let (H4.1) hold with \rf{eq:4.1} replaced by the following: for each $i=1,2,\cds,n$, and each $(t,s,y)\in\D^*[0,T]\times\dbR^n$, $z^i\mapsto g^i(t,s,y,z^i)$ be convex, or concave.

\ms

It is clear that when (H4.2) holds, then (H4.1)$'$ must hold for $g$. Hence, (H4.1) (or (H4.1)$'$)  is not very restrictive.

\ms

Further, if the generator $g$ satisfies (\ref{eq:4.1*}) or (\ref{eq:4.1**}), then $g^i(t,s,y,\cd\,)$ is locally Lipschitz continuous for each $(t,s,y)\in \D^*[0,T]\times\dbR^n$ and each $i=1,2,\cdots,n$. We only verify the case of (\ref{eq:4.1*}). The other case can be proved in the same way. Suppose that \eqref{eq:4.1*} holds. For each fixed $z_0\in\dbR$, let
$$\wt g^{\,i}(t,s,y,z):=g^i(t,s,y,z_0+z)-g^i(t,s,y,z_0),\q \forall (t,s,y,z)\in \D^*[0,T]\times\dbR^n\times\dbR,\ \ 1\les i\les n.$$
By \rf{eq:4.1*}, for each $(z_1,z_2)\in\dbR\times\dbR$ and $\th\in (0,1)$, we have $\wt g^i(t,s,y,0)=0$ and
\bel{eq:2.47}\ba{ll}
\ns\ds\wt g^{\,i}(t,s,y,(1-\th)z_1+\th z_2)-\th\wt g^{\,i}(t,s,y,z_2)\\
\ns\ds=g^i(t,s,y,(1-\th)(z_0+z_1)+\th(z_0+z_2))-\th g^i(t,s,y,z_0+ z_2)-(1-\th)g^i(t,s,y,z_0)\\
\ns\ds\les(1-\th)\big(\a(t,s)+\b|y|+\b|z_0+z_1|^{1+\d}\big)
+(1-\th)|g^i(t,s,y,z_0)|\\
\ns\ds\les(1-\th)\big(K(t,s,y,z_0)+2^\d\b|z_1|^{1+\d}\big),\ea\ee
with
$$
K(t,s,y,z_0)=\a(t,s)+\b|y|+2^\d\b|z_0|^{1+\d}+|g^i(t,s,y,z_0)|.
\vspace{0.2cm}
$$
Letting $z_1=1,z_2=0$ and $z_1=1,z_2=1-1/\th$ in \rf{eq:2.47} respectively yields that for each $\th\in (0,1)$,
$$\wt g^{\,i}(t,s,y,1-\th)\les\big(K(t,s,y,z_0)+2^\d\b\big)(1-\th)$$
and
$$\wt g^{\,i}\big(t,s,y,1-{1\over\th}\big)\ges-\big(K(t,s,y,z_0)+2^\d\b)
\Big|1-{1\over\th}\Big|,$$
which means that
$$|g^i(t,s,y,z_0+z)-g^i(t,s,y,z_0)|=|\wt g^{\,i}(t,s,y,z)|\les\big(K(t,s,y,z_0)+2^\d\b\big)|z|,\q\forall z\in (-1,1).$$
This gives the local Lipschitz of $g^i(t,s,y,\cdot)$.

\ms

\ms

 In what follows, we use $K>0$ to represent a generic positive constant depending only on $(\b,\d,T,n)$ which could be different from line to line, and if it also depends on $p$, we will denote it by $K_p$. In addition, for notational convenience, we will frequently use the following function
\bel{eq:new1}\F_\l(x;\m):=\exp(\m x^\l),\q  \forall x,\m\ges 0,\ \ \l>0.\ee
Note that for each $x_1,x_2,\m\ges 0$ and $\l\in (0,1]$,
\bel{eq:new}\F_\l(x_1+x_2;\m)\les\F_\l(x_1;\m)\F_\l(x_2;\m).\ee

\bp{pro:2.1} \sl Let {\rm (H4.1)} hold and $\lambda\={2\d\over 1+\d}$. Then, for any $\psi(\cd)\in \cap_{p\ges 1}\cE^{p,\lambda}_{\cF_T}(\O;L^{\infty}(0,T;\dbR^n)) $ and $y(\cd)\in \cap_{p\ges 1}\cE_\dbF^{p,\lambda}(\O;L^{\infty}(0,T;\dbR^n)) $, the following BSVIE
\bel{eq:2.1}Y(t)=\psi(t)+\int_t^Tg(t,s,y(s),Z(t,s))ds-\int_t^T Z(t,s)dW(s),\ \  t\in\T\ee
admits a unique adapted solution
$$(Y(\cd),Z(\cd\,,\cd))\in  \Big[\bigcap_{p\ges 1}\cE_\dbF^{p,\lambda}(\O;L^{\infty}(0,T;\dbR^n))\Big] \times\Big[ \bigcap_{p\ges1}L^\infty(0,T;L^p_\dbF(\O;L^2(\cd\,,T;\dbR^n)))\Big],$$ %
such that $\eta(t,\cd)\in \cap_{p\ges 1}\cE_\dbF^{p,\lambda}(\O;C([t,T];\dbR^n)) $ for almost every $t\in[0,T]$, where
\bel{eq:4.5} \eta(t,r):=Y(t)-\int_t^r g(t,s,y(s),Z(t,s))ds+\int_t^r Z(t,s)dW(s),\ \ r\in[t,T].\ee

\ep

\it Proof. \rm Fix $y(\cd)\in \cap_{p\ges 1}\cE_\dbF^{p,\lambda}(\O;L^{\infty}(0,T;\dbR^n))$. For each $i=1,\cds,n$, and almost every $t\in[0,T]$, it follows from \rf{Add-4.1} and \rf{eq:4.1} that
\bel{eq:2.5}|g^i(t,s,y(s),z)|\les\a(t,s)+\b |y(s)|+\b|z|^{1+\d},\ \ \forall (s,z)\in[t,T]\times\dbR.\ee
and
\bel{eq:2.5**}\ba{ll}
\ss\ds
g^i(t,s,y(s),(1-\th)z^i+\th\bar z^i)-\th g^i(t,s,y(s),\bar z^i)\les (1-\th)
\big[\a(t,s)+\b|y(s)|+\b|z^i|^{1+\d}\big],\\
\ss\ds\qq\qq\qq\qq
\qq\forall(t,s)\in\D^*[0,T],~z,\bar z\in\dbR^n,~\th\in(0,1),~1\les i\les n.\ea\ee
Furthermore, due to integrability of $\a(\cd\,,\cd)$, $\psi(\cd)$ and $y(\cd)$, by H\"{o}lder's inequality we have
\bel{eq:2.6}\dbE\Big[ \Phi_\l\Big(|\psi^i(t)|+\int_t^T\big(\a(t,s)+\b |y(s)|\big)ds; p\Big) \Big]<+\infty,\ \ \forall p>1.\ee
Then, by virtue of Proposition \ref{Prop2.4} we know that the following one-dimensional parameterized BSDE
\bel{eq:2.7}
\eta^i(t,r)=\psi^i(t)+\int_r^Tg^i(t,s,y(s),\z^i(t,s))ds-\int_r^T \z^i(t,s)dW(s),\ \ r\in [t,T]\ee
admits a unique adapted solution $(\eta^i(t,\cd),\z^i(t,\cd))$, taking value in $\dbR^2$ such that
$$\dbE\[\Phi_\l\(\sup_{s\in [t,T]}|\eta^i(t,s)|;p\)+\(\int_t^T |\z^i(t,s)|^2  d t\)^{p\over 2}\]<+\infty,\ \ \forall p>1.$$
%
%
Let
\bel{eq:2.8}\eta(t,\cd):=(\eta^1(t,\cd),\cds,\eta^n(t,\cd))^\top \ \ {\rm and}\ \ \z(t,\cd):=(\z^1(t,\cd),\cds,\z^n(t,\cd)^\top.\vspace{0.1cm}\ee
Then, by virtue of H\"{o}lder's inequality we know that for almost every $t\in[0,T]$, $(\eta(t,\cdot),\z(t,\cd))$ is the unique solution in the space $ [\bigcap_{p\ges 1}\cE_\dbF^{p,\lambda}(\O;C([t,T];\dbR^n))] \times
[\bigcap_{p\ges1}L^p_\dbF(\Omega;L^2(t,T;\dbR^n))]$ of the following BSDE
\bel{eq:2.9}\eta(t,r)=\psi(t)+\int_r^Tg(t,s,y(s),\z(t,s))ds-\int_r^T \z(t,s)dW(s),\ \ r\in [t,T].\ee
Let
\bel{eq:2.10}\ba{ll}
\ss\ds Y(t)=\eta(t,t), \ \ Z(t,s)=\z(t,s),\ \ \ (t,s)\in \D^*[0,T].\ea\ee
Then, it follows from \eqref{eq:2.9} and \eqref{eq:2.10} that $(Y(\cdot),Z(\cdot,\cdot))$ solves BSVIE \eqref{eq:2.1}, and that the $\eta(t,\cd)$ defined in (\ref{eq:4.5}) equals to that in \eqref{eq:2.9} and then it belongs to $\bigcap_{p\ges 1}\cE_\dbF^{p,\lambda}(\O;C([t,T];\dbR^n))$.  \vspace{0.4cm}

Furthermore, define
\bel{Psi-t-p-def}\ba{ll}
\ss\ds \Psi_\l(\nu):=\esssup_{t\in[0,T]}\Phi_\l\Big(|\psi(t)|+\int_t^T \a(t,s)ds; \nu\Big),\ \ \nu\in [0,+\infty).
\ea\ee
In view of \rf{eq:2.5}--\rf{eq:2.7}, by using Lemma \ref{Lemma2.5} we can deduce that for each $i=1,\cdots,n$, almost every $t\in[0,T]$ and each $p>1$,
$$\ba{ll}
\ds\Phi_\l\Big(|Y^i(t)|;1\Big)=\ds \Phi_\l\Big( |\eta^i(t,t)|;1\Big)\les K\dbE_t\[\Phi_\l\(|\psi^i(t)|+\int_t^T \big(\a(t,s)+\b |y(s)|\big)ds; K\)\]\\
\hspace{2.25cm} \les \ds K\dbE_t\[\Psi_\l(K)\Phi_\l\(\int_t^T |y(s)| ds; K\)\]\ea$$
and
$$\ds\dbE\[\Phi_\l\Big(\sup_{s\in [t,T]}|\eta^i(t,s)|;p\Big)\]+\dbE\[\(\int_t^T |Z^i(t,s)|^2 ds\)^{p\over 2}\]\les\ds
 K\dbE\[\Psi_\l(K_p)\Phi_\l\(\int_t^T |y(s)| ds; K_p\)\],$$
and then, by virtue of Jensen's inequality, Doob's martingale inequality and H\"{o}lder's inequality together with \eqref{eq:2.8}, \eqref{eq:new} and \eqref{eq:2.10}, we have
\bel{eq:2.3*}\ba{ll}
\ds\Phi_\l\Big(|Y(t)|;1\Big)\les\Phi_\l\Big(\sum_{i=1}^n |Y^i(t)|;1\Big)\ds =\prod_{i=1}^n \Phi_\l\Big( |Y^i(t)|;1\Big)\\
\ \ \ds\les \Big\{K\dbE_t\[\Psi_\l(K)\Phi_\l\(\int_t^T |y(s)| ds; K\)\]\Big\}^{n}\les K\dbE_t\[\Psi_\l(K)\Phi_\l\(\int_t^T |y(s)| ds; K\)\],
\ea \ee
\bel{eq:2.3}\ba{ll}
\ds\dbE\[\Phi_\l\Big(\esssup_{t\in[0,T]} |Y(t)|;p\Big)\]=\ds \dbE\[\esssup_{t\in[0,T]}\Phi_\l\Big( |Y(t)|;p\Big)\]\\
\ns\ds\les K_p\dbE\Big\{\esssup_{t\in[0,T]}\(\dbE_t\[\Psi_\l({\color{red}K})\Phi_\l\(\int_0^T |y(s)|ds;{\color{red}K}\)\]\)^p\Big\}\les\ds K_p\dbE\[\Psi_\l(K_p)\Phi_\l\(\int_0^T |y(s)| ds;K_p\)\]\vspace{0.2cm}
\ea\ee
and
\bel{eq:2.11}\ba{ll}
\ds\dbE\[\Phi_\l\Big(\sup_{s\in [t,T]}|\eta(t,s)|;p\Big)\]+\dbE\[\(\int_t^T |Z(t,s)|^2 ds\)^{p\over 2}\]\\
\ns\ds\les \ds \dbE\[\prod_{i=1}^{n}\Phi_\l\Big(\sup_{s\in [t,T]}|\eta^i(t,s)|;p\Big)\]+n^{p-1}\sum_{i=1}^{n}\dbE\[\(\int_t^T |Z^i(t,s)|^2 ds\)^{p\over 2}\]\\
\ns\ds\les\prod_{i=1}^n\Big\{\dbE\[\Phi_\l\(\sup_{s\in [t,T]}|\eta^i(t,s)|;np\Big)\]\Big\}^{1\over n}+K_p\dbE\[\Psi_\l(K_p)\Phi_\l\(\int_t^T|y(s)|ds;K_p\)\]\\
\ns\ds\les \ds K_p\dbE\[\Psi_\l(K_p)\Phi_\l\(\int_t^T|y(s)|ds;K_p\)\].
\ea\ee
Therefore, in view of the assumptions on $\psi(\cd)$, $\a(\cd\,,\cd)$ and $y(\cd)$, it follows from \eqref{eq:2.3}, \eqref{eq:2.11} and H\"{o}lder's inequality that
$$(Y(\cdot),Z(\cdot,\cdot))\in \Big[\bigcap_{p\ges 1}\cE_\dbF^{p,\lambda}(\O;L^\infty(0,T;\dbR^n))\Big]\times\Big[\bigcap_{p\ges1} L^\infty(0,T;L^p_\dbF(\O;L^2(\cd\,,T;\dbR^n)))\Big].$$
Finally, we look at the uniqueness. Suppose
$$(\wt Y(\cdot),\wt Z(\cdot,\cdot))\in \Big[\bigcap_{p\ges 1}\cE_\dbF^{p,\lambda}(\O;L^\infty(0,T;\dbR^n))\Big] \times\Big[\bigcap_{p\ges1} L^\infty(0,T;L^p_\dbF(\O;L^2(\cd\,,T;\dbR^n)))\Big]$$
is another pair of adapted solution of BSVIE \eqref{eq:2.1} such that $\wt\eta(t,\cd)\in \cap_{p\ges 1}\cE_\dbF^{p,\lambda}(\O;C([t,T];\dbR^n)) $ for almost every $t\in[0,T]$, where
\bel{Added-property-eta}\ba{ll}
\ss\ds \wt\eta(t,r):=\wt Y(t)-\int_t^r g(t,s,U(s),\wt Z(t,s))ds+\int_t^r \wt Z(t,s)dW(s),\ \ r\in[t,T].
\ea
\ee
Then for every almost $t\in[0,T]$,
$$(\wt \eta(t,\cd),\wt Z(t,\cd))\in  \Big[\bigcap_{p\ges 1}\cE_\dbF^{p,\lambda}(\O;C([t,T];\dbR^n))\Big] \times\Big[\bigcap_{p\ges1} L^p_{\dbF}(\Omega;L^2(t,T;\dbR^n))\Big]$$
satisfies (\ref{eq:2.9}). By the uniqueness of BSDEs, we see that $\wt\eta(t,\cdot)=\eta(t,\cdot)$, $\wt Z(t,\cdot)=Z(t,\cdot)$ for almost every $t\in[0,T]$, which implies that $Y(\cdot)=\wt Y(\cdot).$
This completes the proof.  \endpf
%
%

\vspace{0.5cm}
The following Theorem \ref{th:2.1} is the main result of this section.
\bt{th:2.1} \sl
Let {\rm (H4.1)} hold and $\lambda={2\d\over 1+\d}$. Then, for any $\psi(\cd)\in \cap_{p\ges 1}\cE^{p,\lambda}(\O;L^{\infty}(0,T;\dbR^n)) $, the following BSVIE
\bel{eq:2.15}
\ba{ll}
\ss\ds Y(t)=\psi(t)+\int_t^Tg(t,s,Y(s),Z(t,s))ds-\int_t^T Z(t,s)dW(s),\ \ t\in[0,T]
\ea
\ee
admits a unique pair of adapted solution
$$(Y(\cdot),Z(\cdot,\cdot))\in   \Big[\bigcap_{p\ges 1}\cE_\dbF^{p,\lambda}(\O;L^\infty(0,T;\dbR^n))\Big] \times\Big[\bigcap_{p\ges 1} L^\infty(0,T;L^p_\dbF(\O;L^2(\cd\,,T;\dbR^n)))\Big]$$
such that $\eta(t,\cd)\in\bigcap_{p\ges 1}\cE_\dbF^{p,\lambda}(\O;C([t,T];\dbR^n)) $
for every almost $t\in[0,T]$, where
\bel{eq:4.21}\ba{ll}
\ss\ds \eta(t,r):=Y(t)-\int_t^r g(t,s,Y(s),Z(t,s))ds+\int_t^r Z(t,s)dW(s),\ \ r\in[t,T].
\ea
\ee
\et

\it Proof. \rm Set $(Y^{(0)}(\cdot),Z^{(0)}(\cdot,\cdot))=(0,0)$ and define recursively the sequence $\{(Y^{(m)}(\cdot),Z^{(m)}(\cdot,\cdot))\}_{m=1}^\infty$ by the adapted solutions of the following BSVIEs given through the way in Proposition \ref{pro:2.1}, with $t\in\T$,
\bel{eq:2.16}
Y^{(m+1)}(t)=\psi(t)+\int_t^T g(t,s,Y^{(m)}(s),Z^{(m+1)}(t,s))ds-\int_t^T Z^{(m+1)}(t,s)dW(s).
\ee
And, for almost every $t\in[0,T]$ and $m\ges 0$, by virtue of Proposition \ref{Prop2.4} we let $(\eta^{(m+1)}(t,\cd),\z^{(m+1)}(t,\cd))$ be the unique solution in the space $ [\bigcap_{p\ges 1}\cE_\dbF^{p,\lambda}(\O;C([t,T];\dbR^n))]\times[\bigcap_{p\ges1} L^p_{\dbF}(\O;L^2(t,T;\dbR^n))]$ of the following BSDE
\bel{eq:2.23}
\ba{ll}
\ss\ds \eta^{(m+1)}(t,r)=\psi(t)+\int_r^Tg(t,s,Y^{(m)}(s),\z^{(m+1)}(t,s))ds-\int_r^T \z^{(m+1)}(t,s)dW(s),\ \ r\in [t,T].\vspace{0.1cm}
\ea
\ee
It then follows from Proposition \ref{pro:2.1} and \eqref{eq:2.16} that
\bel{eq:2.24}
\ba{ll}
\ss\ds Y^{(m+1)}(t)=\eta^{(m+1)}(t,t), \ \ Z^{(m+1)}(t,s)=\z^{(m+1)}(t,s),\ \ \ \ae\ (t,s)\in \D^* [0,T].\vspace{0.1cm}
\ea
\ee
In addition, for almost every $t\in[0,T]$ and each $r\in [t,T]$, we have
$$\ba{ll}
\ss\ds \eta^{(m+1)}(t,r)=Y^{(m+1)}(t)-\int_t^rg(t,s,Y^{(m)}(s),\z^{(m+1)}(t,s))ds+\int_t^r \z^{(m+1)}(t,s)dW(s).
\ea
$$
Then, in view of \rf{eq:2.23} and \rf{eq:2.24}, by using Lemma \ref{Lemma2.5} together with a similar argument to \eqref{eq:2.3*} and \eqref{eq:2.11} we can deduce that for almost every $t\in[0,T]$ and for each $m\ges 0$ and $p>1$,
\bel{eq:2.17}\Phi_\l\Big(|Y^{(m+1)}(t)|,1\Big)\les\ds
K\dbE_t\[\Psi_\l(K)\Phi_\l\(\int_t^T|Y^{(m)}(s)|ds;K\)\],\ee
and
\bel{eq:2.18*}\ba{ll}
\ds\dbE\[\Phi_\l\Big(\sup_{s\in [t,T]}|\eta^{(m+1)}(t,s)|;p\Big)\]
+\dbE\[\(\int_t^T|Z^{(m+1)}(t,s)|^2ds\)^{p\over 2}\]\\
\ds\les K_p\dbE\[\Psi_\l(K_p)\Phi_\l\(\int_0^T|Y^{(m)}(s)|ds;K_p\)\],\vspace{0.1cm}
\ea\ee
where and hereafter, $\Psi_\l(\cd)$ and $\Phi_\l(\cd\,;\cd)$ is defined respectively in \rf{Psi-t-p-def} and \rf{eq:new1}.\vspace{0.4cm}

In the sequel, we will show that $\{(Y^{(m)}(\cdot),Z^{(m)}(\cd,\cd))\}_{m=1}^\infty$ is a Cauchy sequence and converges to a process $(Y(\cdot), Z(\cdot,\cdot))$ in the space $ [\bigcap_{p\ges 1}\cE_\dbF^{p,\lambda}(\O;L^\infty(0,T;\dbR^n))]\times[\bigcap_{p\ges 1} L^\infty(0,T;L^p_\dbF(\O;L^2(\cd\,,T;\dbR^n)))]$, which is a desired solution of BSVIE \eqref{eq:2.15}. First, we prove that for each $p>1$,\vspace{0.1cm}
\bel{eq:2.19}\sup_{m\ges 1}\dbE\[\Phi_\l\(\esssup_{t\in [0,T]}|Y^{(m)}(t)|; p\)\]<+\infty,\ee
and
\bel{eq:2.25}
\sup_{m\ges 1}\esssup_{t\in[0,T]}\Big\{\dbE\[\Phi_\l\Big(\sup_{s\in [t,T]} |\eta^{m}(t,s)|;p\Big)\]+\dbE\[\(\int_t^T |Z^{(m)}(t,s)|^2 ds\)^{p\over 2}\]\Big\}<+\infty.
\ee
In fact, in view of \eqref{eq:2.17}, by Jensen's inequality we deduce that for almost every $t\in[0,T]$ and for each $m\ges 0$ and $k>0$,
\bel{eq:2.21}
\ds \Phi_\l\Big(e^{kt}|Y^{(m+1)}(t)|; 1\Big)\les\bar K\dbE_t\[\Psi_\l(\bar K)
\Phi_\l\(e^{kt}\int_t^T|Y^{(m)}(s)|ds;K\)\]\ee
with $\bar K$ being a constant depending on $k$. Observe that
$$
e^{kt}\int_t^T |Y^{(m)}(s)| ds \les \ds e^{kt}\int_t^T e^{-ks} ds\sup_{s\in[t,T]}\big(e^{ks}  |Y^{(m)}(s)|\big)\les \frac{1} {k} \sup_{s\in [0,T]}\big(e^{ks} |Y^{(m)}(s)|\big).
$$
We have
$$\Phi_\l\(e^{kt}|Y^{(m+1)}(t)|;1\)\les\bar K\dbE_t\[\Psi_\l(\bar K)\Phi_\l\(
\sup_{s\in[0,T]}\big(e^{ks}|Y^{(m)}(s)|\big);\frac{K} {k^\l}\)\],\vspace{0.1cm}
$$
and by Doob's martingale inequality and H\"{o}lder's inequality,
$$\ba{ll}
\ss\ds \dbE\[\esssup_{t\in[0,T]}\Phi_\l\Big(e^{kt}|Y^{(m+1)}(t)|; p\Big)\]\les\bar K \Big({p\over p-1}\Big)^p \dbE\[\Psi_\l(p\bar K)
\Phi_\l\(\sup_{s\in [0,T]}\big(e^{ks} |Y^{(m)}(s)| \big); \frac{p K} {k^\l})
\]\\
\ \ \ds \les\bar K_p\sqrt{\dbE\big[\Psi_\l(\bar K_p)]}\Big\{\dbE\[\Phi_\l\(\esssup_{s\in \T}\big(e^{ks}|Y^{(m)}(s)|\big); \frac{2p K} {k^\l}\)\]\Big\}^{1\over 2},\ \ \forall p>1,\ea$$
where $\bar K_p$ being a variable constant further depending on $k$ and $p$.
Letting $k_0=(2K)^{1\over \l}$ in the last inequality yields that for each $p>1$,
$$\dbE\[\Phi_\l\(\esssup_{t\in[0,T]}\big(e^{k_0 t}|Y^{(m+1)}(t)|\big); p\)\]\les\h K_p
\sqrt{\dbE\big[\Psi_\l(\h K_p)
\big]}\{\dbE\[\Phi_\l\(\esssup_{t\in[0,T]}\big(e^{k_0 t}|Y^{(m)}(t)|\big); p\)\]\Big\}^{1\over 2}
$$
with $\h K_p$ being a constant depending on $k_0$ and $p$.
Then by induction,
\bel{eq:2.22}\ba{ll}
\ds\dbE\[\Phi_\l\(\esssup_{t\in[0,T]}\big(e^{k_0 t}|Y^{(m+1)}(t)|\big); p\)\]\les\h K_p\Big\{\sqrt{\dbE\big[\Psi_\l(\h K_p)
\big]}\Big\}^{1+\frac 1 2+\cdots+\frac{1}{2^m}}\\
\ \ \ \ \ds \cd\Big\{\dbE\[\Phi_\l\(\esssup_{t\in[0,T]}
\big(e^{k_0 t}|Y^{(0)}(t)|\big); p\g\)\]\Big\}^{\frac {1} {2^{m+1}}}\les \h K_p\dbE\[\Psi_\l(\h K_p)\]\ea\ee
Therefore, \eqref{eq:2.19} follows from \eqref{eq:2.22}. Furthermore, in view of \rf{eq:2.18*}, \eqref{eq:2.19} and the assumptions on $\psi(\cd)$ and $\a(\cd\,,\cd)$, it follows from H\"{o}lder's inequality that \eqref{eq:2.25} is also true. \vspace{0.4cm}

Next, for each fixed $m,p\ges 1$, $\theta\in (0,1)$ and $(t,s)\in \D\T$, define\vspace{0.1cm}
$$
\delta_{\theta}\eta^{(m,p)}(t,s):=\frac{\eta^{(m+p)}(t,s)-\theta \eta^{(m)}(t,s)}{1-\theta},\ \ \ \delta_{\theta}Y^{(m,p)}(t):=\frac{Y^{(m+p)}(t)-\theta Y^{(m)}(t)}{1-\theta}
$$
and
$$
\delta_{\theta}Z^{(m,p)}(t,s):=\frac{Z^{(m+p)}(t,s)-\theta Z^{(m)}(t,s)}{1-\theta}.\vspace{0.2cm}
$$
Then for almost every $t\in\T$,  $\delta_{\theta}Y^{(m,p)}(t)=\delta_{\theta}\eta^{(m,p)}(t,t)$, and
$$(\delta_{\theta}\eta^{(m,p)}(t,\cdot),\delta_{\theta}Z^{(m,p)}(t,\cdot))\in \Big[\bigcap_{p\ges 1}\cE_\dbF^{p,\lambda}(\O;C([t,T];\dbR^n))\Big]\times \Big[\bigcap_{p\ges 1} L^p_{\dbF} (\Omega;L^2(t,T;\dbR^n))\Big] $$
solves the following system of BSDEs: for $i=1,\cdots,n$ and $r\in [t,T]$,
$$
\ba{ll}
\ss\ds \delta_{\theta} \eta^{(m,p);i}(t,r)=\psi^i(t)+\int_r^T \delta_{\theta} g^{(m,p);i}\left(t,s,\delta_{\theta} Z^{(m,p);i}(t,s)\right)ds-\int_r^T \delta_{\theta} Z^{(m,p);i}(t,s)dW(s),
\ea
$$
where $\Pi^{\cdots;i}$ represents the $i$-th component $\Pi^{\cdots}$ and for each $(t,s,z)\in \D^*\T\times \dbR$,
\bel{eq:2.26}
\delta_{\theta} g^{(m,p);i}(t,s,z):=\frac{g^i\big(t,s,Y^{(m+p-1)}(s), (1-\theta)z+\theta Z^{(m);i}(t,s)\big)-\theta g^i\big(t,s,Y^{(m-1)}(s),Z^{(m);i}(t,s)\big)}{1-\theta}.
\ee
It follows from \eqref{eq:2.26}, \eqref{Add-4.1} and \rf{eq:4.1} that for each $(t,s,z)\in \D^*\T\times\dbR$,\vspace{0.2cm}
$$
\ba{ll}
\ns\ds \delta_{\theta} g^{(m,p);i}(t,s,z)\les \frac{ \beta|Y^{(m+p-1)}(s)-Y^{(m-1)}(s)|+ (1-\theta)\big(\alpha(t,s)+\beta|Y^{(m-1)}(s)|+\beta|z|^{1+\d}\big)}{1-\theta}\\
\ns\ds \hspace{2.55cm}\les \beta |\delta_{\theta}Y^{(m-1,p)}(s)|+\beta|Y^{(m-1)}(s)|+\alpha(t,s)+\beta|Y^{(m-1)}(s)|+\beta|z|^{1+\d}\\
\ns\ds \hspace{2.55cm}= \alpha(t,s)+2\beta|Y^{(m-1)}(s)|+\beta |\delta_{\theta}Y^{(m-1,p)}(s)|+\beta|z|^{1+\d}.\ea$$
Then, by virtue of Remark \ref{rmk:2.6} we deduce that
for almost every $t\in\T$ and each $i=1,\cdots,n$,
\bel{eq:2.27}\ba{ll}
\ds\Phi_\l\(\[\d_\th\eta^{(m,p);i}(t,r)\]^+;1\)\\
\ns\ds\les K\dbE_r\[\Phi_\l\((\psi^i(t))^+ +\int_r^T \big(\alpha(t,s)+|Y^{(m-1)}(s)|+ |\delta_{\theta}Y^{(m-1,p)}(s)|\big)ds; K\)\],\ \ r\in [t,T],\ea\ee
and then,
\bel{eq:2.28}\ba{ll}
\ds\Phi_\l\(\[\d_\th Y^{(m,p);i}(t)\]^+;1\)\\
\ns\ds\les K\dbE_t\[\Phi_\l\((\psi^i(t))^+ +\int_t^T\big(\alpha(t,s)+|Y^{(m-1)}(s)|+ |\delta_{\theta}Y^{(m-1,p)}(s)|\big)ds;K\)\].\ea\ee
On the other hand, define\vspace{0.1cm}
$$
\delta_{\theta}\tilde\eta^{(m,p)}(t,s):=\frac{\eta^{(m)}(t,s)-\theta \eta^{(m+p)}(t,s)}{1-\theta},\ \ \ \delta_{\theta}\tilde Y^{(m,p)}(t):=\frac{Y^{(m)}(t)-\theta Y^{(m+p)}(t)}{1-\theta}
$$
and \vspace{0.2cm}
$$
\delta_{\theta}\tilde Z^{(m,p)}(t,s):=\frac{Z^{(m)}(t,s)-\theta Z^{(m+p)}(t,s)}{1-\theta}.\vspace{0.2cm}
$$
The same computation as above yields that for almost every $t\in\T$ and each $i=1,\cdots,n$,
\bel{eq:2.29}\ba{ll}
\ds\Phi_\l\(\[\d_\th\wt\eta^{(m,p);i}(t,r)\]^+;1\)\\
\ns\ds\les K\dbE_r\[\Phi_\l\(\big(\psi^i(t)\big)^+ +\int_r^T \big(\a(t,s)+|Y^{(m+p-1)}(s)|+ |\d_\th\wt Y^{(m-1,p)}(s)|\big)ds;K\)\],\ \ r\in [t,T],\ea\ee
and then,
\bel{eq:2.30}\ba{ll}
\ds\Phi_\l\(\big(\d_\th\wt Y^{(m,p);i}(t)\big)^+;1\)\\
\ns\ds\les K\dbE_t\[\Phi_\l\((\psi^i(t))^+ +\int_t^T\big(\a(t,s)+|Y^{(m+p-1)}(s)|+ |\d_\th\wt Y^{(m-1,p)}(s)|\big)ds;K\)\].
\ea
\end{equation}

Furthermore, observe that for almost every $t\in\T$ and for each $i=1,\cdots,n$ and $r\in [t,T]$, we have
\begin{equation}\label{eq:2.31}
\begin{array}{ll}
\ds \left(\delta_{\theta} Y^{(m,p);i}(t)\right)^- = \ds \frac{\left(Y^{(m+p);i}(t)-\theta Y^{(m);i}(t)\right)^-} {1-\theta}= \ds \frac{\left(\theta Y^{(m);i}(t)-Y^{(m+p);i}(t)\right)^+} {1-\theta} \vspace{0.1cm}\\
\ \ \les \ds \frac{\theta\left( Y^{(m);i}(t)-\theta Y^{(m+p);i}(t)\right)^+ +(1-\theta^2)|Y^{(m+p);i}(t)|} {1-\theta}\les \ds \left(\delta_{\theta} \widetilde Y^{(m,p);i}(t)\right)^+ +2|Y^{(m+p)}(t)|,
\end{array}
\end{equation}
and similarly,
\begin{equation}\label{eq:2.32}
\left(\delta_{\theta} \widetilde Y^{(m,p);i}(t)\right)^- \les \left(\delta_{\theta} Y^{(m,p);i}(t)\right)^+ +2|Y^{(m)}(t)|,\vspace{0.1cm}
\end{equation}
\begin{equation}\label{eq:2.33}
\left(\delta_{\theta}\eta^{(m,p);i}(t,r)\right)^- \les \left(\delta_{\theta} \widetilde \eta^{(m,p);i}(t,r)\right)^+ +2|\eta^{(m+p)}(t,r)|,\vspace{0.1cm}
\end{equation}
\begin{equation}\label{eq:2.34}
\left(\delta_{\theta}\tilde\eta^{(m,p);i}(t,r)\right)^- \les \left(\delta_{\theta} \eta^{(m,p);i}(t,r)\right)^+ +2|\eta^{(m)}(t,r)|.\vspace{0.2cm}
\end{equation}
It then follows from \eqref{eq:2.28}, \eqref{eq:2.30}, \eqref{eq:2.31}, \eqref{eq:2.32}, \rf{eq:new} and Jensen's inequality that
$$\ba{ll}
\ds \Phi_\l\(|\d_\th Y^{(m,p);i}(t)|;1\)\les\Phi_\l\(\(\delta_{\theta} Y^{(m,p);i}(t)\)^+; 1\)\cdot \Phi_\l\(\(\delta_{\theta} Y^{(m,p);i}(t)\)^-; 1\) \vspace{0.1cm}\\
\ns\ds K\dbE_t\[\Phi_\l(|\psi(t)|+|Y^{(m+p)}(t)| +\int_t^T \(\alpha(t,s)+|Y^{(m-1)}(s)|+ |Y^{(m+p-1)}(s)|\)  ds \vspace{0.1cm}\\
\ds \hspace{2.4cm} + \int_t^T \( |\delta_{\theta}Y^{(m-1,p)}(s)|+|\delta_{\theta}\widetilde Y^{(m-1,p)}(s)| \)  ds; K\)\]
\end{array}
$$
and
$$\ba{ll}
\ds \Phi_\l\(|\d_\th\wt Y^{(m,p);i}(t)|; 1\)\les \Phi_\l\(\big(\d_\th\widetilde Y^{(m,p);i}(t)\)^+; 1\)\cdot \Phi_\l\(\(\delta_{\theta} \widetilde Y^{(m,p);i}(t)\t)^-; 1\) \vspace{0.1cm}\\
\ \ \les \ds
K\dbE_t\[\Phi_\l\(|\psi(t)|+|Y^{(m)}(t)| +\int_t^T \(\alpha(t,s)+|Y^{(m-1)}(s)|+ |Y^{(m+p-1)}(s)|\)  ds  \vspace{0.1cm}\\
\ds \hspace{2.4cm} + \int_t^T \( |\delta_{\theta}Y^{(m-1,p)}(s)|+|\delta_{\theta}\widetilde Y^{(m-1,p)}(s)| \) ds; K\)\].
\end{array}
$$
Consequently, by \rf{eq:new} and Jensen's inequality again we get that for almost every $t\in \T$,
$$\ba{ll}
\ds\Phi_\l\(|\d_\th Y^{(m,p);i}(t)|+|\d_\th\wt Y^{(m,p);i}(t)|;1\)\vspace{0.2cm}\\
\ \ \les \ds
K\dbE_t\[\Phi_\l\(|\psi(t)|+|Y^{(m)}(t)| + |Y^{(m+p)}(t)| +\int_t^T \big(\a(t,s)+|Y^{(m-1)}(s)|+ |Y^{(m+p-1)}(s)|\big)ds \vspace{0.1cm}\\
\ds \hspace{2.4cm}+\int_t^T\big(|\d_\th Y^{(m-1,p)}(s)|+|\delta_{\theta}\widetilde Y^{(m-1,p)}(s)|\big)ds;K\Big)\],\ \ \ \ i=1,\cds,n,\ea$$
and then
\bel{eq:2.35}\ba{ll}
\ds\Phi_\l\(|\d_\th Y^{(m,p)}(t)|+|\d_\th\wt Y^{(m,p)}(t)|;1\)\\
\ns\ds\les K\dbE_t\[\Phi_\l\(\esssup_{s\in\T}\big( |Y^{(m)}(s)|+|Y^{(m+p)}(s)|\big) +\int_0^T \big(|Y^{(m-1)}(s)|+|Y^{(m+p-1)}(s)|\big)ds\\
\ns\ds \hspace{2.4cm}+\int_t^T\big(|\d_\th Y^{(m-1,p)}(s)|+|\d_\th\wt Y^{(m-1,p)}(s)| \big)ds;K\)\Psi_\l(K)\].\ea\ee
On the other hand, in view of \eqref{eq:2.27}, \eqref{eq:2.29}, \eqref{eq:2.33} and \eqref{eq:2.34}, the same computation as above yields that for almost every $t\in\T$ and each $r\in [t,T]$,
\bel{eq:2.36}\ba{ll}
\ds\Phi_\l\(|\d_\th\eta^{(m,p)}(t,r)|+|\d_\th\wt\eta^{(m,p)}(t,r)|;1\)\\
\ \ \les \ds
K\dbE_r\Big[\Phi_\l\Big(\esssup_{s\in [t,T]}\big(|\eta^{(m)}(t,s)|+ |\eta^{(m+p)}(t,s)|\big)+\int_0^T\big(|Y^{(m-1)}(s)|+|Y^{(m+p-1)}(s)|\big)ds \vspace{0.1cm}\\
\ds \hspace{2.4cm}+\int_0^T\big( |\delta_{\theta}Y^{(m-1,p)}(s)|+|\delta_{\theta}\widetilde Y^{(m-1,p)}(s)|\big) ds; K\)\Psi_\l(K)\].
\end{array}
\end{equation}

Observing the similarity of \eqref{eq:2.35} and \eqref{eq:2.17}, we can induce with respect to $m$ and use a similar argument as that obtaining \eqref{eq:2.22} to derive that there exist a positive constant $k_0>0$ depending only on $(\b,\d,n,T)$ such that for each $q>1$ and $\theta\in (0,1)$,
\bel{eq:2.37}\ba{ll}
\ds\dbE\[\Phi_\l\(\esssup_{t\in \T}\big[e^{k_0t}\big(|\d_\th Y^{(m,p)}(t)|+|\d_\th\wt Y^{(m,p)}(t)|\big)\big]; q\)\]\\
\ns\ds\les\bar K(q)\Big\{\dbE\[\Phi_\l\(\esssup_{t\in\T}\big[e^{k_0 t}\big(|\d_\th Y^{(1,p)}(t)|+|\d_\th\wt Y^{(1,p)}(t)|\big)\big];q\)\]\Big\}^{1\over 2^m},\ea\ee
where, in view of \eqref{eq:2.19}, \rf{eq:new} and assumptions on $\psi(\cd)$ and $\a(\cd,,\cd)$ together with H\"{o}lder's inequality,\vspace{0.1cm}
$$\ba{ll}
\ss\ds M(q):=\h K_q\sup_{m,p\ges 1}\dbE\[
\Psi_\l(\h K_q)\Phi_\l\Big(\esssup_{s\in\T}\big( |Y^{(m)}(s)|+|Y^{(m+p)}(s)|\big)\\
\ds \qq\qq\qq\qq\qq\qq +\int_0^T\big(|Y^{(m-1)}(s)|+|Y^{(m+p-1)}(s)|\big)ds;\h K_q\Big)\Big]<+\infty\ea$$
with $\h K_q$ being a variable constant further depending on $k_0$ and $q$.
In view of \eqref{eq:2.19}, it follows from \eqref{eq:2.37} that for each $q>1$ and $\theta\in (0,1)$,
\bel{eq:2.38}
\limsup_{m\to\infty}\sup_{p\ges 1}\dbE\[\Phi_\l\(\esssup_{t\in \T}\big(|\d_\th Y^{(m,p)}(t)|+|\d_\th\wt Y^{(m,p)}(t)|\big); q\)\]\les M(q).\ee
Thus, for each $\theta\in (0,1)$,
$$\limsup_{m\to\infty}\sup_{p\ges 1}\dbE\[\esssup_{t\in \T}|Y^{(m+p)}(t)-\theta Y^{(m)}(t)|^\l\]\les {(1-\theta)^\l \over 2}M(2),$$
and then, in view of \eqref{eq:2.19} again,
$$\limsup_{m\to\infty}\sup_{p\ges 1}\dbE\[\esssup_{t\in \T}|Y^{(m+p)}(t)- Y^{(m)}(t)|^\l\]\les(1-\th)^\l\Big\{{M(2)\over2}+\sup_{m\ges 1}\dbE\[\esssup_{t\in \T}|Y^{(m)}(t)|\]\Big\}<+\infty.$$
Sending $\th$ to $1$ in the last inequality, we get that
$$\ba{ll}
\ss\ds \limsup_{m\rightarrow\infty}\sup_{p\ges 1}\dbE\left[\esssup_{t\in \T}\left|Y^{(m+p)}(t)- Y^{(m)}(t)\right|^\l\right]=0.
\ea
$$
Using again \eqref{eq:2.19} and H\"{o}lder's inequality, we have for each $q>1$,
$$\ba{ll}
\ds \lim_{m\to\infty}\sup_{p\ges 1}\dbE\[\esssup_{t\in \T}|Y^{(m+p)}(t)-Y^{(m)}(t)|^q\]\\
\ns\ds\les\lim_{m\to\infty} \sup_{p\ges 1}\Big\{\dbE\[\esssup_{t\in \T}|Y^{(m+p)}(t)-Y^{(m)}(t)|^{2q-\l}\]\Big\}^{1\over2}\Big\{\dbE\[\esssup_{t\in \T}|Y^{(m+p)}(t)-Y^{(m)}(t)|^\l\]\Big\}^{1\over2}\\
\ns\ds\les\sup_{m\ges1}\Big\{\dbE\[\esssup_{t\in \T}\Phi_\l\(|Y^{(m)}(t)|; K_q\)\]\Big\}^{1\over2}\lim_{m\to\infty}\sup_{p\ges1}\Big\{\dbE\[\esssup_{t\in \T}|Y^{(m+p)}(t)-Y^{(m)}(t)|^\l\]\Big\}^{1\over2}=0,\ea$$
and then
$$\ba{ll}
\ds\lim_{m\to\infty}\sup_{p\ges 1}\dbE\[\esssup_{t\in \T}\Phi_\l\(|Y^{(m+p)}(t)-Y^{(m)}(t)|;q\)-1\]\\
\ns\ds\les \lim_{m\to\infty}\sup_{p\ges1}\dbE\[\esssup_{t\in \T}\{\Phi_\l\(|Y^{(m+p)}(t)-Y^{(m)}(t)|;q\)q|Y^{(m+p)}(t)
-Y^{(m)}(t)|^\l\}\]\\
\ns\ds\les \sup_{m\ges 1}\dbE\Big\{\[\esssup_{t\in \T}\Phi_\l\(|Y^{(m)}(t)|; 2q\)\]\Big\}^{1\over2}\lim_{m\to\infty}\sup_{p\ges 1}\Big\{\dbE\[q^2\esssup_{t\in \T}|Y^{(m+p)}(t)-Y^{(m)}(t)|^{2\l}\]\Big\}^{1\over2}=0.
\ea\vspace{0.2cm}
$$
Consequently, there is an adapted process $Y(\cdot)\in \bigcap_{p\ges 1}\cE_\dbF^{p,\lambda}(\O;L^\infty(0,T;\dbR^n))$ such that for each $q>1$,
\begin{equation}\label{eq:2.39}
\lim_{m\to\infty} \dbE\[\esssup_{t\in \T}|Y^{(m)}(t)-Y(t)|^q\]=0
\end{equation}
and
\begin{equation}\label{eq:2.40}
\lim_{m\to\infty} \dbE\[\esssup_{t\in \T}\Phi_\l\(|Y^{(m)}(t)-Y(t)|; q\)\]=1.\vspace{0.2cm}
\end{equation}

Finally, coming back to \eqref{eq:2.36}, in view of \eqref{eq:2.19}, \eqref{eq:2.25} and \eqref{eq:2.38}, by virtue of Doob's martingale inequality and H\"{o}lder's inequality we can deduce that for each $q>1$ and $\theta\in (0,1)$,
$$\limsup_{m\to\infty}\sup_{p\ges1}\esssup_{t\in\T}\dbE\[\Phi_\l\(\sup_{r\in [t,T]}\big(|\d_\th\eta^{(m,p)}(t,r)|+|\d_\th\tilde \eta^{(m,p)}(t,r)|\big);q\)\]\les\h M(q),$$
where $\h M(q)$ is a positive constant depending on $q$ and being independent of $\theta$. Thus, a similar proof to \eqref{eq:2.39} and (\ref{eq:2.40}) yields that for almost every $t\in\T$, there is an adapted process $\eta(t,\cd)\in\bigcap_{p\ges 1}\cE_\dbF^{p,\lambda}(\O;C([t,T];\dbR^n)) $ such that for each $q>1$,
\begin{equation}\label{eq:2.41}
\lim_{m\to\infty} \esssup_{t\in\T} \dbE\(\sup_{s\in [t,T]}|\eta^{(m)}(t,s)-\eta(t,s)|^q\)=0,
\end{equation}
and
\begin{equation}\label{eq:2.41--1}
\lim_{m\to\infty} \esssup_{t\in\T}
\dbE\[\Phi_\l\(\sup_{s\in [t,T]}|\eta^{(m)}(t,s)-\eta(t,s)| ; q\)\]
=1.\vspace{0.2cm}
\end{equation}
Furthermore, it follows from It\^{o}'s formula that for almost every $t\in \T$ and for each $m,p\ges 1$,
\begin{equation}\label{eq:2.42}
\dbE\[\int_t^T |Z^{(m+p)}(t,s)-Z^{(m)}(t,s)|^2ds\]\les 2\dbE\[\sup_{s\in [t,T]}|\eta^{(m+p)}(t,s)-\eta^{(m)}(t,s)|\cd\D^m(t)\],
\end{equation}
with
$$\D^m(t):=\int_t^T\sum_{i=1}^n|g^i(t,s,Y^{(m+p-1)}(s), Z^{(m+p);i}(t,s))-g^i(t,s,Y^{(m-1)}(s),Z^{(m);i}(t,s))|ds.$$
And, by virtue of  the first inequality in (\ref{Add-4.1}), the assumption on $\a(\cd\,,\cd)$, \eqref{eq:2.19} and \eqref{eq:2.25} we get that
\begin{equation}\label{eq:2.43}
\sup_{m,p\ges 1}\esssup_{t\in\T}\dbE\[|\D^m(t)|^2\]<+\infty.
\end{equation}
Then, applying H\"{o}lder's inequality to \eqref{eq:2.42} and using \eqref{eq:2.41} and \eqref{eq:2.43} leads to that
$$\lim_{m\to\infty}\sup_{p\ges 1}\esssup_{t\in\T}\dbE\[\int_t^T |Z^{(m+p)}(t,s)-Z^{(m)}(t,s)|^2ds\]=0,$$
from which together with \eqref{eq:2.25} it follows that
$$\lim_{m\to\infty}\sup_{p\ges 1}\esssup_{t\in\T} \dbE\[\(\int_t^T |Z^{(m+p)}(t,s)-Z^{(m)}(t,s)|^2 ds\)^{q\over 2}\]=0,\q\forall q>1.$$
Consequently, there exists a process $Z(\cd\,,\cd)\in\bigcap_{p\ges1}  L^\infty(0,T;L^p_\dbF(\O;L^2(\cd\,,T;\dbR^n)))$ such that
\bel{eq:2.44}
\lim_{m\to\infty}\esssup_{t\in\T} \dbE\[\(\int_t^T |Z^{(/
m)}(t,s)-Z(t,s)|^2 ds\)^{q\over 2}\]=0,\qq\forall q>1.\ee
Thus, in view of \eqref{eq:2.39}, \eqref{eq:2.40} and \eqref{eq:2.44}, by sending $m$ to infinity in \eqref{eq:2.16} we deduce that $(Y(\cdot),Z(\cdot,\cdot))$ is a desired solution of BSVIE \eqref{eq:2.15}.
For the above $\eta(t,\cd)$ in (\ref{eq:2.41--1}), it is easy to see that
$$\ba{ll}
\ss\ds \eta (t,r)=\psi(t)+\int_r^Tg(t,s,Y (s),Z(t,s))ds-\int_r^T Z(t,s)dW(s),\ \ r\in[t,T],\vspace{0.1cm}
\ea
$$
from which we see that \eqref{eq:4.21} holds.\vspace{0.2cm}

Next we look at the uniqueness. Suppose $(\wt Y(\cdot),\wt Z(\cdot,\cdot)) $ is another pair of solution
such that $\wt \eta(t,\cd)\in \cap_{p\ges 1}\cE_\dbF^{p,\lambda}(\O;C([t,T];\dbR^n)) $ for almost every $t\in\T$, where
\bel{Relation-two-solutions}\ba{ll}
\ss\ds \wt\eta(t,r):=\wt Y(t)-\int_t^r g(t,s,\wt Y(s),\wt Z(t,s))ds+\int_t^r \wt Z(t,s)dW(s),\ \ r\in[t,T].
\ea
\ee

For each fixed $\theta\in (0,1)$ and $(t,s)\in \D\T$, define\vspace{0.1cm}
$$\ba{ll}
\ds\delta_{\theta}\eta (t,s):=\frac{\eta (t,s)-\theta \wt\eta (t,s)}{1-\theta},\ \ \ \delta_{\theta}Y (t):=\frac{Y (t)-\theta \wt Y (t)}{1-\theta},\ \ \ \delta_{\theta}\tilde\eta (t,s):=\frac{\wt\eta(t,s)-\theta \eta (t,s)}{1-\theta},\\
\ds \delta_{\theta}\tilde Y (t):=\frac{\wt Y(t)-\theta Y (t)}{1-\theta},\ \ \ \delta_{\theta}Z (t,s):=\frac{Z (t,s)-\theta \wt Z (t,s)}{1-\theta},\ \ \
\delta_{\theta}\wt Z (t,s):=\frac{\wt Z (t,s)-\theta  Z (t,s)}{1-\theta}.
\ea
$$
Following the similar arguments from (\ref{eq:2.26}) to (\ref{eq:2.38}), for each $q>1$ and $\theta\in (0,1)$, we have
\begin{equation}
 \dbE\[\Phi_\l\(\esssup_{t\in \T}\big(|\delta_{\theta} Y (t)|+|\delta_{\theta} \widetilde Y (t)|\big); q\)\]\les\wt M(q),
\end{equation}
where $\wt M(q)$ is a positive constant depending on $q$ and being independent of $\theta$. Thus, for each $\theta\in (0,1)$, it follows that
$$
\dbE\[\esssup_{t\in \T}|Y (t)-\theta \wt Y (t)|^\l\]\les \frac{(1-\theta)^\l}{2}\wt M(2),
$$
and then,
$$\ba{ll}
\ss\ds  \dbE\[\esssup_{t\in \T}|Y (t)- \wt Y (t)|^\l\] \les (1-\theta)^\l \l\{ \frac{\wt M(2)}{2} +\dbE\[\esssup_{t\in \T}|\wt Y (t)|\]\}<+\infty.\vspace{0.2cm}
\ea
$$
Sending $\theta$ to $1$, we see that $Y(\cdot)=\wt Y(\cdot).$ Finally, the conclusions of $Z(\cdot,\cdot)=\wt Z(\cdot,\cdot)$, $\eta(\cdot,\cdot)=\wt \eta(\cdot,\cdot)$ follows from Proposition \ref{pro:2.1}.
The proof is complete.  \endpf \vspace{0.2cm}

\br{} \rm
Proposition \ref{pro:2.1} and Theorem \ref{th:2.1} are obtained under the condition that $g^i(t, s, y, z)$ varies with $(t, s, y)$ and the $i$th component $z^i$ of $z\in \dbR^n$ for each $i=1,\cdots,n$, see assumption {\rm (H4.1)} for details. When $g^i$ also depends on the $z^j$ for some $j\neq i$, the complexity of the problem increases significantly. We hope to address it in future work.
\er
\vspace{0.2cm}

\section{Super Quadratic BSVIEs}

In the previous sections, we discuss the case when the generator $g(t,s,y,z)$ of the BSVIEs is at most of quadratic growth in $z$. In this part, we begin to look at the last case, i.e., the case of super-quadratic growth of the map $z\mapsto g(t,s,y,z)$.

For later convenience, let us first recall some existing conclusions of the BSDEs case.
Let us consider the following BSDE
\bel{Quad-BSDE}\ba{ll}
\ss\ds Y(t)=\xi+\int_t^T f(Z(s))ds-\int_t^T Z(s)dW(s), \ \ t\in[0,T],\ea\ee
where the free term $\xi$ and the generator $f$ satisfy the following assumption:

\ms

{\bf (H5.1)} Random variable $\xi$ is bounded $\cF_T$-measurable and $\dbR$-valued, and $f:\dbR\rightarrow\dbR^+$ is convex with
$$f(0)=0\ \ {\rm and}\ \ \limsup_{|z|\rightarrow +\infty}\frac{f(z)}{|z|^2}=\infty.$$

Since the growth order of $z$ in the generator is larger than two, it becomes harder to discuss the integrability of solution $(Y(\cdot),Z(\cdot))$. Hence, in this setting, a pair of $\dbF$-adapted processes $(Y(\cdot),Z(\cdot))$ taking value in $\dbR\times\dbR$ is called a bounded solution by which we mean that they satisfy BSDE (\ref{Quad-BSDE}) for each $t\in[0,T]$ such that $Y(\cd)$ is bounded and
$$\dbE\left[\int_0^T f(Z(t))dt\right]<+\infty.$$
The following conclusion comes from \cite[Theorem 3.1]{Delbaen-Hu-Bao 2011}.

\bp{pro:5.1} \sl Suppose {\rm (H5.1)} holds. Then there exists a bounded $\xi$ such that BSDE \eqref{Quad-BSDE} has no bounded solution.
\ep

At this moment, one may ask: \it what is the sufficient condition of existence of the bounded solution? \rm According to the above conclusion, one needs to impose proper conditions on $\xi$. To this end, given convex $f$ satisfying  {\rm (H5.1)} , let us define
$$\bar f(x):=\sup_{z\in\dbR}\big(z x-f(z)\big),\ \  \forall x\in\dbR.$$
With proper function $q(\cd)$, let us define
$$\ba{ll}
\ss\ds U_t(\xi):=\essinf\Big\{\dbE^{\dbQ}_t\Big[
\xi+\int_t^T \bar f(q(u))du\Big]\Big| \dbQ\sim\dbP\Big\},\ \ t\in[0,T],
\ea
$$
where $\dbQ$ is another probability measure
$$
\ba{ll}
\ss\ds   \frac{d\dbQ}{d\dbP}
=\exp\left[\int_0^T q(s)dW(s)-\frac 1 2 \int_0^T
|q(s)|^2ds\right].
\ea
$$
A bounded random variable $\xi$ is said to be {\it minimal} if
$$\eta\les\xi,\q\dbP(\eta<\xi)>0\q\Ra\q U_0(\eta)<U_0(\xi).$$
The following result comes from \cite[Theorem 3.2]{Delbaen-Hu-Bao 2011}.

\bp{pro:5.2} \sl Suppose $\xi$ is a minimal random variable and {\rm (H5.1)} holds. Then BSDE (\ref{Quad-BSDE}) admits a bounded solution.
\ep

Besides the existence, the uniqueness is another tough issue as is shown in \cite[Theorem 3.3]{Delbaen-Hu-Bao 2011}.

\bp{pro:5.3} \sl Suppose (H5.1) holds and (\ref{Quad-BSDE}) admits a bounded solution. Then for each $y<Y(0)$, BSDE \eqref{Quad-BSDE} admits a bounded solution $(Y'(\cd),Z'(\cd))$ satisfying $Y'(0)=y$.
\ep

Now let us consider the following BSVIE:
\bel{Super-BSVIE-1}\ba{ll}
\ss\ds Y(t)=\psi(t)+\int_t^T g(t,Z(t,s))ds-\int_t^T Z(t,s)dW(s),\ \ \ t\in \T.
\ea
\ee
We discuss its bounded solution $(Y(\cd),Z(\cd\,,\cd))$ taking value in $\dbR\times\dbR$ by which we mean that $(Y(\cd),Z(\cd,\cd))$ is a solution of \eqref{Super-BSVIE-1} such that $Y(\cd)$ is bounded and for each $t\in \T$,
$$
\dbE\left[\int_t^T g(t,Z(t,s))ds\right]<+\infty.
$$
\ms
Similar as {\rm (H5.1)}, for the coefficients in (\ref{Super-BSVIE-1}) we use the following assumptions.

{\bf (H5.2)} The free term $\psi(\cd)$ is a bounded, $\cF_T$-measurable and $\dbR$-valued process, and the generator $g:\T\times\dbR\rightarrow\dbR^+$ is convex  in the second variable with, for each $t\in \T$,
$$\ba{ll}
\ss\ds g(t,0)=0,\ \ {\rm and}\ \ \limsup_{|z|\rightarrow\infty}\frac{g(t,z)}{|z|^2}=\infty.
\ea$$

Similar as Proposition \ref{pro:5.1}, we have
\begin{theorem} \sl Suppose {\rm (H5.2)} holds. Then there exists a bounded $\cF_T$-measurable process $\psi(\cd)$ such that BSVIE \eqref{Super-BSVIE-1} has no bounded solution.
\end{theorem}
\it Proof. \rm
Let us take $\psi(\cd)\equiv \bar\xi$, where the bounded random variable $\bar\xi$ is determined by $g(0,\cd)$ such that the following BSDE
$$\ba{ll}
\ss\ds  \bar Y(t)=\bar\xi+\int_t^T g(0,\bar Z(s))ds-\int_t^T \bar Z(s)dW(s), \ \ t\in[0,T],
\ea
$$
has no bounded solution according to Proposition \ref{pro:5.1}. We prove the conclusion by contradiction. Suppose \eqref{Super-BSVIE-1} admits a bounded solution $(Y(\cd),Z(\cd,\cd))$. For any $t\in[0,T]$, define
$$\ba{ll}
\ss\ds \eta(t,r):=Y(t)-\int_t^r g(t,Z(t,s))ds+\int_t^r Z(t,s)dW(s), \ \ r\in[t,T].
\ea
$$
It is easy to see that for each $t\in\T$, $(\eta(t,\cd),Z(t,\cd))$ satisfies the following BSDE:
$$\ba{ll}
\ss\ds \eta(t,r) =\bar\xi+\int_r^T g(t,Z(t,s))ds-\int_r^T Z(t,s)dW(s), \ \ r\in[t,T].
\ea
$$
Let $t=0$, we see that $(\eta(0,\cd),Z(0,\cd))$ solves
$$\ba{ll}
\ss\ds \eta(0,r) =\bar\xi+\int_r^T g(0,Z(0,s))ds-\int_r^T Z(0,s)dW(s), \ \ r\in[0,T].
\ea
$$
Obviously, it is contradicted with the choice of $\bar\xi$.\vspace{0.2cm}
\endpf

Similar as the above, we call a bounded measurable process $\psi(\cd):\T\times\Omega\rightarrow \dbR$ minimal if for any $t\in[0,T]$, $\psi(t)$ is a minimal random variable.

\bt{} \sl Suppose $\psi(\cd)$ is minimal. Then BSVIE \eqref{Super-BSVIE-1} admits a bounded solution under {\rm (H5.2)}. In addition, the bounded solution of \eqref{Super-BSVIE-1}  is not unique.
\et
\it Proof. \rm
Given parameter $t\in[0,T]$, we consider the following BSDE:
$$\ba{ll}
\ss\ds \eta(t,r)=\psi(t)+\int_r^T g(t,\z(t,s))ds-\int_r^T \z(t,s)dW(s),\ \ r\in[0,T].
\ea
$$
Since $\psi(t)$ is minimal and $g$ satisfies {\rm (H5.2)}, according to Proposition \ref{pro:5.2}, the above BSDE admits a bounded solution $(\eta(t,\cd),\z(t,\cd))$ for almost every $t\in\T$. For each $(t,s)\in \D\T$, by defining $Y(t):=\eta(t,t)$ and $Z(t,s):=\z(t,s)$, we obtain a bounded solution $(Y(\cd),Z(\cd,\cd))$ to BSVIE \eqref{Super-BSVIE-1}. Finally, the non-uniqueness is implied by Proposition \ref{pro:5.3}. \endpf

%

%
%
%
%
%
%
%
%
%

\end{document}